\documentclass[10pt,lettersize,leqno,fullpage]{amsart}
\usepackage[usenames,dvipsnames]{color}
\usepackage{amsmath, amssymb, amscd, amsthm, amsfonts, amsbsy, appendix}
\usepackage{bbm, stmaryrd, mathrsfs, bm, linearb, upgreek, skak, textcomp, mathdots, mathtools}
\usepackage[all, cmtip]{xy} 
\usepackage{graphicx, pgf}
\definecolor{AliceBlue}{rgb}{0.94,0.97,1.00}
\definecolor{AntiqueWhite1}{rgb}{1.00,0.94,0.86}
\definecolor{AntiqueWhite2}{rgb}{0.93,0.87,0.80}
\definecolor{AntiqueWhite3}{rgb}{0.80,0.75,0.69}
\definecolor{AntiqueWhite4}{rgb}{0.55,0.51,0.47}
\definecolor{AntiqueWhite}{rgb}{0.98,0.92,0.84}
\definecolor{BlanchedAlmond}{rgb}{1.00,0.92,0.80}
\definecolor{BlueViolet}{rgb}{0.54,0.17,0.89}
\definecolor{CadetBlue1}{rgb}{0.60,0.96,1.00}
\definecolor{CadetBlue2}{rgb}{0.56,0.90,0.93}
\definecolor{CadetBlue3}{rgb}{0.48,0.77,0.80}
\definecolor{CadetBlue4}{rgb}{0.33,0.53,0.55}
\definecolor{CadetBlue}{rgb}{0.37,0.62,0.63}
\definecolor{CornflowerBlue}{rgb}{0.39,0.58,0.93}
\definecolor{DarkBlue}{rgb}{0.00,0.00,0.55}
\definecolor{DarkCyan}{rgb}{0.00,0.55,0.55}
\definecolor{DarkGoldenrod1}{rgb}{1.00,0.73,0.06}
\definecolor{DarkGoldenrod2}{rgb}{0.93,0.68,0.05}
\definecolor{DarkGoldenrod3}{rgb}{0.80,0.58,0.05}
\definecolor{DarkGoldenrod4}{rgb}{0.55,0.40,0.03}
\definecolor{DarkGoldenrod}{rgb}{0.72,0.53,0.04}
\definecolor{DarkGray}{rgb}{0.66,0.66,0.66}
\definecolor{DarkGreen}{rgb}{0.00,0.39,0.00}
\definecolor{DarkGrey}{rgb}{0.66,0.66,0.66}
\definecolor{DarkKhaki}{rgb}{0.74,0.72,0.42}
\definecolor{DarkMagenta}{rgb}{0.55,0.00,0.55}
\definecolor{DarkOliveGreen1}{rgb}{0.79,1.00,0.44}
\definecolor{DarkOliveGreen2}{rgb}{0.74,0.93,0.41}
\definecolor{DarkOliveGreen3}{rgb}{0.64,0.80,0.35}
\definecolor{DarkOliveGreen4}{rgb}{0.43,0.55,0.24}
\definecolor{DarkOliveGreen}{rgb}{0.33,0.42,0.18}
\definecolor{DarkOrange1}{rgb}{1.00,0.50,0.00}
\definecolor{DarkOrange2}{rgb}{0.93,0.46,0.00}
\definecolor{DarkOrange3}{rgb}{0.80,0.40,0.00}
\definecolor{DarkOrange4}{rgb}{0.55,0.27,0.00}
\definecolor{DarkOrange}{rgb}{1.00,0.55,0.00}
\definecolor{DarkOrchid1}{rgb}{0.75,0.24,1.00}
\definecolor{DarkOrchid2}{rgb}{0.70,0.23,0.93}
\definecolor{DarkOrchid3}{rgb}{0.60,0.20,0.80}
\definecolor{DarkOrchid4}{rgb}{0.41,0.13,0.55}
\definecolor{DarkOrchid}{rgb}{0.60,0.20,0.80}
\definecolor{DarkRed}{rgb}{0.55,0.00,0.00}
\definecolor{DarkSalmon}{rgb}{0.91,0.59,0.48}
\definecolor{DarkSeaGreen1}{rgb}{0.76,1.00,0.76}
\definecolor{DarkSeaGreen2}{rgb}{0.71,0.93,0.71}
\definecolor{DarkSeaGreen3}{rgb}{0.61,0.80,0.61}
\definecolor{DarkSeaGreen4}{rgb}{0.41,0.55,0.41}
\definecolor{DarkSeaGreen}{rgb}{0.56,0.74,0.56}
\definecolor{DarkSlateBlue}{rgb}{0.28,0.24,0.55}
\definecolor{DarkSlateGray1}{rgb}{0.59,1.00,1.00}
\definecolor{DarkSlateGray2}{rgb}{0.55,0.93,0.93}
\definecolor{DarkSlateGray3}{rgb}{0.47,0.80,0.80}
\definecolor{DarkSlateGray4}{rgb}{0.32,0.55,0.55}
\definecolor{DarkSlateGray}{rgb}{0.18,0.31,0.31}
\definecolor{DarkSlateGrey}{rgb}{0.18,0.31,0.31}
\definecolor{DarkTurquoise}{rgb}{0.00,0.81,0.82}
\definecolor{DarkViolet}{rgb}{0.58,0.00,0.83}
\definecolor{DeepPink1}{rgb}{1.00,0.08,0.58}
\definecolor{DeepPink2}{rgb}{0.93,0.07,0.54}
\definecolor{DeepPink3}{rgb}{0.80,0.06,0.46}
\definecolor{DeepPink4}{rgb}{0.55,0.04,0.31}
\definecolor{DeepPink}{rgb}{1.00,0.08,0.58}
\definecolor{DeepSkyBlue1}{rgb}{0.00,0.75,1.00}
\definecolor{DeepSkyBlue2}{rgb}{0.00,0.70,0.93}
\definecolor{DeepSkyBlue3}{rgb}{0.00,0.60,0.80}
\definecolor{DeepSkyBlue4}{rgb}{0.00,0.41,0.55}
\definecolor{DeepSkyBlue}{rgb}{0.00,0.75,1.00}
\definecolor{DimGray}{rgb}{0.41,0.41,0.41}
\definecolor{DimGrey}{rgb}{0.41,0.41,0.41}
\definecolor{DodgerBlue1}{rgb}{0.12,0.56,1.00}
\definecolor{DodgerBlue2}{rgb}{0.11,0.53,0.93}
\definecolor{DodgerBlue3}{rgb}{0.09,0.45,0.80}
\definecolor{DodgerBlue4}{rgb}{0.06,0.31,0.55}
\definecolor{DodgerBlue}{rgb}{0.12,0.56,1.00}
\definecolor{FloralWhite}{rgb}{1.00,0.98,0.94}
\definecolor{ForestGreen}{rgb}{0.13,0.55,0.13}
\definecolor{GhostWhite}{rgb}{0.97,0.97,1.00}
\definecolor{GreenYellow}{rgb}{0.68,1.00,0.18}
\definecolor{HotPink1}{rgb}{1.00,0.43,0.71}
\definecolor{HotPink2}{rgb}{0.93,0.42,0.65}
\definecolor{HotPink3}{rgb}{0.80,0.38,0.56}
\definecolor{HotPink4}{rgb}{0.55,0.23,0.38}
\definecolor{HotPink}{rgb}{1.00,0.41,0.71}
\definecolor{IndianRed1}{rgb}{1.00,0.42,0.42}
\definecolor{IndianRed2}{rgb}{0.93,0.39,0.39}
\definecolor{IndianRed3}{rgb}{0.80,0.33,0.33}
\definecolor{IndianRed4}{rgb}{0.55,0.23,0.23}
\definecolor{IndianRed}{rgb}{0.80,0.36,0.36}
\definecolor{LavenderBlush1}{rgb}{1.00,0.94,0.96}
\definecolor{LavenderBlush2}{rgb}{0.93,0.88,0.90}
\definecolor{LavenderBlush3}{rgb}{0.80,0.76,0.77}
\definecolor{LavenderBlush4}{rgb}{0.55,0.51,0.53}
\definecolor{LavenderBlush}{rgb}{1.00,0.94,0.96}
\definecolor{LawnGreen}{rgb}{0.49,0.99,0.00}
\definecolor{LemonChiffon1}{rgb}{1.00,0.98,0.80}
\definecolor{LemonChiffon2}{rgb}{0.93,0.91,0.75}
\definecolor{LemonChiffon3}{rgb}{0.80,0.79,0.65}
\definecolor{LemonChiffon4}{rgb}{0.55,0.54,0.44}
\definecolor{LemonChiffon}{rgb}{1.00,0.98,0.80}
\definecolor{LightBlue1}{rgb}{0.75,0.94,1.00}
\definecolor{LightBlue2}{rgb}{0.70,0.87,0.93}
\definecolor{LightBlue3}{rgb}{0.60,0.75,0.80}
\definecolor{LightBlue4}{rgb}{0.41,0.51,0.55}
\definecolor{LightBlue}{rgb}{0.68,0.85,0.90}
\definecolor{LightCoral}{rgb}{0.94,0.50,0.50}
\definecolor{LightCyan1}{rgb}{0.88,1.00,1.00}
\definecolor{LightCyan2}{rgb}{0.82,0.93,0.93}
\definecolor{LightCyan3}{rgb}{0.71,0.80,0.80}
\definecolor{LightCyan4}{rgb}{0.48,0.55,0.55}
\definecolor{LightCyan}{rgb}{0.88,1.00,1.00}
\definecolor{LightGoldenrod1}{rgb}{1.00,0.93,0.55}
\definecolor{LightGoldenrod2}{rgb}{0.93,0.86,0.51}
\definecolor{LightGoldenrod3}{rgb}{0.80,0.75,0.44}
\definecolor{LightGoldenrod4}{rgb}{0.55,0.51,0.30}
\definecolor{LightGoldenrodYellow}{rgb}{0.98,0.98,0.82}
\definecolor{LightGoldenrod}{rgb}{0.93,0.87,0.51}
\definecolor{LightGray}{rgb}{0.83,0.83,0.83}
\definecolor{LightGreen}{rgb}{0.56,0.93,0.56}
\definecolor{LightGrey}{rgb}{0.83,0.83,0.83}
\definecolor{LightPink1}{rgb}{1.00,0.68,0.73}
\definecolor{LightPink2}{rgb}{0.93,0.64,0.68}
\definecolor{LightPink3}{rgb}{0.80,0.55,0.58}
\definecolor{LightPink4}{rgb}{0.55,0.37,0.40}
\definecolor{LightPink}{rgb}{1.00,0.71,0.76}
\definecolor{LightSalmon1}{rgb}{1.00,0.63,0.48}
\definecolor{LightSalmon2}{rgb}{0.93,0.58,0.45}
\definecolor{LightSalmon3}{rgb}{0.80,0.51,0.38}
\definecolor{LightSalmon4}{rgb}{0.55,0.34,0.26}
\definecolor{LightSalmon}{rgb}{1.00,0.63,0.48}
\definecolor{LightSeaGreen}{rgb}{0.13,0.70,0.67}
\definecolor{LightSkyBlue1}{rgb}{0.69,0.89,1.00}
\definecolor{LightSkyBlue2}{rgb}{0.64,0.83,0.93}
\definecolor{LightSkyBlue3}{rgb}{0.55,0.71,0.80}
\definecolor{LightSkyBlue4}{rgb}{0.38,0.48,0.55}
\definecolor{LightSkyBlue}{rgb}{0.53,0.81,0.98}
\definecolor{LightSlateBlue}{rgb}{0.52,0.44,1.00}
\definecolor{LightSlateGray}{rgb}{0.47,0.53,0.60}
\definecolor{LightSlateGrey}{rgb}{0.47,0.53,0.60}
\definecolor{LightSteelBlue1}{rgb}{0.79,0.88,1.00}
\definecolor{LightSteelBlue2}{rgb}{0.74,0.82,0.93}
\definecolor{LightSteelBlue3}{rgb}{0.64,0.71,0.80}
\definecolor{LightSteelBlue4}{rgb}{0.43,0.48,0.55}
\definecolor{LightSteelBlue}{rgb}{0.69,0.77,0.87}
\definecolor{LightYellow1}{rgb}{1.00,1.00,0.88}
\definecolor{LightYellow2}{rgb}{0.93,0.93,0.82}
\definecolor{LightYellow3}{rgb}{0.80,0.80,0.71}
\definecolor{LightYellow4}{rgb}{0.55,0.55,0.48}
\definecolor{LightYellow}{rgb}{1.00,1.00,0.88}
\definecolor{LimeGreen}{rgb}{0.20,0.80,0.20}
\definecolor{MediumAquamarine}{rgb}{0.40,0.80,0.67}
\definecolor{MediumBlue}{rgb}{0.00,0.00,0.80}
\definecolor{MediumOrchid1}{rgb}{0.88,0.40,1.00}
\definecolor{MediumOrchid2}{rgb}{0.82,0.37,0.93}
\definecolor{MediumOrchid3}{rgb}{0.71,0.32,0.80}
\definecolor{MediumOrchid4}{rgb}{0.48,0.22,0.55}
\definecolor{MediumOrchid}{rgb}{0.73,0.33,0.83}
\definecolor{MediumPurple1}{rgb}{0.67,0.51,1.00}
\definecolor{MediumPurple2}{rgb}{0.62,0.47,0.93}
\definecolor{MediumPurple3}{rgb}{0.54,0.41,0.80}
\definecolor{MediumPurple4}{rgb}{0.36,0.28,0.55}
\definecolor{MediumPurple}{rgb}{0.58,0.44,0.86}
\definecolor{MediumSeaGreen}{rgb}{0.24,0.70,0.44}
\definecolor{MediumSlateBlue}{rgb}{0.48,0.41,0.93}
\definecolor{MediumSpringGreen}{rgb}{0.00,0.98,0.60}
\definecolor{MediumTurquoise}{rgb}{0.28,0.82,0.80}
\definecolor{MediumVioletRed}{rgb}{0.78,0.08,0.52}
\definecolor{MidnightBlue}{rgb}{0.10,0.10,0.44}
\definecolor{MintCream}{rgb}{0.96,1.00,0.98}
\definecolor{MistyRose1}{rgb}{1.00,0.89,0.88}
\definecolor{MistyRose2}{rgb}{0.93,0.84,0.82}
\definecolor{MistyRose3}{rgb}{0.80,0.72,0.71}
\definecolor{MistyRose4}{rgb}{0.55,0.49,0.48}
\definecolor{MistyRose}{rgb}{1.00,0.89,0.88}
\definecolor{NavajoWhite1}{rgb}{1.00,0.87,0.68}
\definecolor{NavajoWhite2}{rgb}{0.93,0.81,0.63}
\definecolor{NavajoWhite3}{rgb}{0.80,0.70,0.55}
\definecolor{NavajoWhite4}{rgb}{0.55,0.47,0.37}
\definecolor{NavajoWhite}{rgb}{1.00,0.87,0.68}
\definecolor{NavyBlue}{rgb}{0.00,0.00,0.50}
\definecolor{OldLace}{rgb}{0.99,0.96,0.90}
\definecolor{OliveDrab1}{rgb}{0.75,1.00,0.24}
\definecolor{OliveDrab2}{rgb}{0.70,0.93,0.23}
\definecolor{OliveDrab3}{rgb}{0.60,0.80,0.20}
\definecolor{OliveDrab4}{rgb}{0.41,0.55,0.13}
\definecolor{OliveDrab}{rgb}{0.42,0.56,0.14}
\definecolor{OrangeRed1}{rgb}{1.00,0.27,0.00}
\definecolor{OrangeRed2}{rgb}{0.93,0.25,0.00}
\definecolor{OrangeRed3}{rgb}{0.80,0.22,0.00}
\definecolor{OrangeRed4}{rgb}{0.55,0.15,0.00}
\definecolor{OrangeRed}{rgb}{1.00,0.27,0.00}
\definecolor{PaleGoldenrod}{rgb}{0.93,0.91,0.67}
\definecolor{PaleGreen1}{rgb}{0.60,1.00,0.60}
\definecolor{PaleGreen2}{rgb}{0.56,0.93,0.56}
\definecolor{PaleGreen3}{rgb}{0.49,0.80,0.49}
\definecolor{PaleGreen4}{rgb}{0.33,0.55,0.33}
\definecolor{PaleGreen}{rgb}{0.60,0.98,0.60}
\definecolor{PaleTurquoise1}{rgb}{0.73,1.00,1.00}
\definecolor{PaleTurquoise2}{rgb}{0.68,0.93,0.93}
\definecolor{PaleTurquoise3}{rgb}{0.59,0.80,0.80}
\definecolor{PaleTurquoise4}{rgb}{0.40,0.55,0.55}
\definecolor{PaleTurquoise}{rgb}{0.69,0.93,0.93}
\definecolor{PaleVioletRed1}{rgb}{1.00,0.51,0.67}
\definecolor{PaleVioletRed2}{rgb}{0.93,0.47,0.62}
\definecolor{PaleVioletRed3}{rgb}{0.80,0.41,0.54}
\definecolor{PaleVioletRed4}{rgb}{0.55,0.28,0.36}
\definecolor{PaleVioletRed}{rgb}{0.86,0.44,0.58}
\definecolor{PapayaWhip}{rgb}{1.00,0.94,0.84}
\definecolor{PeachPuff1}{rgb}{1.00,0.85,0.73}
\definecolor{PeachPuff2}{rgb}{0.93,0.80,0.68}
\definecolor{PeachPuff3}{rgb}{0.80,0.69,0.58}
\definecolor{PeachPuff4}{rgb}{0.55,0.47,0.40}
\definecolor{PeachPuff}{rgb}{1.00,0.85,0.73}
\definecolor{PowderBlue}{rgb}{0.69,0.88,0.90}
\definecolor{RosyBrown1}{rgb}{1.00,0.76,0.76}
\definecolor{RosyBrown2}{rgb}{0.93,0.71,0.71}
\definecolor{RosyBrown3}{rgb}{0.80,0.61,0.61}
\definecolor{RosyBrown4}{rgb}{0.55,0.41,0.41}
\definecolor{RosyBrown}{rgb}{0.74,0.56,0.56}
\definecolor{RoyalBlue1}{rgb}{0.28,0.46,1.00}
\definecolor{RoyalBlue2}{rgb}{0.26,0.43,0.93}
\definecolor{RoyalBlue3}{rgb}{0.23,0.37,0.80}
\definecolor{RoyalBlue4}{rgb}{0.15,0.25,0.55}
\definecolor{RoyalBlue}{rgb}{0.25,0.41,0.88}
\definecolor{SaddleBrown}{rgb}{0.55,0.27,0.07}
\definecolor{SandyBrown}{rgb}{0.96,0.64,0.38}
\definecolor{SeaGreen1}{rgb}{0.33,1.00,0.62}
\definecolor{SeaGreen2}{rgb}{0.31,0.93,0.58}
\definecolor{SeaGreen3}{rgb}{0.26,0.80,0.50}
\definecolor{SeaGreen4}{rgb}{0.18,0.55,0.34}
\definecolor{SeaGreen}{rgb}{0.18,0.55,0.34}
\definecolor{SkyBlue1}{rgb}{0.53,0.81,1.00}
\definecolor{SkyBlue2}{rgb}{0.49,0.75,0.93}
\definecolor{SkyBlue3}{rgb}{0.42,0.65,0.80}
\definecolor{SkyBlue4}{rgb}{0.29,0.44,0.55}
\definecolor{SkyBlue}{rgb}{0.53,0.81,0.92}
\definecolor{SlateBlue1}{rgb}{0.51,0.44,1.00}
\definecolor{SlateBlue2}{rgb}{0.48,0.40,0.93}
\definecolor{SlateBlue3}{rgb}{0.41,0.35,0.80}
\definecolor{SlateBlue4}{rgb}{0.28,0.24,0.55}
\definecolor{SlateBlue}{rgb}{0.42,0.35,0.80}
\definecolor{SlateGray1}{rgb}{0.78,0.89,1.00}
\definecolor{SlateGray2}{rgb}{0.73,0.83,0.93}
\definecolor{SlateGray3}{rgb}{0.62,0.71,0.80}
\definecolor{SlateGray4}{rgb}{0.42,0.48,0.55}
\definecolor{SlateGray}{rgb}{0.44,0.50,0.56}
\definecolor{SlateGrey}{rgb}{0.44,0.50,0.56}
\definecolor{SpringGreen1}{rgb}{0.00,1.00,0.50}
\definecolor{SpringGreen2}{rgb}{0.00,0.93,0.46}
\definecolor{SpringGreen3}{rgb}{0.00,0.80,0.40}
\definecolor{SpringGreen4}{rgb}{0.00,0.55,0.27}
\definecolor{SpringGreen}{rgb}{0.00,1.00,0.50}
\definecolor{SteelBlue1}{rgb}{0.39,0.72,1.00}
\definecolor{SteelBlue2}{rgb}{0.36,0.67,0.93}
\definecolor{SteelBlue3}{rgb}{0.31,0.58,0.80}
\definecolor{SteelBlue4}{rgb}{0.21,0.39,0.55}
\definecolor{SteelBlue}{rgb}{0.27,0.51,0.71}
\definecolor{VioletRed1}{rgb}{1.00,0.24,0.59}
\definecolor{VioletRed2}{rgb}{0.93,0.23,0.55}
\definecolor{VioletRed3}{rgb}{0.80,0.20,0.47}
\definecolor{VioletRed4}{rgb}{0.55,0.13,0.32}
\definecolor{VioletRed}{rgb}{0.82,0.13,0.56}
\definecolor{WhiteSmoke}{rgb}{0.96,0.96,0.96}
\definecolor{YellowGreen}{rgb}{0.60,0.80,0.20}
\definecolor{aliceblue}{rgb}{0.94,0.97,1.00}
\definecolor{antiquewhite}{rgb}{0.98,0.92,0.84}
\definecolor{aquamarine1}{rgb}{0.50,1.00,0.83}
\definecolor{aquamarine2}{rgb}{0.46,0.93,0.78}
\definecolor{aquamarine3}{rgb}{0.40,0.80,0.67}
\definecolor{aquamarine4}{rgb}{0.27,0.55,0.45}
\definecolor{aquamarine}{rgb}{0.50,1.00,0.83}
\definecolor{azure1}{rgb}{0.94,1.00,1.00}
\definecolor{azure2}{rgb}{0.88,0.93,0.93}
\definecolor{azure3}{rgb}{0.76,0.80,0.80}
\definecolor{azure4}{rgb}{0.51,0.55,0.55}
\definecolor{azure}{rgb}{0.94,1.00,1.00}
\definecolor{beige}{rgb}{0.96,0.96,0.86}
\definecolor{bisque1}{rgb}{1.00,0.89,0.77}
\definecolor{bisque2}{rgb}{0.93,0.84,0.72}
\definecolor{bisque3}{rgb}{0.80,0.72,0.62}
\definecolor{bisque4}{rgb}{0.55,0.49,0.42}
\definecolor{bisque}{rgb}{1.00,0.89,0.77}
\definecolor{black}{rgb}{0.00,0.00,0.00}
\definecolor{blanchedalmond}{rgb}{1.00,0.92,0.80}
\definecolor{blue1}{rgb}{0.00,0.00,1.00}
\definecolor{blue2}{rgb}{0.00,0.00,0.93}
\definecolor{blue3}{rgb}{0.00,0.00,0.80}
\definecolor{blue4}{rgb}{0.00,0.00,0.55}
\definecolor{blueviolet}{rgb}{0.54,0.17,0.89}
\definecolor{blue}{rgb}{0.00,0.00,1.00}
\definecolor{brown1}{rgb}{1.00,0.25,0.25}
\definecolor{brown2}{rgb}{0.93,0.23,0.23}
\definecolor{brown3}{rgb}{0.80,0.20,0.20}
\definecolor{brown4}{rgb}{0.55,0.14,0.14}
\definecolor{brown}{rgb}{0.65,0.16,0.16}
\definecolor{burlywood1}{rgb}{1.00,0.83,0.61}
\definecolor{burlywood2}{rgb}{0.93,0.77,0.57}
\definecolor{burlywood3}{rgb}{0.80,0.67,0.49}
\definecolor{burlywood4}{rgb}{0.55,0.45,0.33}
\definecolor{burlywood}{rgb}{0.87,0.72,0.53}
\definecolor{cadetblue}{rgb}{0.37,0.62,0.63}
\definecolor{chartreuse1}{rgb}{0.50,1.00,0.00}
\definecolor{chartreuse2}{rgb}{0.46,0.93,0.00}
\definecolor{chartreuse3}{rgb}{0.40,0.80,0.00}
\definecolor{chartreuse4}{rgb}{0.27,0.55,0.00}
\definecolor{chartreuse}{rgb}{0.50,1.00,0.00}
\definecolor{chocolate1}{rgb}{1.00,0.50,0.14}
\definecolor{chocolate2}{rgb}{0.93,0.46,0.13}
\definecolor{chocolate3}{rgb}{0.80,0.40,0.11}
\definecolor{chocolate4}{rgb}{0.55,0.27,0.07}
\definecolor{chocolate}{rgb}{0.82,0.41,0.12}
\definecolor{coral1}{rgb}{1.00,0.45,0.34}
\definecolor{coral2}{rgb}{0.93,0.42,0.31}
\definecolor{coral3}{rgb}{0.80,0.36,0.27}
\definecolor{coral4}{rgb}{0.55,0.24,0.18}
\definecolor{coral}{rgb}{1.00,0.50,0.31}
\definecolor{cornflowerblue}{rgb}{0.39,0.58,0.93}
\definecolor{cornsilk1}{rgb}{1.00,0.97,0.86}
\definecolor{cornsilk2}{rgb}{0.93,0.91,0.80}
\definecolor{cornsilk3}{rgb}{0.80,0.78,0.69}
\definecolor{cornsilk4}{rgb}{0.55,0.53,0.47}
\definecolor{cornsilk}{rgb}{1.00,0.97,0.86}
\definecolor{cyan1}{rgb}{0.00,1.00,1.00}
\definecolor{cyan2}{rgb}{0.00,0.93,0.93}
\definecolor{cyan3}{rgb}{0.00,0.80,0.80}
\definecolor{cyan4}{rgb}{0.00,0.55,0.55}
\definecolor{cyan}{rgb}{0.00,1.00,1.00}
\definecolor{darkblue}{rgb}{0.00,0.00,0.55}
\definecolor{darkcyan}{rgb}{0.00,0.55,0.55}
\definecolor{darkgoldenrod}{rgb}{0.72,0.53,0.04}
\definecolor{darkgray}{rgb}{0.66,0.66,0.66}
\definecolor{darkgreen}{rgb}{0.00,0.39,0.00}
\definecolor{darkgrey}{rgb}{0.66,0.66,0.66}
\definecolor{darkkhaki}{rgb}{0.74,0.72,0.42}
\definecolor{darkmagenta}{rgb}{0.55,0.00,0.55}
\definecolor{darkolive}{rgb}{0.33,0.42,0.18}
\definecolor{darkorange}{rgb}{1.00,0.55,0.00}
\definecolor{darkorchid}{rgb}{0.60,0.20,0.80}
\definecolor{darkred}{rgb}{0.55,0.00,0.00}
\definecolor{darksalmon}{rgb}{0.91,0.59,0.48}
\definecolor{darksea}{rgb}{0.56,0.74,0.56}
\definecolor{darkslate}{rgb}{0.18,0.31,0.31}
\definecolor{darkslate}{rgb}{0.18,0.31,0.31}
\definecolor{darkslate}{rgb}{0.28,0.24,0.55}
\definecolor{darkturquoise}{rgb}{0.00,0.81,0.82}
\definecolor{darkviolet}{rgb}{0.58,0.00,0.83}
\definecolor{deeppink}{rgb}{1.00,0.08,0.58}
\definecolor{deepsky}{rgb}{0.00,0.75,1.00}
\definecolor{dimgray}{rgb}{0.41,0.41,0.41}
\definecolor{dimgrey}{rgb}{0.41,0.41,0.41}
\definecolor{dodgerblue}{rgb}{0.12,0.56,1.00}
\definecolor{firebrick1}{rgb}{1.00,0.19,0.19}
\definecolor{firebrick2}{rgb}{0.93,0.17,0.17}
\definecolor{firebrick3}{rgb}{0.80,0.15,0.15}
\definecolor{firebrick4}{rgb}{0.55,0.10,0.10}
\definecolor{firebrick}{rgb}{0.70,0.13,0.13}
\definecolor{floralwhite}{rgb}{1.00,0.98,0.94}
\definecolor{forestgreen}{rgb}{0.13,0.55,0.13}
\definecolor{gainsboro}{rgb}{0.86,0.86,0.86}
\definecolor{ghostwhite}{rgb}{0.97,0.97,1.00}
\definecolor{gold1}{rgb}{1.00,0.84,0.00}
\definecolor{gold2}{rgb}{0.93,0.79,0.00}
\definecolor{gold3}{rgb}{0.80,0.68,0.00}
\definecolor{gold4}{rgb}{0.55,0.46,0.00}
\definecolor{goldenrod1}{rgb}{1.00,0.76,0.15}
\definecolor{goldenrod2}{rgb}{0.93,0.71,0.13}
\definecolor{goldenrod3}{rgb}{0.80,0.61,0.11}
\definecolor{goldenrod4}{rgb}{0.55,0.41,0.08}
\definecolor{goldenrod}{rgb}{0.85,0.65,0.13}
\definecolor{gold}{rgb}{1.00,0.84,0.00}
\definecolor{gray0}{rgb}{0.00,0.00,0.00}
\definecolor{gray100}{rgb}{1.00,1.00,1.00}
\definecolor{gray10}{rgb}{0.10,0.10,0.10}
\definecolor{gray11}{rgb}{0.11,0.11,0.11}
\definecolor{gray12}{rgb}{0.12,0.12,0.12}
\definecolor{gray13}{rgb}{0.13,0.13,0.13}
\definecolor{gray14}{rgb}{0.14,0.14,0.14}
\definecolor{gray15}{rgb}{0.15,0.15,0.15}
\definecolor{gray16}{rgb}{0.16,0.16,0.16}
\definecolor{gray17}{rgb}{0.17,0.17,0.17}
\definecolor{gray18}{rgb}{0.18,0.18,0.18}
\definecolor{gray19}{rgb}{0.19,0.19,0.19}
\definecolor{gray1}{rgb}{0.01,0.01,0.01}
\definecolor{gray20}{rgb}{0.20,0.20,0.20}
\definecolor{gray21}{rgb}{0.21,0.21,0.21}
\definecolor{gray22}{rgb}{0.22,0.22,0.22}
\definecolor{gray23}{rgb}{0.23,0.23,0.23}
\definecolor{gray24}{rgb}{0.24,0.24,0.24}
\definecolor{gray25}{rgb}{0.25,0.25,0.25}
\definecolor{gray26}{rgb}{0.26,0.26,0.26}
\definecolor{gray27}{rgb}{0.27,0.27,0.27}
\definecolor{gray28}{rgb}{0.28,0.28,0.28}
\definecolor{gray29}{rgb}{0.29,0.29,0.29}
\definecolor{gray2}{rgb}{0.02,0.02,0.02}
\definecolor{gray30}{rgb}{0.30,0.30,0.30}
\definecolor{gray31}{rgb}{0.31,0.31,0.31}
\definecolor{gray32}{rgb}{0.32,0.32,0.32}
\definecolor{gray33}{rgb}{0.33,0.33,0.33}
\definecolor{gray34}{rgb}{0.34,0.34,0.34}
\definecolor{gray35}{rgb}{0.35,0.35,0.35}
\definecolor{gray36}{rgb}{0.36,0.36,0.36}
\definecolor{gray37}{rgb}{0.37,0.37,0.37}
\definecolor{gray38}{rgb}{0.38,0.38,0.38}
\definecolor{gray39}{rgb}{0.39,0.39,0.39}
\definecolor{gray3}{rgb}{0.03,0.03,0.03}
\definecolor{gray40}{rgb}{0.40,0.40,0.40}
\definecolor{gray41}{rgb}{0.41,0.41,0.41}
\definecolor{gray42}{rgb}{0.42,0.42,0.42}
\definecolor{gray43}{rgb}{0.43,0.43,0.43}
\definecolor{gray44}{rgb}{0.44,0.44,0.44}
\definecolor{gray45}{rgb}{0.45,0.45,0.45}
\definecolor{gray46}{rgb}{0.46,0.46,0.46}
\definecolor{gray47}{rgb}{0.47,0.47,0.47}
\definecolor{gray48}{rgb}{0.48,0.48,0.48}
\definecolor{gray49}{rgb}{0.49,0.49,0.49}
\definecolor{gray4}{rgb}{0.04,0.04,0.04}
\definecolor{gray50}{rgb}{0.50,0.50,0.50}
\definecolor{gray51}{rgb}{0.51,0.51,0.51}
\definecolor{gray52}{rgb}{0.52,0.52,0.52}
\definecolor{gray53}{rgb}{0.53,0.53,0.53}
\definecolor{gray54}{rgb}{0.54,0.54,0.54}
\definecolor{gray55}{rgb}{0.55,0.55,0.55}
\definecolor{gray56}{rgb}{0.56,0.56,0.56}
\definecolor{gray57}{rgb}{0.57,0.57,0.57}
\definecolor{gray58}{rgb}{0.58,0.58,0.58}
\definecolor{gray59}{rgb}{0.59,0.59,0.59}
\definecolor{gray5}{rgb}{0.05,0.05,0.05}
\definecolor{gray60}{rgb}{0.60,0.60,0.60}
\definecolor{gray61}{rgb}{0.61,0.61,0.61}
\definecolor{gray62}{rgb}{0.62,0.62,0.62}
\definecolor{gray63}{rgb}{0.63,0.63,0.63}
\definecolor{gray64}{rgb}{0.64,0.64,0.64}
\definecolor{gray65}{rgb}{0.65,0.65,0.65}
\definecolor{gray66}{rgb}{0.66,0.66,0.66}
\definecolor{gray67}{rgb}{0.67,0.67,0.67}
\definecolor{gray68}{rgb}{0.68,0.68,0.68}
\definecolor{gray69}{rgb}{0.69,0.69,0.69}
\definecolor{gray6}{rgb}{0.06,0.06,0.06}
\definecolor{gray70}{rgb}{0.70,0.70,0.70}
\definecolor{gray71}{rgb}{0.71,0.71,0.71}
\definecolor{gray72}{rgb}{0.72,0.72,0.72}
\definecolor{gray73}{rgb}{0.73,0.73,0.73}
\definecolor{gray74}{rgb}{0.74,0.74,0.74}
\definecolor{gray75}{rgb}{0.75,0.75,0.75}
\definecolor{gray76}{rgb}{0.76,0.76,0.76}
\definecolor{gray77}{rgb}{0.77,0.77,0.77}
\definecolor{gray78}{rgb}{0.78,0.78,0.78}
\definecolor{gray79}{rgb}{0.79,0.79,0.79}
\definecolor{gray7}{rgb}{0.07,0.07,0.07}
\definecolor{gray80}{rgb}{0.80,0.80,0.80}
\definecolor{gray81}{rgb}{0.81,0.81,0.81}
\definecolor{gray82}{rgb}{0.82,0.82,0.82}
\definecolor{gray83}{rgb}{0.83,0.83,0.83}
\definecolor{gray84}{rgb}{0.84,0.84,0.84}
\definecolor{gray85}{rgb}{0.85,0.85,0.85}
\definecolor{gray86}{rgb}{0.86,0.86,0.86}
\definecolor{gray87}{rgb}{0.87,0.87,0.87}
\definecolor{gray88}{rgb}{0.88,0.88,0.88}
\definecolor{gray89}{rgb}{0.89,0.89,0.89}
\definecolor{gray8}{rgb}{0.08,0.08,0.08}
\definecolor{gray90}{rgb}{0.90,0.90,0.90}
\definecolor{gray91}{rgb}{0.91,0.91,0.91}
\definecolor{gray92}{rgb}{0.92,0.92,0.92}
\definecolor{gray93}{rgb}{0.93,0.93,0.93}
\definecolor{gray94}{rgb}{0.94,0.94,0.94}
\definecolor{gray95}{rgb}{0.95,0.95,0.95}
\definecolor{gray96}{rgb}{0.96,0.96,0.96}
\definecolor{gray97}{rgb}{0.97,0.97,0.97}
\definecolor{gray98}{rgb}{0.98,0.98,0.98}
\definecolor{gray99}{rgb}{0.99,0.99,0.99}
\definecolor{gray9}{rgb}{0.09,0.09,0.09}
\definecolor{gray}{rgb}{0.75,0.75,0.75}
\definecolor{green1}{rgb}{0.00,1.00,0.00}
\definecolor{green2}{rgb}{0.00,0.93,0.00}
\definecolor{green3}{rgb}{0.00,0.80,0.00}
\definecolor{green4}{rgb}{0.00,0.55,0.00}
\definecolor{greenyellow}{rgb}{0.68,1.00,0.18}
\definecolor{green}{rgb}{0.00,1.00,0.00}
\definecolor{grey0}{rgb}{0.00,0.00,0.00}
\definecolor{grey100}{rgb}{1.00,1.00,1.00}
\definecolor{grey10}{rgb}{0.10,0.10,0.10}
\definecolor{grey11}{rgb}{0.11,0.11,0.11}
\definecolor{grey12}{rgb}{0.12,0.12,0.12}
\definecolor{grey13}{rgb}{0.13,0.13,0.13}
\definecolor{grey14}{rgb}{0.14,0.14,0.14}
\definecolor{grey15}{rgb}{0.15,0.15,0.15}
\definecolor{grey16}{rgb}{0.16,0.16,0.16}
\definecolor{grey17}{rgb}{0.17,0.17,0.17}
\definecolor{grey18}{rgb}{0.18,0.18,0.18}
\definecolor{grey19}{rgb}{0.19,0.19,0.19}
\definecolor{grey1}{rgb}{0.01,0.01,0.01}
\definecolor{grey20}{rgb}{0.20,0.20,0.20}
\definecolor{grey21}{rgb}{0.21,0.21,0.21}
\definecolor{grey22}{rgb}{0.22,0.22,0.22}
\definecolor{grey23}{rgb}{0.23,0.23,0.23}
\definecolor{grey24}{rgb}{0.24,0.24,0.24}
\definecolor{grey25}{rgb}{0.25,0.25,0.25}
\definecolor{grey26}{rgb}{0.26,0.26,0.26}
\definecolor{grey27}{rgb}{0.27,0.27,0.27}
\definecolor{grey28}{rgb}{0.28,0.28,0.28}
\definecolor{grey29}{rgb}{0.29,0.29,0.29}
\definecolor{grey2}{rgb}{0.02,0.02,0.02}
\definecolor{grey30}{rgb}{0.30,0.30,0.30}
\definecolor{grey31}{rgb}{0.31,0.31,0.31}
\definecolor{grey32}{rgb}{0.32,0.32,0.32}
\definecolor{grey33}{rgb}{0.33,0.33,0.33}
\definecolor{grey34}{rgb}{0.34,0.34,0.34}
\definecolor{grey35}{rgb}{0.35,0.35,0.35}
\definecolor{grey36}{rgb}{0.36,0.36,0.36}
\definecolor{grey37}{rgb}{0.37,0.37,0.37}
\definecolor{grey38}{rgb}{0.38,0.38,0.38}
\definecolor{grey39}{rgb}{0.39,0.39,0.39}
\definecolor{grey3}{rgb}{0.03,0.03,0.03}
\definecolor{grey40}{rgb}{0.40,0.40,0.40}
\definecolor{grey41}{rgb}{0.41,0.41,0.41}
\definecolor{grey42}{rgb}{0.42,0.42,0.42}
\definecolor{grey43}{rgb}{0.43,0.43,0.43}
\definecolor{grey44}{rgb}{0.44,0.44,0.44}
\definecolor{grey45}{rgb}{0.45,0.45,0.45}
\definecolor{grey46}{rgb}{0.46,0.46,0.46}
\definecolor{grey47}{rgb}{0.47,0.47,0.47}
\definecolor{grey48}{rgb}{0.48,0.48,0.48}
\definecolor{grey49}{rgb}{0.49,0.49,0.49}
\definecolor{grey4}{rgb}{0.04,0.04,0.04}
\definecolor{grey50}{rgb}{0.50,0.50,0.50}
\definecolor{grey51}{rgb}{0.51,0.51,0.51}
\definecolor{grey52}{rgb}{0.52,0.52,0.52}
\definecolor{grey53}{rgb}{0.53,0.53,0.53}
\definecolor{grey54}{rgb}{0.54,0.54,0.54}
\definecolor{grey55}{rgb}{0.55,0.55,0.55}
\definecolor{grey56}{rgb}{0.56,0.56,0.56}
\definecolor{grey57}{rgb}{0.57,0.57,0.57}
\definecolor{grey58}{rgb}{0.58,0.58,0.58}
\definecolor{grey59}{rgb}{0.59,0.59,0.59}
\definecolor{grey5}{rgb}{0.05,0.05,0.05}
\definecolor{grey60}{rgb}{0.60,0.60,0.60}
\definecolor{grey61}{rgb}{0.61,0.61,0.61}
\definecolor{grey62}{rgb}{0.62,0.62,0.62}
\definecolor{grey63}{rgb}{0.63,0.63,0.63}
\definecolor{grey64}{rgb}{0.64,0.64,0.64}
\definecolor{grey65}{rgb}{0.65,0.65,0.65}
\definecolor{grey66}{rgb}{0.66,0.66,0.66}
\definecolor{grey67}{rgb}{0.67,0.67,0.67}
\definecolor{grey68}{rgb}{0.68,0.68,0.68}
\definecolor{grey69}{rgb}{0.69,0.69,0.69}
\definecolor{grey6}{rgb}{0.06,0.06,0.06}
\definecolor{grey70}{rgb}{0.70,0.70,0.70}
\definecolor{grey71}{rgb}{0.71,0.71,0.71}
\definecolor{grey72}{rgb}{0.72,0.72,0.72}
\definecolor{grey73}{rgb}{0.73,0.73,0.73}
\definecolor{grey74}{rgb}{0.74,0.74,0.74}
\definecolor{grey75}{rgb}{0.75,0.75,0.75}
\definecolor{grey76}{rgb}{0.76,0.76,0.76}
\definecolor{grey77}{rgb}{0.77,0.77,0.77}
\definecolor{grey78}{rgb}{0.78,0.78,0.78}
\definecolor{grey79}{rgb}{0.79,0.79,0.79}
\definecolor{grey7}{rgb}{0.07,0.07,0.07}
\definecolor{grey80}{rgb}{0.80,0.80,0.80}
\definecolor{grey81}{rgb}{0.81,0.81,0.81}
\definecolor{grey82}{rgb}{0.82,0.82,0.82}
\definecolor{grey83}{rgb}{0.83,0.83,0.83}
\definecolor{grey84}{rgb}{0.84,0.84,0.84}
\definecolor{grey85}{rgb}{0.85,0.85,0.85}
\definecolor{grey86}{rgb}{0.86,0.86,0.86}
\definecolor{grey87}{rgb}{0.87,0.87,0.87}
\definecolor{grey88}{rgb}{0.88,0.88,0.88}
\definecolor{grey89}{rgb}{0.89,0.89,0.89}
\definecolor{grey8}{rgb}{0.08,0.08,0.08}
\definecolor{grey90}{rgb}{0.90,0.90,0.90}
\definecolor{grey91}{rgb}{0.91,0.91,0.91}
\definecolor{grey92}{rgb}{0.92,0.92,0.92}
\definecolor{grey93}{rgb}{0.93,0.93,0.93}
\definecolor{grey94}{rgb}{0.94,0.94,0.94}
\definecolor{grey95}{rgb}{0.95,0.95,0.95}
\definecolor{grey96}{rgb}{0.96,0.96,0.96}
\definecolor{grey97}{rgb}{0.97,0.97,0.97}
\definecolor{grey98}{rgb}{0.98,0.98,0.98}
\definecolor{grey99}{rgb}{0.99,0.99,0.99}
\definecolor{grey9}{rgb}{0.09,0.09,0.09}
\definecolor{grey}{rgb}{0.75,0.75,0.75}
\definecolor{honeydew1}{rgb}{0.94,1.00,0.94}
\definecolor{honeydew2}{rgb}{0.88,0.93,0.88}
\definecolor{honeydew3}{rgb}{0.76,0.80,0.76}
\definecolor{honeydew4}{rgb}{0.51,0.55,0.51}
\definecolor{honeydew}{rgb}{0.94,1.00,0.94}
\definecolor{hotpink}{rgb}{1.00,0.41,0.71}
\definecolor{indianred}{rgb}{0.80,0.36,0.36}
\definecolor{ivory1}{rgb}{1.00,1.00,0.94}
\definecolor{ivory2}{rgb}{0.93,0.93,0.88}
\definecolor{ivory3}{rgb}{0.80,0.80,0.76}
\definecolor{ivory4}{rgb}{0.55,0.55,0.51}
\definecolor{ivory}{rgb}{1.00,1.00,0.94}
\definecolor{khaki1}{rgb}{1.00,0.96,0.56}
\definecolor{khaki2}{rgb}{0.93,0.90,0.52}
\definecolor{khaki3}{rgb}{0.80,0.78,0.45}
\definecolor{khaki4}{rgb}{0.55,0.53,0.31}
\definecolor{khaki}{rgb}{0.94,0.90,0.55}
\definecolor{lavenderblush}{rgb}{1.00,0.94,0.96}
\definecolor{lavender}{rgb}{0.90,0.90,0.98}
\definecolor{lawngreen}{rgb}{0.49,0.99,0.00}
\definecolor{lemonchiffon}{rgb}{1.00,0.98,0.80}
\definecolor{lightblue}{rgb}{0.68,0.85,0.90}
\definecolor{lightcoral}{rgb}{0.94,0.50,0.50}
\definecolor{lightcyan}{rgb}{0.88,1.00,1.00}
\definecolor{lightgoldenrod}{rgb}{0.93,0.87,0.51}
\definecolor{lightgoldenrod}{rgb}{0.98,0.98,0.82}
\definecolor{lightgray}{rgb}{0.83,0.83,0.83}
\definecolor{lightgreen}{rgb}{0.56,0.93,0.56}
\definecolor{lightgrey}{rgb}{0.83,0.83,0.83}
\definecolor{lightpink}{rgb}{1.00,0.71,0.76}
\definecolor{lightsalmon}{rgb}{1.00,0.63,0.48}
\definecolor{lightsea}{rgb}{0.13,0.70,0.67}
\definecolor{lightsky}{rgb}{0.53,0.81,0.98}
\definecolor{lightslate}{rgb}{0.47,0.53,0.60}
\definecolor{lightslate}{rgb}{0.47,0.53,0.60}
\definecolor{lightslate}{rgb}{0.52,0.44,1.00}
\definecolor{lightsteel}{rgb}{0.69,0.77,0.87}
\definecolor{lightyellow}{rgb}{1.00,1.00,0.88}
\definecolor{limegreen}{rgb}{0.20,0.80,0.20}
\definecolor{linen}{rgb}{0.98,0.94,0.90}
\definecolor{magenta1}{rgb}{1.00,0.00,1.00}
\definecolor{magenta2}{rgb}{0.93,0.00,0.93}
\definecolor{magenta3}{rgb}{0.80,0.00,0.80}
\definecolor{magenta4}{rgb}{0.55,0.00,0.55}
\definecolor{magenta}{rgb}{1.00,0.00,1.00}
\definecolor{maroon1}{rgb}{1.00,0.20,0.70}
\definecolor{maroon2}{rgb}{0.93,0.19,0.65}
\definecolor{maroon3}{rgb}{0.80,0.16,0.56}
\definecolor{maroon4}{rgb}{0.55,0.11,0.38}
\definecolor{maroon}{rgb}{0.69,0.19,0.38}
\definecolor{mediumaquamarine}{rgb}{0.40,0.80,0.67}
\definecolor{mediumblue}{rgb}{0.00,0.00,0.80}
\definecolor{mediumorchid}{rgb}{0.73,0.33,0.83}
\definecolor{mediumpurple}{rgb}{0.58,0.44,0.86}
\definecolor{mediumsea}{rgb}{0.24,0.70,0.44}
\definecolor{mediumslate}{rgb}{0.48,0.41,0.93}
\definecolor{mediumspring}{rgb}{0.00,0.98,0.60}
\definecolor{mediumturquoise}{rgb}{0.28,0.82,0.80}
\definecolor{mediumviolet}{rgb}{0.78,0.08,0.52}
\definecolor{midnightblue}{rgb}{0.10,0.10,0.44}
\definecolor{mintcream}{rgb}{0.96,1.00,0.98}
\definecolor{mistyrose}{rgb}{1.00,0.89,0.88}
\definecolor{moccasin}{rgb}{1.00,0.89,0.71}
\definecolor{navajowhite}{rgb}{1.00,0.87,0.68}
\definecolor{navyblue}{rgb}{0.00,0.00,0.50}
\definecolor{navy}{rgb}{0.00,0.00,0.50}
\definecolor{oldlace}{rgb}{0.99,0.96,0.90}
\definecolor{olivedrab}{rgb}{0.42,0.56,0.14}
\definecolor{orange1}{rgb}{1.00,0.65,0.00}
\definecolor{orange2}{rgb}{0.93,0.60,0.00}
\definecolor{orange3}{rgb}{0.80,0.52,0.00}
\definecolor{orange4}{rgb}{0.55,0.35,0.00}
\definecolor{orangered}{rgb}{1.00,0.27,0.00}
\definecolor{orange}{rgb}{1.00,0.65,0.00}
\definecolor{orchid1}{rgb}{1.00,0.51,0.98}
\definecolor{orchid2}{rgb}{0.93,0.48,0.91}
\definecolor{orchid3}{rgb}{0.80,0.41,0.79}
\definecolor{orchid4}{rgb}{0.55,0.28,0.54}
\definecolor{orchid}{rgb}{0.85,0.44,0.84}
\definecolor{palegoldenrod}{rgb}{0.93,0.91,0.67}
\definecolor{palegreen}{rgb}{0.60,0.98,0.60}
\definecolor{paleturquoise}{rgb}{0.69,0.93,0.93}
\definecolor{paleviolet}{rgb}{0.86,0.44,0.58}
\definecolor{papayawhip}{rgb}{1.00,0.94,0.84}
\definecolor{peachpuff}{rgb}{1.00,0.85,0.73}
\definecolor{peru}{rgb}{0.80,0.52,0.25}
\definecolor{pink1}{rgb}{1.00,0.71,0.77}
\definecolor{pink2}{rgb}{0.93,0.66,0.72}
\definecolor{pink3}{rgb}{0.80,0.57,0.62}
\definecolor{pink4}{rgb}{0.55,0.39,0.42}
\definecolor{pink}{rgb}{1.00,0.75,0.80}
\definecolor{plum1}{rgb}{1.00,0.73,1.00}
\definecolor{plum2}{rgb}{0.93,0.68,0.93}
\definecolor{plum3}{rgb}{0.80,0.59,0.80}
\definecolor{plum4}{rgb}{0.55,0.40,0.55}
\definecolor{plum}{rgb}{0.87,0.63,0.87}
\definecolor{powderblue}{rgb}{0.69,0.88,0.90}
\definecolor{purple1}{rgb}{0.61,0.19,1.00}
\definecolor{purple2}{rgb}{0.57,0.17,0.93}
\definecolor{purple3}{rgb}{0.49,0.15,0.80}
\definecolor{purple4}{rgb}{0.33,0.10,0.55}
\definecolor{purple}{rgb}{0.63,0.13,0.94}
\definecolor{red1}{rgb}{1.00,0.00,0.00}
\definecolor{red2}{rgb}{0.93,0.00,0.00}
\definecolor{red3}{rgb}{0.80,0.00,0.00}
\definecolor{red4}{rgb}{0.55,0.00,0.00}
\definecolor{red}{rgb}{1.00,0.00,0.00}
\definecolor{rosybrown}{rgb}{0.74,0.56,0.56}
\definecolor{royalblue}{rgb}{0.25,0.41,0.88}
\definecolor{saddlebrown}{rgb}{0.55,0.27,0.07}
\definecolor{salmon1}{rgb}{1.00,0.55,0.41}
\definecolor{salmon2}{rgb}{0.93,0.51,0.38}
\definecolor{salmon3}{rgb}{0.80,0.44,0.33}
\definecolor{salmon4}{rgb}{0.55,0.30,0.22}
\definecolor{salmon}{rgb}{0.98,0.50,0.45}
\definecolor{sandybrown}{rgb}{0.96,0.64,0.38}
\definecolor{seagreen}{rgb}{0.18,0.55,0.34}
\definecolor{seashell1}{rgb}{1.00,0.96,0.93}
\definecolor{seashell2}{rgb}{0.93,0.90,0.87}
\definecolor{seashell3}{rgb}{0.80,0.77,0.75}
\definecolor{seashell4}{rgb}{0.55,0.53,0.51}
\definecolor{seashell}{rgb}{1.00,0.96,0.93}
\definecolor{sienna1}{rgb}{1.00,0.51,0.28}
\definecolor{sienna2}{rgb}{0.93,0.47,0.26}
\definecolor{sienna3}{rgb}{0.80,0.41,0.22}
\definecolor{sienna4}{rgb}{0.55,0.28,0.15}
\definecolor{sienna}{rgb}{0.63,0.32,0.18}
\definecolor{skyblue}{rgb}{0.53,0.81,0.92}
\definecolor{slateblue}{rgb}{0.42,0.35,0.80}
\definecolor{slategray}{rgb}{0.44,0.50,0.56}
\definecolor{slategrey}{rgb}{0.44,0.50,0.56}
\definecolor{snow1}{rgb}{1.00,0.98,0.98}
\definecolor{snow2}{rgb}{0.93,0.91,0.91}
\definecolor{snow3}{rgb}{0.80,0.79,0.79}
\definecolor{snow4}{rgb}{0.55,0.54,0.54}
\definecolor{snow}{rgb}{1.00,0.98,0.98}
\definecolor{springgreen}{rgb}{0.00,1.00,0.50}
\definecolor{steelblue}{rgb}{0.27,0.51,0.71}
\definecolor{tan1}{rgb}{1.00,0.65,0.31}
\definecolor{tan2}{rgb}{0.93,0.60,0.29}
\definecolor{tan3}{rgb}{0.80,0.52,0.25}
\definecolor{tan4}{rgb}{0.55,0.35,0.17}
\definecolor{tan}{rgb}{0.82,0.71,0.55}
\definecolor{thistle1}{rgb}{1.00,0.88,1.00}
\definecolor{thistle2}{rgb}{0.93,0.82,0.93}
\definecolor{thistle3}{rgb}{0.80,0.71,0.80}
\definecolor{thistle4}{rgb}{0.55,0.48,0.55}
\definecolor{thistle}{rgb}{0.85,0.75,0.85}
\definecolor{tomato1}{rgb}{1.00,0.39,0.28}
\definecolor{tomato2}{rgb}{0.93,0.36,0.26}
\definecolor{tomato3}{rgb}{0.80,0.31,0.22}
\definecolor{tomato4}{rgb}{0.55,0.21,0.15}
\definecolor{tomato}{rgb}{1.00,0.39,0.28}
\definecolor{turquoise1}{rgb}{0.00,0.96,1.00}
\definecolor{turquoise2}{rgb}{0.00,0.90,0.93}
\definecolor{turquoise3}{rgb}{0.00,0.77,0.80}
\definecolor{turquoise4}{rgb}{0.00,0.53,0.55}
\definecolor{turquoise}{rgb}{0.25,0.88,0.82}
\definecolor{violetred}{rgb}{0.82,0.13,0.56}
\definecolor{violet}{rgb}{0.93,0.51,0.93}
\definecolor{wheat1}{rgb}{1.00,0.91,0.73}
\definecolor{wheat2}{rgb}{0.93,0.85,0.68}
\definecolor{wheat3}{rgb}{0.80,0.73,0.59}
\definecolor{wheat4}{rgb}{0.55,0.49,0.40}
\definecolor{wheat}{rgb}{0.96,0.87,0.70}
\definecolor{whitesmoke}{rgb}{0.96,0.96,0.96}
\definecolor{white}{rgb}{1.00,1.00,1.00}
\definecolor{yellow1}{rgb}{1.00,1.00,0.00}
\definecolor{yellow2}{rgb}{0.93,0.93,0.00}
\definecolor{yellow3}{rgb}{0.80,0.80,0.00}
\definecolor{yellow4}{rgb}{0.55,0.55,0.00}
\definecolor{yellowgreen}{rgb}{0.60,0.80,0.20}
\definecolor{yellow}{rgb}{1.00,1.00,0.00}

\usepackage[T1]{fontenc}

\DeclareMathAlphabet{\mathpzc}{OT1}{pzc}{m}{it} 

\theoremstyle{plain}
\newtheorem{thm}{Theorem}[section]
\newtheorem{cor}[thm]{Corollary}
\newtheorem{lmm}[thm]{Lemma}
\newtheorem{prpn}[thm]{Proposition}
\newtheorem{rem}[thm]{Remark}

\theoremstyle{definition}
\newtheorem{defn}[thm]{Definition}
\newtheorem{eg}[thm]{Example}

\newtheorem*{thma}{Theorem A}
\newtheorem*{thmb}{Theorem B}
\newtheorem*{thmc}{Theorem C}
\newtheorem*{thmd}{Theorem D}

\newcommand{\R}{\mathbb{R}}
\newcommand{\C}{\mathbb{C}}

\newcommand{\bgd}{\begin{displaymath}}
\newcommand{\edd}{\end{displaymath}}
\newcommand{\bgc}{\begin{center}}
\newcommand{\edc}{\end{center}}
\newcommand{\hf}{\hspace*{0.5cm}}
\newcommand{\lan}{\left\langle}
\newcommand{\ran}{\right\rangle}
\newcommand{\upq}{U(p,q)}
\newcommand{\ep}{\varepsilon}

\numberwithin{equation}{section}

\usepackage[left=3.00cm, right=3.00cm, top=2.00cm, bottom=3.00cm]{geometry} 
\usepackage{amsmath,amsthm,amssymb,amsfonts} 
\usepackage{thmtools}
\usepackage{etexcmds}
\usepackage{mathtools} 
\usepackage{xcolor} 
\usepackage{wasysym,textcomp,array,arcs,polynom,cancel,enumerate,dcolumn} 
\usepackage{float} 

\usepackage{graphicx,wrapfig}
\graphicspath{{Figures/}} 
\usepackage{tikz-cd}
\usepackage{varwidth} 
\usepackage[all,cmtip]{xy}
\usepackage{bbm}
\usepackage{enumitem}
\usepackage{calc}
\usepackage{caption}
\usepackage{mathrsfs}
\numberwithin{equation}{section}

\usepackage[colorinlistoftodos]{todonotes}

\allowdisplaybreaks

\definecolor{citecol}{RGB}{12,127,172}
\definecolor{example}{rgb}{0.6, 0.4, 0.8}
\definecolor{remark}{rgb}{0.0, 0.5, 0.0}

\newcommand{\rbb}{\ensuremath{\mathbb{R}}} 
\newcommand{\cbb}{\ensuremath{\mathbb{C}}}

\newcommand{\vbf}{\ensuremath{\mathbf{v}}}
\renewcommand{\bf}[1]{\ensuremath{\mathbf{#1}}}
\newcommand{\sbb}{\ensuremath{\mathbb{S}}}
\newcommand{\cali}[1]{\ensuremath{\mathcal{#1}}}

\newcommand{\bb}[1]{\ensuremath{\mathbb{#1}}} 
\newcommand{\delbydel}[2]{\ensuremath{\dfrac{\partial#1}{\partial#2}}} 
\newcommand{\crf}{\ensuremath{\mathrm{Cr}(f)}} 
\newcommand{\grad}{\ensuremath{\nabla}} 
\newcommand{\comp}{\mathbin{\mathchoice {\compcent\scriptstyle}{\compcent\scriptstyle} {\compcent\scriptscriptstyle}{\compcent\scriptscriptstyle}}} 
\newcommand{\compcent}[1]{\vcenter{\hbox{$#1\circ$}}} 
 
\newcommand{\dist}{\ensuremath{\operatorname{dist}}} 
 
\newcommand{\isom}{\ensuremath{\cong}} 

\newcommand{\defeq}{\vcentcolon=}

\newcommand{\cutn}{\ensuremath{\mathrm{Cu}(N)}}
\newcommand{\sen}{\ensuremath{\mathrm{Se}(N)}}
\newcommand{\innerprod}[2]{\ensuremath{\left\langle #1,#2\right\rangle}}
\newcommand{\norm}[1]{\ensuremath{\left\|#1\right\|}}

\newcommand{\paran}[1]{\ensuremath{\left( #1 \right)}}
\newcommand{\curlybracket}[1]{\ensuremath{\left\{ #1 \right\}}}

\newcommand{\abs}[1]{\ensuremath{\left|#1\right|}}
\newcommand{\aTransInverse}{\ensuremath{\paran{A^T}^{-1}}}

\newcommand{\trace}[1]{\ensuremath{\mathrm{tr}\left( #1 \right)}}
\newcommand{\cu}{\ensuremath{\mathrm{Cu}(p)}}
 
\newcommand{\co}{\ensuremath{\mathrm{Co}(x_0,\delta)~}} 
\newcommand{\costar}{\ensuremath{\mathrm{Co}^\star(x_0,\delta)~}}
\newcommand{\Ball}{\ensuremath{\overline{B(x_0,\delta)}~}}

\newcommand{\hess}{\mathrm{Hess}}

\newcommand{\spmat}[1]{%
  \left(\begin{smallmatrix}#1\end{smallmatrix}\right)%
}

\usepackage{pdfpages}
\usepackage[pdfauthor={},%
        pdftitle={},%
        pdftex]{hyperref}
\hypersetup{
	colorlinks,
	citecolor=red,
	filecolor=black,
	linkcolor=blue,
	urlcolor=blue
}

\begin{document}

\title{A connection between cut locus, Thom space and Morse-Bott functions}

\author{Somnath Basu}
\address{Indian Institute of Science Education \& Research, Mohanpur 741246, West Bengal, India}
\email{somnath.basu@iiserkol.ac.in}

\author{Sachchidanand Prasad}
\address{Indian Institute of Science Education \& Research, Mohanpur 741246, West Bengal, India}
\email{sp17rs038@iiserkol.ac.in}
\thanks{The second author is funded by UGC (NET)-JRF fellowship.}
\date{}

\begin{abstract}
Associated to every closed, embedded submanifold $N$ in a connected Riemannian manifold $M$, there is the distance function $d_N$ which measures the distance of a point in $M$ from $N$. We analyze the square of this function and show that it is Morse-Bott on the complement of the cut locus $\cutn$ of $N$, provided $M$ is complete. Moreover, the gradient flow lines provide a deformation retraction of $M-\cutn$ to $N$. If $M$ is a closed manifold, then we prove that the Thom space of the normal bundle of $N$ is homeomorphic to $M/\cutn$. We also discuss several interesting results which are either applications of these or related observations regarding the theory of cut locus. These results include, but are not limited to, a computation of the local homology of singular matrices, a classification of the homotopy type of the cut locus of a homology sphere inside a sphere, a deformation of the indefinite unitary group $U(p,q)$ to $U(p)\times U(q)$ and a geometric deformation of $GL(n,\R)$ to $O(n,\R)$ which is different from the Gram-Schmidt retraction.
\end{abstract}


\maketitle

\tableofcontents


	
\section{Introduction} \label{intro}

\hf On a Riemannian manifold $M$, the distance function $d_N(\cdot):=d(N,\cdot)$ from a closed subset $N$ is fundamental in the study of variational problems. For instance, the viscosity solution of the Hamilton-Jacobi equation is given by the flow of the gradient vector of the distance function $d_N$, when $N$ is the smooth boundary of a relatively compact domain in manifolds \cite{MaMe03, LiNi05}. Although the distance function $d_N$ is not differentiable at $N$, squaring the function removes this issue. Associated to $N$ and the distance function $d_N$ is a set $\cutn$, the cut locus of $N$ in $M$. Cut locus of a point, a notion initiated by Poincare \cite{Poin05}, has been extensively studied (see \cite{Kob67} for a survey as well as \cite{Mye35, Buc77, Wol79, Sak96}). There has been work on the structure of the cut locus of submanifolds or closed sets \cite{Heb83, Singh87, Heb95, TaSa16}. Suitable simple examples indicate that $M-\cutn$ topologically deforms to $N$. One of our main results is the following (Theorem \ref{thm: Morse-Bott}).
\begin{thma}
\textit{Let $N$ be a closed embedded submanifold of a complete Riemannian manifold $M$ and $d:M\to \mathbb{R}$ denote the distance function with respect to $N$. If $f=d^2$, then its restriction to $M-\mathrm{Cu}(N)$ is a Morse-Bott function, with $N$ as the critical submanifold. Moreover, $M-\mathrm{Cu}(N)$ deforms to $N$ via the gradient flow of $f$.}
\end{thma}
\noindent It is observed that this deformation takes infinite time. To obtain a strong deformation retract, one reparaterizes the flow lines to be defined over $[0,1]$. It can be shown (Lemma \ref{defretM-N}) that the cut locus $\mathrm{Cu}(N)$ is a strong deformation retract of $M-N$. A primary motivation for Theorem A came from understanding the cut locus of $N=O(n,\rbb)$ inside $M= M(n,\rbb)$, equipped with the Euclidean metric. We show in \S \ref{Sec: Example} that the cut locus is the set $\mathrm{Sing}$ of singular matrices and the deformation of its complement is not the Gram-Schmidt deformation but rather the deformation obtained from the polar decomposition, i.e., $A\in GL(n,\R)$ deforms to $A\big(\sqrt{A^T A}\,\big)^{-1}$. Combining with a result of Hebda \cite[Theorem 1.4]{Heb83} we are able to compute the local homology of $\mathrm{Sing}$ (cf. Lemma \ref{link-sing} and Corollary \ref{locsinghom}).
\begin{thmb}
\textit{For $A\in M(n,\R)$
\bgd
H_{n^2-1-i}(\mathrm{Sing},\mathrm{Sing} -A)\cong \widetilde{H}^i(O(n-k,\R))
\edd
where $A\in \mathrm{Sing}$ has rank $k<n$. }
\end{thmb}
\noindent
\hf When the cut locus is empty, we deduce that $M$ is diffeomorphic to the normal bundle $\nu$ of $N$ in $M$. In particular, $M$ deforms to $N$. Among applications, we discuss two families of examples. We reprove the known fact that $GL(n,R)$ deforms to $O(n,\R)$ for any choice of left-invariant metric on $GL(n,\R)$ which is right-$O(n,\R)$-invariant. However, this deformation is not obtained topologically but by Morse-Bott flows. For a natural choice of such a metric, this deformation \eqref{GLdefOver2} is not the Gram-Schmidt deformation but one obtained from the polar decomposition. We also consider $U(p,q)$, the group preserving the indefinite form of signature $(p,q)$ on $\C^n$. We show (Theorem \ref{mainthm}) that $U(p,q)$ deforms to $U(p)\times U(q)$ for the left-invariant metric given by $\left\langle X,Y\right\rangle:=\textup{tr}(X^\ast Y)$. In particular, we show that the exponential map is surjective for $U(p,q)$ (Corollary \ref{expsurj}). To our knowledge, this method is different from the standard proof. \\
\hf For a Riemannian manifold we have the exponential map at $p\in M$, $\exp_p:T_pM\to M$. Let $\nu$ denote the normal bundle of $N$ in $M$. We will modify the exponential map (see \S \ref{Sec: Thom}) to define the \emph{rescaled exponential} $\widetilde{\exp}:D(\nu)\to M$, the domain of which is the unit disk bundle of $\nu$. The main result (Theorem \ref{Thomsp}) here is the observation that there is a connection between the cut locus $\mathrm{Cu}(N)$ and Thom space $\mathrm{Th}(\nu):=D(\nu)/S(\nu)$ of $\nu$. 
\begin{thmc}
\textit{ Let $N$ be an embedded submanifold inside a closed, connected Riemannian manifold $M$. If $\nu$ denotes the normal bundle of $N$ in $M$, then there is a homeomorphism
\begin{displaymath}
\widetilde{\exp}:D(\nu)/S(\nu) \xrightarrow{\cong}M/\mathrm{Cu}(N).
\end{displaymath}}
\end{thmc}
\noindent This immediately leads to a long exact sequence in homology (see \eqref{lesThom})
\begin{displaymath}
	\cdots\to  H_j(\mathrm{Cu}(N)) \stackrel{i_*}{\longrightarrow}H_j(M)\stackrel{q}{\longrightarrow} \widetilde{H}_j(\mathrm{Th}(\nu))\stackrel{\partial}{\longrightarrow} H_{j-1}(\mathrm{Cu}(N))\to \cdots.
\end{displaymath}
This is an useful tool in characterizing the homotopy type of the cut locus. We list a few applications and related results.
\begin{thmd}
\textit{Let $N$ be a homology $k$-sphere embedded in a Riemannian manifold $M^d$ homeomorphic to $S^d$.}\\
(1) \textit{If $d\ge k+3$, then $\mathrm{Cu}(N)$ is homotopy equivalent to $S^{d-k-1}$. Moreover, if $M,N$ are real analytic and the embedding is real analytic, then $\cutn$ is a simpicial complex of dimension at most $d-1$.}\\
(2) \textit{If $d=k+2$, then $\cutn$ has the homology of $S^1$. There exists homology $3$-spheres in $S^5$ for which $\cutn\simeq S^1$. However, for non-trivial knots $K$ in $S^3$, the cut locus is not homotopy equivalent to $S^1$. }
\end{thmd}
\noindent The above results are a combination of Theorem \ref{homsph}, Theorem \ref{Buchner} and Example \ref{codim2}. In general, the structure of the cut locus may be wild \cite{GlSi78, ItVi15, ItSa16}. Myers \cite{Mye35} had shown that if $M$ is a real analytic sphere, then $\mathrm{Cu}(p)$ is a finite tree each of whose edges is an analytic curve with finite length. Buchner \cite{Buc77} later generalized this result to cut locus of a point in higher dimensional manifolds. Theorem \ref{Buchner}, which states that the cut locus of an analytic submanifold (in an analytic manifold) is a simplicial complex, is a natural generalization of Buchner's result (and its proof). We attribute it to Buchner although it is not present in the original paper. This analyticity assumption also helps us to compute the homotopy type of the cut locus of a finite set of points in any closed, orientable, real analytic surface of genus $g$ (Theorem \ref{cutlocus-surface}). In Example \ref{codim2} we make some observations about the cut locus of embedded homology spheres of codimension $2$. This includes the case of real analytic knots in the round sphere $\mathbb{S}^3$.\\
\hf We apply our study of gradient of distance squared function to two families of Lie groups - $GL(n,\R)$ and $U(p,q)$. With a particular choice of left-invariant Riemannian metric which is right-invariant with respect to a maximally compact subgroup $K$, we analyze the geodesics and the cut locus of $K$. In both cases, we obtain that $G$ deforms to $K$ via Morse-Bott flow (Lemma \ref{CartanGLn} and Theorem \ref{mainthm}). Although these results are deducible from classical results of Cartan and Iwasawa, our method is geometric and specific to suitable choices of Riemannian metrics. It also makes very little use of structure theory of Lie algebras.\\[0.2cm]
\textbf{Organization of the paper:} In \S \ref{Sec: Preliminaries} we first recall basic definitions of Morse-Bott functions and cut locus of a subset (see \S \ref{Sec: Background}). In \S \ref{Sec: Example} we analyze the distance function from $O(n,\R)$ in $M(n,\R)$. This highlights and motivates Theorem A as well as allows for computation of local homology of singular matrices (Theorem B). In \S \ref{Sec: results} we first recall some relevant basic definitions from geometry (see \S \ref{Sec: basic results}). We make some observations about differentiablity of distance function (following Wolter \cite{Wol79}) and show that the cut locus is a simplicial complex for an analytic pair (following Buchner \cite{Buc77}). In \S \ref{Sec: Thom} we prove Theorem C and discuss some applications including Theorem D. In \S \ref{Sec: Morse-Bott} we prove Theorem A. In \S \ref{Sec: Lie} we discuss two specific examples - we analyze the cut locus of $O(n,\R)$ inside $GL(n,\R)$ in \S \ref{Sec: GLnSOn} and the cut locus of $U(p)\times U(q)$ inside $U(p,q)$ in \S \ref{Sec: Upq}. In appendix \ref{Sec: cts-s} we prove Proposition \ref{snucts}, the continuity of the map $\mathpzc{s}$ (see Definition \ref{snu}). This result is crucial for \S \ref{Sec: Thom}. In appendix \ref{Sec: der-app} we compute the derivative of the square root map for positive definite matrices (Lemma \ref{lemma: A.1}). We also analyze the differentiability of the map $A\mapsto \trace{\sqrt{A^TA}}$ in Lemma \ref{lemma: A.2}.\\


\section{Preliminaries}\label{Sec: Preliminaries}

\hf We recall the notion of Morse function and Morse-Bott function in \S \ref{Sec: Background}, keeping in mind the square of the distance function from a submanifold being a potential Morse-Bott function which we will analyze in \S  \ref{Sec: Morse-Bott}. We also recall the definition of cut locus of a subset in a Riemannian manifold. In Example \ref{join} we observe that the join of spheres being a sphere can be observed geometrically via cut locus. In \S \ref{Sec: Example} we analyze the cut locus of orthogonal matrices and compute the relative homology of the cut locus \eqref{locsing}. Along the way, we note that the geometric deformation of $GL(n,\R)$ to $O(n,\R)$, obtained via the distance squared function, is \textit{not} the Gram-Schmidt deformation.

\subsection{Background}\label{Sec: Background}

\noindent \hspace*{0.5cm}Given a smooth $n$-dimensonal manifold $M$, we say that a point $p\in M$ is a {\emph{critical point}} of a smooth function $f:M\to \rbb$ if 
\begin{displaymath}
df_p:T_pM\to T_{f(p)}\rbb
\end{displaymath}
vanishes. In a coordinate neighborhood $(\phi=(x_1,x_2,\ldots,x_n),U)$ around $p$ for all $j=1,2,\ldots,n$ we have 
\begin{displaymath}
\delbydel{(f\comp \phi^{-1})}{x_j}(\phi(p))=0. 
\end{displaymath}
A critical point $p$ is called {\emph{non-degenerate}} if determinant of the Hessian matrix
\begin{displaymath}
\hess_p(f) \defeq \left(\delbydel{^2(f\comp \phi^{-1})}{x_i\partial x_j}(\phi(p))\right)
\end{displaymath}
is non-zero. Let us denote the set of all critical points of $f$ by $\crf$. If all the critical points are non-degenerate, then $f$ is said to be a {\emph{Morse function}}. Morse-Bott functions are generalizations of Morse functions, where we are allowed to have non-degenerate critical submanifold. 

\begin{defn}[Morse-Bott functions] 
Let $M$ be a Riemannian manifold. A smooth submanifold $ N\subset M $ is said to be {\emph{non-degenerate critical submanifold}} of $f$ if $N\subseteq \crf$ and for any $p\in N$, $\hess_{p}(f)$ is non-degenerate in the direction normal to $N$ at $p$. The function $f$ is said to be {\emph{Morse-Bott}} if the connected components of $ \crf $ are non-degenerate critical submanifolds.
\end{defn}
\noindent Note that $\hess_p(f)$ is non-degenerate in the direction normal to $N$ at $p$ means for any $V\in (T_pN)^\perp$ there exists $W\in  (T_pN)^\perp$ such that $\hess_p(f)(V,W)\neq 0$.
\begin{eg}\label{eg: circle}
Let $M=\rbb^{n+1}$ equipped with the Euclidean metric $d$. If $N=\sbb^n$ is the unit sphere, then the distance between a point $\bf{p}\in \rbb^{n+1}$ and $N$ is given by
\begin{displaymath}
d(\bf{p},N) \defeq \inf_{\bf{q}\in N} d(\bf{p},\bf{q}).
\end{displaymath}
We shall denote by $d^2$ the square of the distance. Now consider the function
\begin{displaymath}
f:M\to \rbb,~~\bf{x}\mapsto d^2(\bf{x},N)= \begin{cases}
	\|\bf{x}\|^2-1, &\text{ if } \|\bf{x}\|\geq 1  \\
	1-\|\bf{x}\|^2,  &\text{ if } \|\bf{x}\|<1.
	\end{cases}
\end{displaymath}
The function $f:M-\{\mathbf{0}\}$ is a Morse-Bott function with $N=\sbb^n$ as the critical submanifold. 
\end{eg}
\noindent The trace function on $SO(n,\rbb), U(n,\cbb)$ and $Sp(n,\cbb)$ is a Morse-Bott function (cf. \cite[p.~90, Exercise~22]{BaHu04}). We refer the interested reader to \cite{BaHu04} for basic results on Morse-Bott theory.\\
\hf We shall now define the cut locus for a point. The notion of cut locus was first introduced for convex surfaces by Poincar\'{e} \cite{Poin05} in 1905 under the name \textit{la ligne de partage} meaning \textit{the dividing line}. 
\begin{defn}[Cut locus]
Let $M$ be a complete Riemannian manifold and $p\in M$. If $\cu$ denotes the \emph{cut locus} of $p$, then a point $q\in \cu$ if there exists a minimal geodesic joining $p$ to $q$ any extension of which beyond $q$ is not minimal. 
\end{defn}
\noindent Recall that a minimal geodesic joining $p$ and $q$ is a geodesic that realizes the distance between $p$ and $q$. The existence of minimal geodesics joining two given points is implied by completeness of the Riemannian manifold. Therefore, in almost all of the examples, the manifolds under consideration will be complete Riemannian manifolds. When $M=\sbb^n$, i.e., an $n$-sphere with the round metric induced from $\rbb^{n+1}$, for any $p\in \sbb^n,$ the cut locus $\cu$ will be the corresponding antipodal point. Later, in \S \ref{Sec: Thom} Definition \ref{cutlocus2}, we give a slightly different but equivalent definition of cut locus, following \cite[\S4.1]{Sak96}. \\
\hspace*{0.5cm}In order to have a definition of the cut locus for a submanifold (or a subset), we need to generalize the notion of a minimal geodesic.
\begin{defn}\label{distmin}
A geodesic $\gamma $ is called a \emph{distance minimal geodesic} joining $N$ to $p$ if there exists $q\in N$ such that $\gamma$ is a minimal geodesic joining $q$ to $p$ and $l(\gamma)= d(p,N) $. We will refer to such geodesics as \textit{$N$-geodesics}.
\end{defn}
\noindent If $N$ is an embedded submanifold, then an $N$-geodesic is necessarily orthogonal to $N$. This follows from the first variational principle. We are ready to define the cut locus for $N\subset M$. 
\begin{defn}[Cut locus]\label{cutlocus1}
Let $M$ be a Riemannian manifold and $N$ be any non-empty subset of $M.$ If $\cutn$ denotes the \emph{cut locus of $N$}, then we say that $q\in \cutn $ if and only if there exists a distance minimal geodesic joining $N$ to $q$ such that any extension of it beyond $q$ is not a distance minimal geodesic.
\end{defn}
\noindent The cut locus of a sphere (refer example \ref{eg: circle}) is its centre. The set $\cu$ is closed \cite[exercise 28.4, p. 363]{Po01}. In general, the cut locus of a subset need not be closed, as the following example from \cite{TaSa16} illustrates.
\begin{eg}[Sabau-Tanaka 2016]
Consider $\rbb^2$ with the Euclidean inner product. Let $\curlybracket{\theta_n}$, with $\theta_1\in (0,\pi)$, be a decreasing sequence converging to $0$. Let $\overline{B(\mathbf{0},1)}$ be the closed unit ball centered at $(0,0)$. Let $B_n:=B(q_n,1)$ be the open ball with radius $1$ and centered at $q_n$. We have chosen $q_n$ such that it does not belong to $\overline{B(\mathbf{0},1)}$ and denotes the center of the circle passing through $p_n=(\cos\theta_n,\sin\theta_n)$ and $p_{n+1} = (\cos\theta_{n+1},\sin\theta_{n+1})$. Define $ N\subset \rbb^2$ by
\begin{displaymath}
N \defeq \overline{B(\mathbf{0},1)}\setminus \cup_{n=1}^\infty B(q_n,1).
\end{displaymath}
\begin{figure}[h!]
	\centering
  	\includegraphics[scale=0.5]{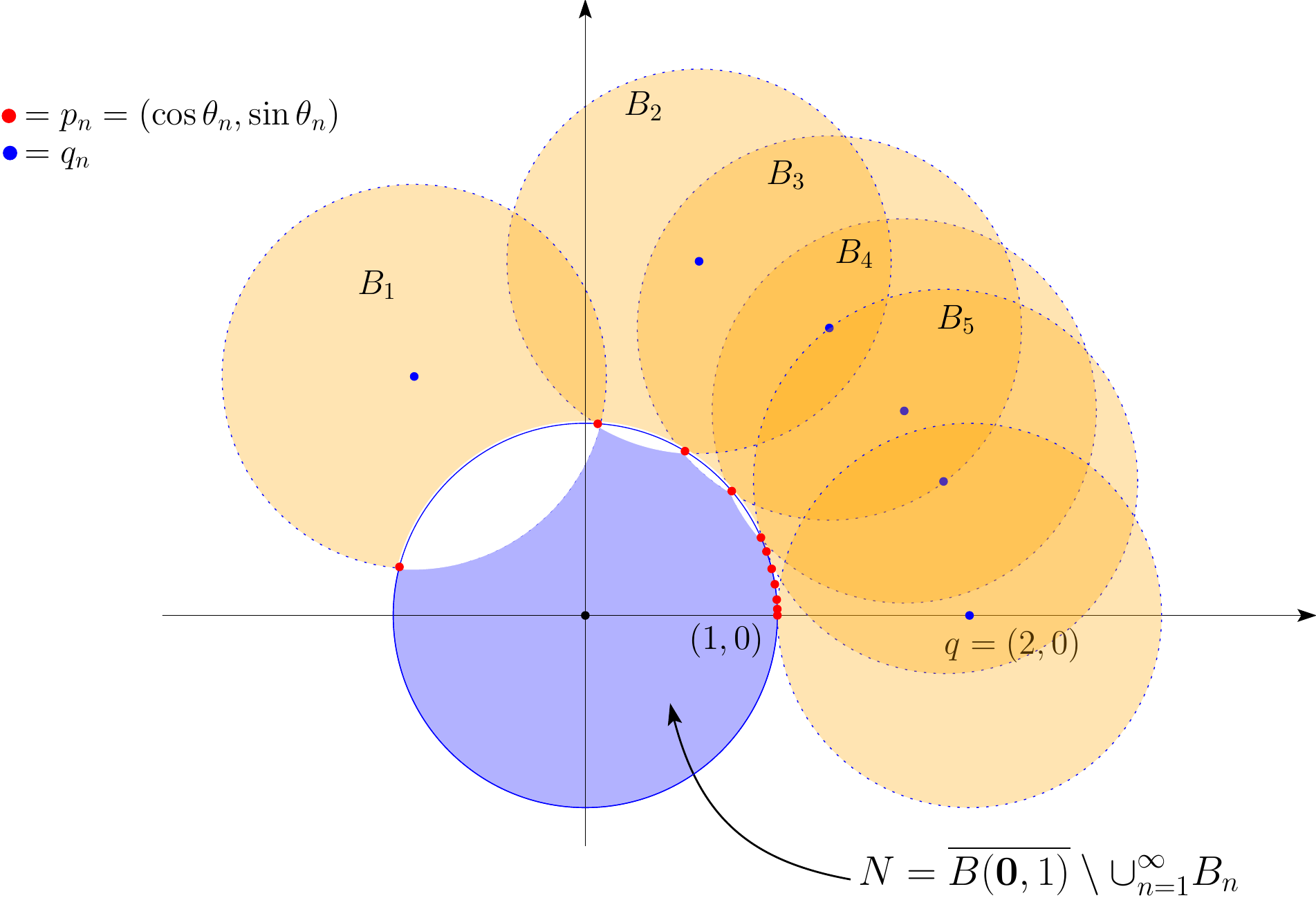}
	\caption{Cut locus need not be closed}
\end{figure}
Note that $N$ is a closed set and the sequence $\curlybracket{q_n}$ of cut points of $N$ converges to the point $(2,0)$. However, $(2,0)$ is not a cut point of $N$.
\end{eg}
In Theorem \ref{thm: Theorem 2} we will prove that for a submanifold $N$ the set $\cutn$ is closed by showing that it is the closure of the set of all points in $M$ which has at least two minimal geodesics joining $N$ to $p\in M$. 
\begin{eg}[Join induced by cut locus]\label{join}
Let $\mathbb{S}_i^k \hookrightarrow \mathbb{S}^n$ denote the embedding of the $k$-sphere in the first $k+1$ coordinates while $\mathbb{S}^{n-k-1}_l$ denote the embedding of the $(n-k-1)$-sphere in the last $n-k$ coordinates. It can be seen that $\textup{Cu}(\mathbb{S}_i^k)=\mathbb{S}_l^{n-k-1}$. In fact, starting at a point $p\in \mathbb{S}^k_i$ and travelling along a unit speed geodesic in a direction normal to $T_p\mathbb{S}^k_i$, we obtain a cut point at a distance $\pi/2$ from $\mathbb{S}^k_i$.
\begin{figure}[h!]
\centering
\includegraphics[scale=0.45]{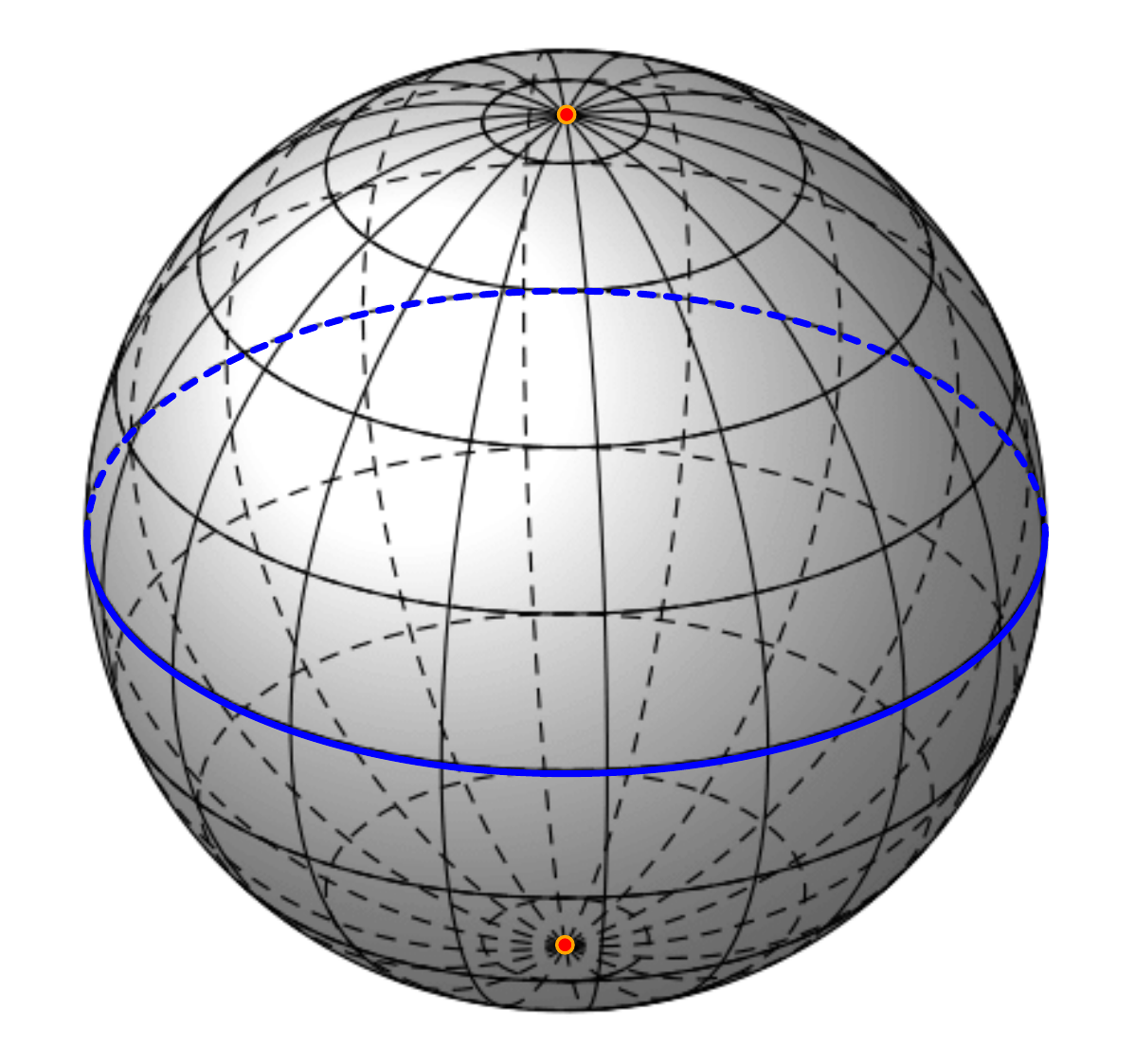}
\caption{The cut locus of the equator in $\mathbb{S}^2$}
\end{figure}
Moreover, in this case $\textup{Cu}(\mathbb{S}_l^{n-k-1})=\mathbb{S}_i^{k}$ and the $n$-sphere $\mathbb{S}^n$ can be expressed as the union of geodesic segments joining $\mathbb{S}_i^k$ to $\mathbb{S}_l^{n-k-1}$. This is a geometric variant of the fact that the $n$-sphere is the (topological) join of $S^k$ and $S^{n-k-1}$. We also observe that $\mathbb{S}^n - \mathbb{S}_l^{n-k-1}$ deforms to $\mathbb{S}_i^k$ while $\mathbb{S}^n - \mathbb{S}_i^{k}$ deforms to $\mathbb{S}_l^{n-k-1}$.\\
\hf In our example, let $\nu_i^{n-k}$ and $\nu_l^{k+1}$ denote the normal bundles of $\mathbb{S}_i^k$ and $\mathbb{S}_l^{n-k-1}$ respectively. We may express $\mathbb{S}^n$ as the union of normal disk bundles $D(\nu_i)$ and $D(\nu_l)$. These disk bundles are trivial and are glued along their common boundary $\mathbb{S}^k_i\times \mathbb{S}_l^{n-k-1}$ to produce $\mathbb{S}^n$. Moreover, $\mathbb{S}^k_i$ is an analytic submanifold of the real analytic Riemannian manifold $\mathbb{S}^n$ with the round metric. There is a generalization of this phenomenon due to Omori \cite[Lemmas 1.3-1.5, Theorem 3.1]{Omo68}.
\begin{thm}[Omori 1968]
Let $M$ be a compact, connected, real analytic Riemannian manifold which has an analytic submanifold $N$ such that the cut  point of $N$ with respect to every geodesic, which starts from $N$ and whose initial  direction is orthogonal to $N$ has a constant distance $\pi$ from $N$. Then $N'=\cutn$ is an analytic submanifold and $M$ has a decomposition $M = DN\cup_{\varphi} DN'$, where $DN,  DN'$ are normal disk bundles of $N, N'$ respectively.
\end{thm}
\end{eg}

\subsection{An illuminating example}\label{Sec: Example}

\hf Let $ M = M(n,\rbb) $, the set of $n\times n$ matrices, and $ N=O(n,\rbb) $, set of all orthogonal $n\times n$ matrices. Let $A,B\in M(n,\rbb)$. We fix the standard flat Euclidean metric on $M(n,\rbb)$ by identifying it with $\mathbb{R}^{n^2}$. This induces a distance function given by
\begin{equation*}\label{Example 1 Eq: 1}
	d(A,B) \defeq \sqrt{\trace{(A-B)^T(A-B)}}
\end{equation*}
Consider the distance squared function
\begin{displaymath}
f: GL(n,\rbb)\to \rbb,~~ A\mapsto d^2(A, O(n,\rbb)).
\end{displaymath}
\begin{lmm}
The function $f$ can be explicitly expressed as 
\begin{equation}
\label{eq:d2On}f(A) = n + \trace{A^TA} - 2\trace{\sqrt{A^TA}}.
\end{equation} 
\end{lmm}

\begin{proof}
 Let $A\in GL(n,\rbb)$ be any invertible matrix. Then,
\begin{align}
 		f(A) & =  \inf_{B\in O(n,\rbb)} \trace{(A-B)^T(A-B)}\nonumber\\
 			 & =  \inf_{B\in O(n,\rbb)} \big[\trace{A^TA}-\trace{A^TB}-\trace{B^TA}+\trace{B^TB} \big]\nonumber\\
 			 & =  \trace{A^TA} + \inf_{B\in O(n,\rbb)}\big[- 2\,\trace{A^TB} \big] + n \nonumber \\ 
 			 & =  \trace{A^TA}+ n -2\sup_{B\in O(n,\rbb)} \trace{A^TB} \label{maxdist}
\end{align}
In order to maximize the function 
$$h_A: O(n,\rbb)\to \rbb, B\mapsto \trace{A^TB}$$
for any invertible matrix $A$, we may first assume that $A$ is a diagonal matrix with positive entries. Then, 
\begin{displaymath}
\left|h_A(B)\right|  = \left|\trace{A^TB}\right|
	 	 = \Big|\sum_{i=1}^n a_{ii} b_{ii} \Big| 
		 \le \sum_{i=1}^n \abs{a_{ii}b_{ii}} 
		 \le \sum_{i=1}^n a_{ii} 
		= \trace{A^T} 
		= h_A(I).
\end{displaymath}
Thus, one of the maximizer is $B=I$. For a general non-singular matrix $A,$ we will use the \textit{singular value decomposition} (SVD). Write $A=UDV^T$, where $U$ and $V$ are $n\times n$ orthogonal matrices and $D$ is a diagonal matrix with positive entries. For any $B\in O(n,\rbb)$, using the cyclic property of trace, we can see that
$$\trace{A^TB} = \trace{D(U^TBV)}.\label{eq: 4.2}$$
Since $U^TBV$ is an orthogonal matrix, maximizing over $B$ reduces to the earlier observation that $B$ will be a maximizer if $U^TBV = I$, which implies $B = UV^T$. \\
\hspace*{0.5cm}Since $A$ is invertible, by the polar decomposition, there exists an orthogonal matrix $Q$ and a symmetric positive definite matrix $S=\sqrt{A^TA}$ such that $A=QS$. Since $S$ is symmetric matrix we can diagonalize it, i.e., $S = P\tilde{D}P^T$, where $P\in O(n,\rbb)$ and $\tilde{D}$ is a diagonal matrix with the eigenvalues of $S$ as its diagonal entries. Thus, 
			\begin{displaymath}
				A = QS = Q P\tilde{D}P^T.
			\end{displaymath}
Set $U=QP$, $V=P$ to obtain the SVD of $A$. In particular, the minimizer is given by
$$B=Q=A\big(\sqrt{A^TA}\big)^{-1}.$$
Therefore, 
	\begin{equation*}
		f(A) = n+\trace{A^TA}  - 2\,\trace{\sqrt{A^TA}}
	\end{equation*}
for invertible matrices.\\
\hspace*{0.5cm}In order to compute $f$ for a non-invertible matrix $A$, we note that $GL(n,\rbb)$ is dense in $M(n,\rbb)$ and that $\sqrt{A^TA}$ is well-defined for $A\in M(n,\rbb)$. The continuity of the map $A\mapsto \sqrt{A^T A}$ on $M(n,\rbb)$ implies that the same formula \eqref{eq:d2On} for $f$ applies to $A$ as well. 
\end{proof}
\noindent In order to understand the differentiability of $f$, it suffices to analyze the function $A\mapsto \trace{\sqrt{A^TA}}$. 
\begin{lmm}
	The map $g : M(n,\rbb)\to \rbb, ~A\mapsto \trace{\sqrt{A^TA}}$ is differentiable if and only if $A$ is invertible. 
\end{lmm}
\noindent
The proof of this postponed to the appendix (see Lemma \ref{lemma: A.2}). \\
\hspace*{0.5cm}Let us define the function 
\begin{displaymath}
	\phi : M(n,\rbb)\to \rbb,~~ A\mapsto \trace{\sqrt{A^TA}}.
\end{displaymath}
{ We claim that 
\begin{equation}\label{derivativeOfSquareRoot}
   		D\phi_A(H) = \trace{\int_0^\infty e^{-t\sqrt{A^TA}}\paran{A^TH+H^TA}e^{-t\sqrt{A^TA}}~dt} 
\end{equation}}
The following lemma (see Lemma \ref{lemma: A.1} for a proof) along with chain rule will prove our claim.
\begin{lmm}
 Let $A$ be a positive definite matrix and $\psi:A\mapsto \sqrt{A}$. Then 
\begin{equation*}\label{eq: sqrtderivative}
D\psi_A(H) = \int_0^\infty e^{-t\sqrt{A}}He^{-t\sqrt{A}}~dt,
\end{equation*}
for any symmetric matrix $H.$
\end{lmm}
\noindent We may drastically simplify, using basic analysis and linear algebra, the derivative of $\phi$ given by \eqref{derivativeOfSquareRoot} to obtain
\begin{equation*}
D\phi_A(H) = \innerprod{A\paran{\sqrt{A^TA}}^{-1}}{H}.
\end{equation*}
\hspace*{0.5cm}For  any $A\in GL(n,\rbb)$
\begin{equation*}\label{gradfForGL}
Df_A  = 2A-2A\paran{\sqrt{A^TA}}^{-1}=-2A \paran{\sqrt{A^TA}^{-1}-I}.
\end{equation*}
Hence, the negative gradient of the function $f$, restricted to $GL(n,\rbb)$ is given by 
\begin{displaymath}
	-\grad f\big|_A = 2A\paran{\sqrt{A^TA}^{-1}-I}.
\end{displaymath}
The critical points are orthogonal matrices. If $\gamma(t)$ is an integral curve of $-\grad f$ initialized at $A$, then $\gamma(0)=A$ and 
\begin{equation}
\dfrac{d\gamma}{dt} =-2\gamma(t)+2\gamma(t)\paran{\sqrt{\gamma(t)^T\gamma(t)}}^{-1}=-2\gamma(t)+2\paran{\gamma(t)^T}^{-1}\sqrt{\gamma(t)^T\gamma(t)} \label{eq: 4.5}.
\end{equation}
Take the test solution of \eqref{eq: 4.5} given by
\begin{equation}\label{eq: 4.6}
\gamma(t)  = Ae^{-2t} + (1-e^{-2t})\paran{A^T}^{-1}\sqrt{A^TA} = Ae^{-2t} + (1-e^{-2t})A\paran{\sqrt{A^TA}}^{-1}.
\end{equation}
\noindent In order to show that $\gamma(t)$ satisfies \eqref{eq: 4.5}, we may verify the following simplifications:
\begin{align*}
\gamma(t)^T\gamma(t) & = \paran{\sqrt{A^TA}~e^{-2t}+\paran{1-e^{-2t}}I}^2 \\
\paran{\sqrt{\gamma(t)^T\gamma(t)}}^T & = \paran{\sqrt{A^TA}A^{-1}\gamma(t)}^T = \gamma(t)^T\aTransInverse \sqrt{A^TA}
\end{align*}
This implies that 
\begin{equation*}\label{sqrtGammaTransGamma}
(\gamma(t)^T)^{-1}\sqrt{\gamma(t)^T\gamma(t)} = (A^T)^{-1}\sqrt{A^TA}.
\end{equation*}
\noindent The right hand side of \eqref{eq: 4.5}, with the test solution, can be simplified to 
\begin{displaymath}
-2Ae^{-2t} + 2e^{-2t}\aTransInverse \sqrt{A^TA}
\end{displaymath}
which is the derivative of $\gamma$. Thus, $\gamma(t)$, as defined in \eqref{eq: 4.6}, is the required flow line which deforms $GL(n,\rbb)$ to $O(n,\rbb).$ In particular, $GL^+(n,\rbb)$ deforms to $SO(n,\rbb)$ and other component of $GL(n,\rbb)$ deforms to $O(n,\rbb)\setminus SO(n,\rbb)$. We note, however, that this deformation takes infinite time to perform the retraction.
\begin{rem}
A modified curve
\begin{equation}\label{GLdefOver1}
\eta(t)=A(1-t)+tA\paran{\sqrt{A^TA}}^{-1}
\end{equation}
with the same image as $\gamma$, defines an actual deformation retraction of $GL(n,\rbb)$ to $O(n,\rbb)$. Apart from its origin via the distance function, this is a geometric deformation in the following sense. Given $A\in GL(n,\rbb)$, consider its columns as an ordered basis. This deformation deforms the ordered basis according to the length of the basis vectors and mutual angles between pairs of basis vectors in a geometrically uniform manner. This is in sharp contrast with Gram-Schmidt orthogonalization, also a deformation of $GL(n,\rbb)$ to $O(n,\rbb)$, which is asymmetric as it never changes the direction of the first column, the modified second column only depends on the first two columns and so on.
\end{rem}
\noindent\hspace*{0.5cm}We now show that $f$ is Morse-Bott. The tangent space $T_I O(n,\rbb)$ consists of skew-symmetric matrices while the normal vectors at $I_n$ are the symmetric matrices. As left translation by an orthogonal matrix is an isometry of $M(n,\rbb)$, normal vectors at $A\in O(n,\rbb)$ are of the form $AW$ for symmetric matrices $W$. Since 
\begin{displaymath}
	Df_A(H) = 2\innerprod{A}{H}-2\innerprod{A\big(\sqrt{A^TA}\big)^{-1}}{H}
\end{displaymath}
the relevant Hessian is 
\begin{equation*}
	\hess(f)_A(H,H') = \lim_{t\to 0}\dfrac{Df_{A+tH'}(H)-Df_A(H)}{t} 
\end{equation*}
with $H=AW, H'=AW'$ and symmetric matrices $W,W'$. A standard computation leads to
\begin{equation*}
\hess(f)_A(H,H') =2\,\trace{H^TH'}=2\left\langle H, H'\right\rangle.
\end{equation*}
Therefore, the Hessian matrix restricted to $(T_A O(n,\rbb))^\perp$ is $2I_{\frac{n(n+1)}{2}}$. This is a recurring feature of distance squared functions associated to embedded submanifolds (see Proposition \ref{dsq-Fermi}).\\
\noindent \hf There is a relationship between the local homology of cut loci and the reduced $\check{\textup{C}}$ech cohomology of the \textit{link} of a point in the cut locus. This is due to Hebda \cite{Heb83}, Theorem 1.4 and the remark following it. 
\begin{defn}
Let $N$ be an embedded submanifold of a complete smooth Riemannian manifold $M$. For each $q\in \cutn$, consider the set $\Lambda(q,N)$ of unit tangent vectors at $q$ so that the associated geodesics realize the distance between $q$ and $N$. This set is called the \textit{link} of $q$ with respect to $N$. \\
\hf The set of points in $N$ obtained by the end points of the geodesics associted to $\Lambda(q,N)$ will be called the \textit{equidistant set}, denoted by $\mathrm{Eq}(q,N)$, of $q$ with respect to $N$. 
\end{defn}
\noindent Since the equidistant set $\mathrm{Eq}(q,N)$, consisting of points which realize the distance $d(q,N)$, is obtained by exponentiating the points in $\Lambda(q,N)$, there is a natural surjection map from $\Lambda(q,N)$ to $\mathrm{Eq}(q,N)$. 
\begin{thm}[Hebda 1983]
Let $N$ be a properly embedded submanifold of a complete Riemannian manifold $M$ of dimension $n$. If $q\in \cutn$ and $v$ is an element of $\Lambda:=\Lambda(q,N)$, then there is an isomorphism
\begin{equation}\label{duality}
\check{H}^i(\Lambda,v)\cong H_{n-1-i}(\cutn,\cutn -q)
\end{equation}
\end{thm}
We are interested in computing $\Lambda(A,O(n,\R))$ for singular matrices $A$. Note that geodesics in $M(n,\R)$, initialized at $A$, are straight lines and any two such geodesics can never meet other then at $A$. Therefore, there is a natural identification between the link and the equidistant set of $A$. 
\begin{lmm}\label{link-sing}
If $A\in M(n,\R)$ is singular of rank $k$, then $\mathrm{Eq}(A,O(n,\R))$ is homeomorphic to $O(n-k,\R)$.
\end{lmm}
\begin{proof}
Using the singular vaue decomposition, we write $A=UDV^T$, where $U,V\in O(n,\R)$ and $D$ is a diagonal matrix with entries the eigenvalues of $\sqrt{A^T A}$. If we specify that the diagonal entries of $D$ are arranged in decreasing order, then $D$ is unique. Moreover, as $A$ has rank $k<n$, the first $k$ diagonal entries of $D$ are positive while the last $n-k$ diagonal entries are zero. In order to find the matrices in $O(n,\R)$ which realize the distance $d(A,O(n,\R))$, by \eqref{maxdist}, it suffices to find $B\in O(n,\R)$ such that 
\bgd
\sup_{B\in O(n,\R)} \trace{A^T B}=\sup_{B\in O(n,\R)} \trace{VDU^T B}=\sup_{B\in O(n,\R)} \trace{DU^T BV}
\edd
is maximized. However, $U^TBV\in O(n,\R)$ has orthonormal rows and the specific form of $D$ implies that the maximum happens if and only if $U^T BV$ has $e_1,\ldots,e_k$ as the the first $k$ rows, in order. Therefore, $U^T BV$ is a block orthogonal matrix, with blocks of $I_k$ and $C\in O(n-k,\R)$, i.e., $B\in U(I_k \times O(n-k,\R))V^T$.
\end{proof}
\begin{cor}\label{locsinghom}
Let $\mathrm{Sing}$ denote the space of singular matrices in $M(n,\R)$. If $A\in \mathrm{Sing}$ is of rank $k<n$, then then there is an isomorphism
\begin{equation}\label{locsing}
\widetilde{H}^i(O(n-k,\R))\cong H_{n^2-1-i}(\mathrm{Sing},\mathrm{Sing} -A).
\end{equation}
\end{cor}
\begin{proof}
It follows from Lemma \ref{link-sing} that $\Lambda(A,O(n,\R))\cong O(n-k,\R)$ if $A$ has rank $k$. Since $O(n-k,\R)$ is a manifold, $\check{\textup{C}}$ech and singular cohomology groups are isomorphic. The space $\mathrm{Sing}$ is a star-convex set, whence all homotopy and homology groups are that of a point. Applying \eqref{duality} in our case, we obtain an isomorphism
\begin{equation*}
\widetilde{H}^i(O(n-k,\R))\cong H_{n^2-1-i}(\mathrm{Sing},\mathrm{Sing} -A)
\end{equation*}
between reduced cohomology and local homology groups. In particular, the local homology of the cut locus at $A$ detects the rank of $A$. 
\end{proof}
\noindent Similar computations hold for $U(n,\C)$ and singular $n\times n$ complex matrices.


\section{Main Results}\label{Sec: results}

\hf We recall some results about exponential maps and Fermi coordinates in \S \ref{Sec: basic results}. A result of Wolter \cite{Wol79} may be generalized to prove (Lemma \ref{Lmm: singdsq}) that the distance squared function from a submanifold is not differentiable on the seperating set. This result may be well known to experts but the proof, following Wolter, is elementary. Buchner's result \cite{Buc77} may be generalized to prove (Theorem \ref{Buchner}) that the cut locus is a simplicial complex for real analytic pairs. In \S \ref{Sec: Thom} we recall the notion of Thom space and apply it to the normal bundle of an embedded submanifold in a closed, connected Riemannian manifold. Our first main result, Theorem \ref{Thomsp}, states that the quotient of the ambient manifold by the cut locus of the submanifold results in the Thom space of the normal bundle. As a consequence we obtain Theorem \ref{homsph} which says that a homology $k$-sphere inside a manifold homeomorphic to $S^d$ has cut locus weakly homotopy equivalent to $S^{d-k-1}$, provided $d-k\geq 3, k>0$ and $H_{d-1}(\cutn)$ is torsion-free. Theorem \ref{cutlocus-surface} is another consequence about analytic surfaces. In \S \ref{Sec: Morse-Bott} we prove (cf. Theorem \ref{thm: Theorem 2}) that the cut locus of a submanifold is closed, essentially following Wolter's arguments \cite{Wol79}. This leads us to the other main result, Theorem \ref{thm: Morse-Bott}, which proves that the complement of the cut locus $\cutn$ deforms to $N$.

\subsection{Basic results}\label{Sec: basic results}

\hf For understanding the geometry in the neighborhood of a submanifold, it is convenient to use the Fermi coordinates, a generalization of normal coordinates. We shall briefly introduce Fermi coordinates and state some of the relevant properties. Let $N$ be an embedded submanifold of a Riemannian manifold $M$. Let $\nu $ be the normal bundle of $N\subseteq M$, i.e., 
\begin{displaymath}
\nu \defeq \curlybracket{(p,v) : p\in N, v\in (T_p N)^\perp}.
\end{displaymath}
In fact, $\nu$ is a subbundle of the restriction of $TM$ to $N$. We define the \emph{exponential map of the normal bundle } as follows:
\begin{equation}\label{norexp}
\exp_\nu:\nu\to M,\,\,\exp_\nu(p,v) \defeq \exp_p(v),~~ (p,v) \in \nu.
\end{equation}
We may write $\textup{exp}_\nu(v)$ in short and call this the \textit{normal exponential} map.\\
\hf Now we will list some lemmas; for proofs see \cite[\S 2.1, \S 2.3]{Gr04}. 
\begin{lmm}
Let $N$ be a topologically embedded submanifold of a Riemannian manifold $M$. Then the normal exponential map $\exp_\nu:\nu\to M$, maps a neighborhood of $N$ in $\nu$ diffeomorphically onto a neighborhood of $N$ in $M$.
\end{lmm}
\noindent Let $\mathcal{O}_N$ denote the largest neighborhood of the zero section of $\nu$ for which $\exp_\nu$ is a diffeomorphism. We shall later be able to describe this neighbourhood in terms of a function $\mathpzc{s}$ \eqref{snu}. To define a system of Fermi coordinates, we need an arbitrary system of coordinates $(y_1,\cdots,y_k)$ defined in a neighborhood  $\mathcal{U}\subset N$ of $p\in P$ together with orthogonal sections $E_{k+1},\cdots,E_n$ of the restriction of $\nu$ to $~\mathcal{U}$.
\begin{defn}[Fermi Coordinates]
The Fermi coordinates $(x_1,\cdots,x_n)$ of $N\subset M$ centered at $p$ (relative to a given coordinate system $(y_1,\cdots,y_k)$ on $N$ and orthogonal sections $E_{k+1},\cdots, E_n$ of $\nu$) are defined by 
\begin{align*}
    x_l\bigg(\exp_\nu\Big(\sum_{j={k+1}}^n t_jE_j(p')\Big)\bigg) = y_l(p'),~~~ l = 1,\cdots,k\\
    x_i\bigg(\exp_\nu \Big(\sum_{j=k+1}^n t_jE_j(p')\Big)\bigg) = t_i,~~~ i = k+1,\cdots,n
\end{align*}
for $p'\in \mathcal{U}$ provided the numbers $t_{k+1},\cdots,t_n$ are small enough so that $t_{k+1}E_{k+1}(p')+\cdots+t_n E_n(p')\in \mathcal{O}_N$.
\end{defn}
\noindent Since $\exp_\nu$ is a diffeomorphism on $\mathcal{O}_N$, $(x_1,\cdots,x_k,x_{k+1},\cdots,x_n)$ defines a coordinate system near $p$. In fact, the restrictions to $N$ of coordinate vector fields $\partial/\partial x_{k+1},\ldots, \partial/ \partial x_n$ are orthonormal. 
\begin{lmm} \label{Lemma: geodesic in fermi}
Let $\gamma$ be  a unit speed geodesic normal to $N$ with $\gamma(0) = p\in N$. If $u = \gamma'(0)$, then there is a system of Fermi coordinates $(x_1,\cdots,x_n)$ such that for small enough $t$, i.e., for $(p,tu)\in \mathcal{O}_N$, we have
\begin{displaymath}
\left.\delbydel{}{x_{k+1}}\right|_{\gamma(t)} = \gamma'(t),\,\,\,
\left.\delbydel{}{x_l}\right|_p\in T_pN,\,\,\, \left.\delbydel{}{x_i}\right|_p \in (T_p N)^\perp
\end{displaymath}
for $1\le l\le q$ and $k+1\le i\le n.$ Furthermore, for $1\le j \le n$
\begin{displaymath}
(x_j \comp \gamma)(t) = t\delta_{j (k+1)}.
\end{displaymath}
\end{lmm}
\begin{defn}
Let $(x_1,\cdots,x_n)$ be a system of Fermi coordinates for $N\subset M$. Define $\sigma(x_1,\cdots,x_n)$ to be the non-negative number satisfying
\begin{displaymath}
\sigma^2 = \sum_{i=k+1}^n x_i^2.
\end{displaymath}
\end{defn}
\noindent It is known that $\sigma$ does not depend on the choice of Fermi coordinates.
%
\begin{prpn}\label{dsq-Fermi}
Let $U$ be a neighbourhood of $N$ such that each point in $U$ admits a unique unit speed $N$-geodesic. If $p\in U$, then 
\begin{displaymath}
\sigma(p) = \dist(N,p).
\end{displaymath}
\end{prpn}
\begin{proof}
Since the expression of $\sigma$ is independent of the choice of the Fermi coordinates, we will make a special choice of the Fermi coordinates $(x_1,\cdots,x_n)$. For $p\in U$, choose the unique unit speed $N$-geodesic $\gamma$ joining $p$ to $N$. This geodesic meets $N$ orthogonally at $\gamma(0)=p'$. Choose $t_0$ such that $\gamma(t_0)= p$.
\begin{figure}[htbp]
\centering
\includegraphics[width=0.85\textwidth]{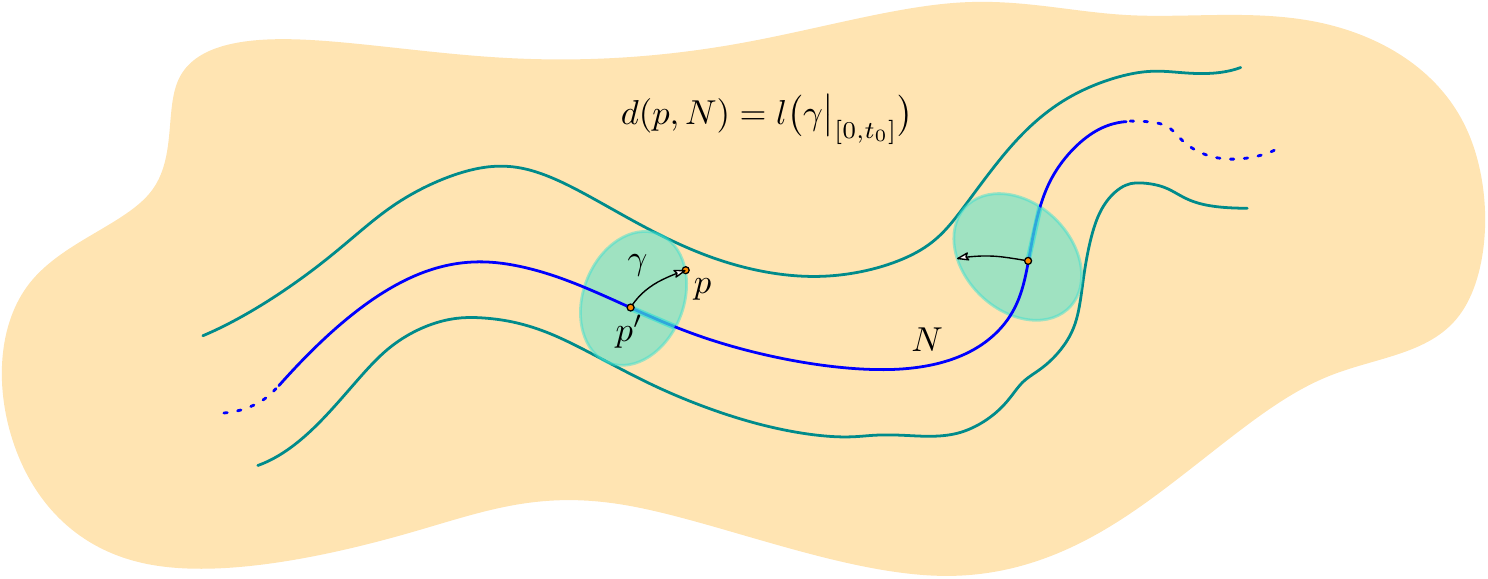}
\caption{Distance via Fermi coodinates}
\end{figure}
According to Lemma \ref{Lemma: geodesic in fermi}, there is a system of Fermi coordinates $(x_1,\cdots,x_n)$ centered at $p'$ such that $x_i(\gamma(t)) = t\delta_{i(k+1)}$. The sequence of equalities 
\begin{displaymath}
\sigma(p) = x_{k+1}(\gamma(t_0)) = t_0 = \dist(p,N)
\end{displaymath}
complete the proof.
\end{proof}
\begin{cor}\label{dsq-MB}
Consider the distance squared function with respect to a submanifold $N$ in $M$. The Hessian of the distance squared function at the critical submanifold $N$ is non-degenerate in the normal direction. 
\end{cor}
Towards the regularity of distance squared function, the following observation will be useful. It is a routine generalization of \cite[Lemma~1]{Wol79}.
\begin{lmm} \label{Lmm: singdsq}
Let $M$ be a connected, complete Riemannian manifold and $N$ be an embedded submanifold of $M$. Suppose two $N$-geodesics exist joining $N$ to $q\in M$. Then $d^2(N,\cdot):M\to \rbb$ has no directional derivative at $q$ for vectors in direction of those two $N$-geodesics. 
\end{lmm}
\begin{proof}
Let us assume that all the geodesics are arc-length parametrized. Let $\gamma_i:[0,\hat{t}]\to M,~~i=1,2$,  be two distinct geodesics with $\gamma_1(0), \gamma_2(0)\in N$ and $\gamma_1(l)=q=\gamma_2(l)$, where $l=d(N,q)$ and $0<l<\hat{t}.$ Let us suppose that the two geodesics start at $p_1$ and $p_2$ and so $d(p_1,q)=l=d(p_2,q)$. Note that the directional derivative of $d^2$ at $q$ in the direction of $\gamma_i'(q)$ from the left is given by
\begin{align*}
(d^2)'_-(q) & := \lim_{\varepsilon\to 0^+} \dfrac{(d(N,\gamma_i(l)))^2-(d(N,\gamma_i(l-\varepsilon)))^2}{\varepsilon}\\
			& = \lim_{\varepsilon\to 0^+} \dfrac{l^2-(l-\varepsilon)^2}{\varepsilon}\\
			& = 2l.
\end{align*}
Next, we claim that the derivative of the same function from the right is strictly bounded above by $2l$. Let $\omega\in(0,\pi]$ be the angle between the two geodesics $\gamma_1$ and $\gamma_2$ at $q$.  Define the function,
\begin{equation*}
u(\tau)\defeq d(N,\gamma_1(l-\varepsilon))+d(\gamma_1(l-\varepsilon),\gamma_2(\tau+l)).
\end{equation*} 
\begin{figure}[h]
\centering
\includegraphics[scale=0.8]{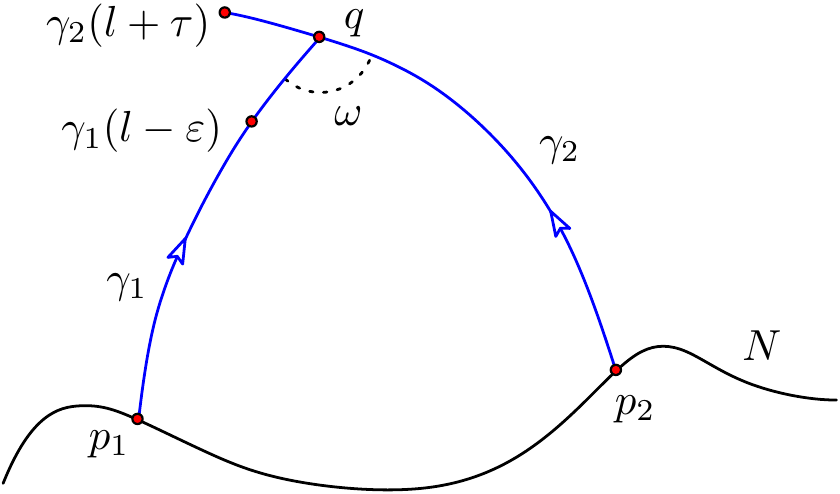}
\caption{When two $N$-geodesics meet}
\label{fig: notDifferentiable}
\end{figure}
By triangle inequality, we observe that
\begin{equation*}
f(\tau)\defeq (u(\tau))^2 \ge d^2(p_1,\gamma_2(\tau+l)) \ge d^2(N,\gamma_2(\tau+l)),
\end{equation*}
and equality holds at $\tau=0$ and $(u(0))^2=d^2(N,q)=l^2$. Thus, in order to prove the claim, it suffices to show that the derivative of $f$ from right, at $\tau=0$, is bounded below by $2l$. We need to invoke a version of the cosine law for small geodesic triangles. Although this may be well-known to experts, we will use the version that appears in Sharafutdinov's work \cite{Sha07} (see \cite[Lemma 2.4]{DaDe18} for a detailed proof). In our case, this means that 
\bgd
d^2(\gamma_1(l-\varepsilon),\gamma_2(\tau+l))=\ep^2+\tau^2+2\ep\tau \cos \omega +K(\tau)\ep^2\tau^2
\edd
where $|K(\tau)|$ is bounded and the side lengths are sufficiently small. Note that we are considering geodesic triangles with two vertices constant and the varying vertex being $\gamma_2(l+\tau)$. It follows from taking a square root and then expanding in powers of $\tau$ that
\bgd
d(\gamma_1(l-\varepsilon),\gamma_2(\tau+l))=\sqrt{\varepsilon^2+\tau^2+2\varepsilon\tau\cos\omega}~(1+O(\tau^2)).
\edd
It follows that 
		\begin{displaymath}
			u(\tau)=l-\ep+\sqrt{\varepsilon^2+\tau^2+2\varepsilon\tau\cos\omega}~(1+O(\tau^2)).
		\end{displaymath} 
		Therefore, $u'_+(0)=\cos \omega=d'_+(\gamma_1(l-\varepsilon),\gamma_2(l))$. Observe that 
		\begin{align*}
			 f'_+(\tau)\Big|_{\tau=0} & = 2d(N,\gamma_1(l-\varepsilon))d'_+(\gamma_1(l-\varepsilon),\gamma_2(l)) +2d(\gamma_1(l-\varepsilon),\gamma_2(l))d'_+(\gamma_1(l-
			\varepsilon),\gamma_2(l))\\
			 & = 2d(N,\gamma_1(l-\varepsilon))\cos \omega+2d(\gamma_1(l-\varepsilon),\gamma_2(l))\cos\omega\\
			 & =2d(N,\gamma_1(l))\cos\omega<2l.
		\end{align*}
Thus, we have proved the claim and subsequently the result.
\end{proof}
The above lemma prompts us to define the following set, the notation being consistent with Wolter's paper \cite{Wol79}.
\begin{defn}
Let $N$ be a subset of a Riemannian manifold $M$. The set $\sen$, called the \textit{separating set}\footnote{We could not find any name for this set in the literature. This terminology is our own although this nomenclature is rarely used in the paper.}, consists of all points $q\in M$ such that at least two distance minimal geodesics from $N$ to $q$ exist.
\end{defn}	
\noindent If $q\in \sen$ but $q\not\in \cutn$, then we have figure \ref{fig: notDifferentiable}, i.e., $\gamma_1$ is an $N$-geodesic beyond $q$ while $\gamma_2$ is another $N$-geodesic for $q$. The triangle inequality applied to $\gamma_1(0)$, $q=\gamma_1(l)$ and $\gamma_2(l+\tau)$ implies that 
\bgd
d(\gamma_2(l+\tau),N)< l+\tau
\edd
while for $\tau$ small enough $d(\gamma_2(l+\tau),N)=l+\tau$ as $\gamma_2$ is an $N$-geodesic beyond $q$. This contradiction establishes the well-known fact $\sen\subseteq \cutn$. In quite a few examples, these two sets are equal. In the case of $M=\mathbb{S}^n$ with $N=\{p\}$, the set $\sen$ consists of $-p$. There is an infinite family of minimal geodesics joining $p$ to $-p$. An appropriate choice of a pair of such minimal geodesics would create a loop, which is permissible in the definition of $\sen$.\\
\hf Regarding the question of cut loci being triangulable, we recall the result of Buchner \cite{Buc77} that the cut locus (of a point) of a real analytic Riemannian manifold (of dimension $d$) is a simplicial complex of dimension at most $d-1$. It follows, without much changes, that the result holds for cut locus of submanifolds as well. Hence, we attribute the following result to Buchner. 
\begin{thm}[Buchner 1977]\label{Buchner}
Let $N$ be an analytic submanifold of a real analytic manifold $M$. If $M$ is of dimension $d$, then the cut locus $\cutn$ is a simplicial complex of dimension at most $d-1$.
\end{thm}
\noindent The obvious modifications to the proof by Buchner are the following:\\
\hf (i) Choose $\ep$ to be such that there is a unique geodesic from $p$ to $q$ if $d(p,q)<\ep$ and if $d(N,q)<\ep$, then there is a unique $N$-geodesic to $q$;\\
\hf (ii) Consider the set $\Omega_N(t_0,t_1,\ldots,t_k)$, the space of piecewise broken geodesics starting at $N$, and define $\Omega_N(t_0,t_1,\ldots,t_k)^s$ analogously;\\
\hf (iii) The map
\bgd
\Omega_N(t_0,t_1,\ldots,t_k)^s\to N\times M\times \cdots\times M,\,\,\omega\mapsto (\omega(t_0),\omega(t_1),\ldots,\omega(t_k))
\edd
determines an analytic structure on $\Omega_N(t_0,t_1,\ldots,t_k)^s$.\\
The remainder of the proof works essentially verbatim.
\begin{rem}
As we have seen in example \ref{join}, the dimension of the cut locus of a $k$-dimensional submanifold is $d-k-1$. However, generically, we may not expect this to be true. In fact, for real analytic knots (except the unknot) in $\mathbb{S}^3$, it is always the case that the cut locus cannot be homotopic to a (connected) $1$-dimensional simplicial complex (cf. Example \ref{codim2}). 
\end{rem}


\subsection{Thom space via cut locus}\label{Sec: Thom}

\hf Let $(M,g)$ be a complete Riemannian manifold with distance function $d$. The exponential map at $p$
\bgd
\textup{exp}_p:T_p M\to M
\edd
is defined on the tangent space. Moreover, there exists minimal geodesic joining any two points in $M$. However, not all geodesics are distance realizing. Given $v\in T_p M$ with $\|v\|=1$, let $\gamma_v$ be the geodesic initialized at $p$ with velocity $v$. Let $S(TM)$ denote the unit tangent bundle and let $ [0,\infty]$ be the one point compactification of $ [0,\infty)$. Define 
\bgd
s:S(TM)\to [0,\infty],\,\,s(v):=\sup\{t\in [0,\infty)\,|\,\gamma_v |_{[0,t]}\,\,\textup{is minimal}\}.
\edd
\begin{defn}\label{cutlocus2}[Cut Locus]
Let $M$ be a complete, connected Riemannian manifold. If $s(v)<\infty$ for some $v\in S(T_pM)$, then $\textup{exp}_p(s(v)v)$ is called a \textit{cut point}. The collection of cut points is defined to be the cut locus of $p$.
\end{defn}
\noindent \hf As geodesics are locally distance realizing, $s(v)> 0$ for any $v\in S(TM)$. The following result \cite[Proposition 4.1]{Sak96} will be important for the underlying ideas in its proof.
\begin{prpn}\label{scts}
The  map $s:S(TM)\to [0,\infty], u\mapsto s(u)$ is continuous.
\end{prpn}
\noindent The proof relies on a characterization of $s(v)$ provided $s(v)<\infty$. A positive real number $T$ is $s(v)$ if and only if $\gamma_v:[0,T]$ is minimal and at least one of the following holds:\\
\hf (i) $\gamma_v(T)$ is the first conjugate point of $p$ along $\gamma_v$,\\
\hf (ii) there exists $u\in S(T_pM), u\neq v$ such that $\gamma_u(T)=\gamma_v(T)$.\\
Recall that if $\gamma:[0,a]\to M$ is a geodesic, then $q=\gamma(t_0)$ is conjugate to $p=\gamma(0)$ along $\gamma$ if $\exp_p$ is singular at $t_0\dot{\gamma}(0)$, i.e., $(D\exp_p)(t_0\dot{\gamma}(0))$ is not of full rank. 
\begin{rem}
If $M$ is compact, then it has bounded diameter, which implies that $s(v)<\infty$ for any $v\in S(TM)$. The converse is also true: if $M$ is complete and connected with $s(v)<\infty$ for any $v\in S(TM)$, then $M$ has bounded diameter, whence it is compact. 
\end{rem}
\noindent \hf We shall be concerned with closed Riemannian manifolds in what follows. Let $N$ be an embedded submanifold inside a 
closed, i.e., compact without boundary, manifold $M$. Let $\nu$ denote the normal bundle of $N$ in $M$ with $D(\nu)$ denoting the unit disk bundle. In the context of $S(\nu)$, the unit normal bundle and the cut locus of $N$, distance minimal geodesics or $N$-geodesics are relevant (see Definitions \ref{distmin} and \ref{cutlocus1}). We want to consider 
\begin{equation}\label{snu}
\mathpzc{s}:S(\nu)\to [0,\infty),\,\,\mathpzc{s}(v):=\sup\{t\in [0,\infty)\,|\,\gamma_v |_{[0,t]}\,\,\textup{is an $N$-geodesic}\}.
\end{equation}
Notice that $0<\mathpzc{s}(v)\leq s(v)$ for any $v\in S(\nu)$. In the special case when $N=\{p\}$, $\mathpzc{s}$ is simply the restriction of $s$ to $T_p M$. Analogous to Proposition \ref{scts}, we have the following result.
\begin{prpn}\label{snucts}
The map $\mathpzc{s}:S(\nu)\to [0,\infty)$, as defined in \eqref{snu}, is continuous.
\end{prpn}
\noindent As expected, the proof of Proposition \ref{snucts} relies on a characterization of $\mathpzc{s}(v)$ similar to that of $s(v)$ (refer to Lemma \ref{chsnu} and \cite[Exercise 23, p. 241]{BiCr64}).\\
\hf Let us postpone the proofs (see Appendix \ref{Sec: cts-s}) and proceed with some immediate applications.
\begin{defn}[Rescaled Exponential]
The \textit{rescaled exponential} or $\mathpzc{s}$-exponential map is defined to be 
\bgd
\widetilde{\textup{exp}}:D(\nu)\to M,\,\,(p,v)\mapsto \left\{\begin{array}{rl}
\textup{exp}_p(\mathpzc{s}(\hat{v})v) & \textup{if $v=\|v\|\hat{v}\neq 0$}\\
p & \textup{if $v=0$.}
\end{array}\right.
\edd
\end{defn}
\noindent We are now ready to prove the main result of this section.
\begin{thm}\label{Thomsp}
Let $N$ be an embedded submanifold inside a closed, connected Riemannian manifold $M$. If $\nu$ denotes the normal bundle of $N$ in $M$, then there is a homeomorphism 
\bgd
\widetilde{\textup{exp}}:D(\nu)/S(\nu)\stackrel{\cong}{\longrightarrow} M/\cutn.
\edd
\end{thm}
\begin{proof}
It follows from Proposition \ref{snucts} that the rescaled exponential is continuous. Moreover, $\widetilde{\textup{exp}}$ is surjective and $\widetilde{\textup{exp}}(S(\nu))=\cutn$. If there exists $(p,v)\neq (q,w)\in D(\nu)$ such that 
\bgd
\widetilde{\textup{exp}}(p,v)=\widetilde{\textup{exp}}(q,w)=p',
\edd
then $d(p',N)$ can be computed in two ways to obtain 
\bgd
d(p',N)=\mathpzc{s}(\hat{v})\|v\|=\mathpzc{s}(\hat{w})\|w\|.
\edd
Thus, $T=d(p',N)$ is a number such that $\gamma_v:[0,T]$ is an $N$-geodesic and $\gamma_v(T)=\gamma_w(T)=p'$. By Lemma \ref{chsnu}, we conclude that $T=\mathpzc{s}(\hat{v})=\mathpzc{s}(\hat{w})$, whence $\|v\|=\|w\|=1$. Therefore, $\widetilde{\textup{exp}}$ is injective on the interior of $D(\nu)$. \\
\hf As $\cutn$ is closed and $M$ is a compact metric space, the quotient space $M/\cutn$ is Hausdorff. As the quotient $D(\nu)/S(\nu)$ is compact, standard topological arguments imply the map induced by the rescaled exponential is a homeomorphism.
\end{proof}
\noindent\hf Recall that the \textit{Thom space} $\textup{Th}(E)$ of a real vector bundle $E\to B$ of rank $k$ is $D(E)/S(E)$, where it is understood that we have chosen a Euclidean metric on $E$. If $B$ is compact, then the Thom space $\textup{Th}(E)$ is the one-point compactification of $E$. In general, we compactify the fibres and then collapse the section at infinity to a point to obtain $\textup{Th}(E)$. Thus, Thom spaces obtained via two different metrics are homeomorphic. We will now revisit a basic property of Thom space via its connection to the cut locus. It can be seen that 
\begin{equation}\label{cutprod}
\mathrm{Cu}(N_1\times N_2)=(\mathrm{Cu}(N_1)\times M_2)\cup (M_1\times \mathrm{Cu}(N_2))
\end{equation}
for an embedding $N_1\times N_2$ inside $M_1\times M_2$. If $\nu_j$ is the normal bundle of $N_j$ inside $M_j$, then Theorem \ref{Thomsp} along with \eqref{cutprod} implies that
\begin{equation*}\label{Thomsmash}
\textup{Th}(\nu_1\oplus \nu_2)\cong \frac{M_1\times M_2}{(M_1\times \mathrm{Cu}(N_2))\cup (\mathrm{Cu}(N_1)\times M_2)}\cong \frac{M_1/\mathrm{Cu}(N_1)\times M_2/\mathrm{Cu}(N_2)}{M_1/\mathrm{Cu}(N_1)\vee M_2/\mathrm{Cu}(N_2)}\cong \textup{Th}(\nu_1)\wedge \textup{Th}(\nu_2).
\end{equation*}
Let $N=N_1\sqcup N_2$ be a disjoint union of connected manifolds of the same dimension. If $N\hookrightarrow M$, then let $\nu_j$ denote the normal bundle of $N_j$ in $M$. If $\nu$ is the normal bundle of $N$ in $M$, then 
\begin{equation}\label{Thomwedge}
\textup{Th}(\nu)\cong \textup{Th}(\nu_1)\vee \textup{Th}(\nu_2).
\end{equation}
This implies that
\bgd
M/\cutn \cong M/\mathrm{Cu}(N_1)\vee M/\mathrm{Cu}(N_2).
\edd
\begin{eg}
Consider the two circles
\bgd
N_1=\{(\cos t,\sin t,0,0)\,|\,t\in\R\},\,\,N_2=\{(0,0,\cos t,\sin t)\,|\,t\in\R\}
\edd
in $\mathbb{S}^3$. The link $N:=N_1\sqcup N_2$ has linking number $1$. It can be checked that 
\bgd
\cutn=\{\textstyle{\frac{1}{\sqrt{2}}}(\cos s,\sin s, \cos t,\sin t)\,|\,s,t\in\R\}
\edd
is a torus. Note that $\mathrm{Cu}(N_1)=N_2$ and vice-versa as well as 
\bgd
\mathbb{S}^3/\mathrm{Cu}(N_j)\cong (S^1\times S^2)/(S^1\times \infty)
\edd
where $S^1\times S^2$ is the fibrewise compatification of the normal bundle of $N_j$. We conclude that 
\bgd
\mathbb{S}^3/\cutn\cong \Big(\frac{S^1\times S^2}{S^1\times \infty}\Big)\vee \Big(\frac{S^1\times S^2}{S^1\times \infty}\Big)
\edd
\end{eg}
\hf There are some topological similarities between $\cutn$ and $M-N$. 
\begin{lmm}\label{defretM-N}
The cut locus $\cutn$ is a strong deformation retract of $M-N$. In particular, $(M,\cutn)$ is a good pair and the number of path components of $\cutn$ equals that of $M-N$.
\end{lmm}
\begin{proof}
Consider the map $H:(M-N)\times [0,1]\to M-N$ defined via the normal exponential map 
	\begin{equation*}
		H(q,t) =\left\{\begin{array}{rl}
			\exp_\nu\left[\left\{t \cdot \mathpzc{s}\paran{\frac{\exp_\nu^{-1}(q) }{\norm{\exp_\nu^{-1}(q)}}}+(1-t)\norm{\exp_\nu^{-1}(q)}\right\}\frac{\exp_\nu^{-1}(q)}{\norm{\exp_\nu^{-1}(q)}}\right]  & \text{ if } q\in M - \paran{\cutn\cup N} \\
			q & \text{ if } q\in \cutn.
			\end{array}\right.
	\end{equation*}
If $q\in M - (\cutn\cup N)$, then let $\gamma$ be the unique $N$-geodesic joining $N$ to $q$. The path $H(q,t)$ is the image of this geodesic from $q$ to the first cut point along $\gamma$. The continuity of $\mathpzc{s}$ implies that $H$ is continuous. It also satisfies $H(q,0)=q$ and $H(q,1)\in\cutn$. The claims about good pair and path components are clear.
\end{proof}
\begin{cor}\label{htpy-cut-locus}
If two embeddings $f,g:N\to M$ are ambient isotopic, then $\mathrm{Cu}(f(N))$ and $\mathrm{Cu}(g(N))$ are homotopy equivalent.
\end{cor}
\begin{proof}
The hypothesis implies that there is a diffeomorphism $\varphi: M\to M$ such that $\varphi(f(N))=g(N)$. Thus, $M-\mathrm{Cu}(f(N))$ is homeomorphic to $M-\mathrm{Cu}(g(N))$ and the claim follows from the lemma above. Note that in the smooth category, the notion of isotopic and ambient isotopic are equivalent (refer to \S 8.1 of \cite{Hir94}). Thus, the same conclusion holds if we assume that the embeddings are isotopic.
\end{proof}
\begin{rem}
Without the assumption of $M$ being closed, the above result fails to be true. One may consider $M=S^1\times \R$ with the natural product metric and $N=S^1$. In fact, the universal cover of $M$ is $\R\times \R$ while that of $N$ is $\R$. If we choose a periodic curve in $\R^2$ which is isotopic to the $x$-axis and has non-empty cut locus in $\R^2$, then we may pass via the covering map to obtain an embedding $g$ of $N$ isotopic to the embedding $f$ identifying $N$ with $S^1\times \{0\}$. For this pair, $\mathrm{Cu}(f(N))=\varnothing$ while $\mathrm{Cu}(g(N))\neq \varnothing$.
\end{rem}
Several other identifications between topological invariants can be explored. For instance, if $\iota:N^k\hookrightarrow M^d$ is as before such that $M-N$ is path connected, then 
\begin{equation}\label{isopi}
\iota_\ast:\pi_j(\cutn)\stackrel{\cong}{\longrightarrow}\pi_j(M)
\end{equation}
if $0\leq j\leq d-k-2$ while $\iota_\ast$ is a surjection for $j=d-k-1$. The proof of this relies on a general position argument, i.e., being able to find a homotopy of the sphere that avoids $N$, followed by Lemma \ref{defretM-N}. Surjectivity of $\iota_\ast$ if $j\leq d-k-1$ is imposed by the requirement that a sphere $S^j$ in general position must not intersect $N^k$. Injectivity of $\iota$ for $j\leq d-k-2$ is imposed by the condition that a homotopy $S^j\times [0,1]$ in general position must not intersect $N^k$. This observation \eqref{isopi} generalizes \cite[Proposition 4.5 (1)]{Sak96}.\\
\hf The inclusion $i:\cutn\hookrightarrow M$ induces a long exact sequence in homology
\bgd
\cdots \to H_j(\cutn)\stackrel{i_\ast}{\longrightarrow} H_j(M)\to H_j(M,\cutn)\stackrel{\partial}{\longrightarrow} H_{j-1}(\cutn)\to \cdots
\edd
As $(M,\cutn)$ is a good pair (cf. Lemma \ref{defretM-N}), we replace the relative homology of $(M,\cutn)$ with reduced homology of $M/\cutn\cong \textup{Th}(\nu)$. This results in the following long exact sequence
\begin{equation}\label{lesThom}
\cdots \to H_j(\cutn)\stackrel{i_\ast}{\longrightarrow} H_j(M)\stackrel{q}{\longrightarrow} \widetilde{H}_j(\textup{Th}(\nu))\stackrel{\partial}{\longrightarrow} H_{j-1}(\cutn)\to \cdots
\end{equation}
If $N=\{p\}$ is a point, then $\textup{Th}(\nu)=S^d$ and \eqref{lesThom} imply isomorphisms
\bgd
i_\ast:H_j(\textup{Cu}(p))\stackrel{\cong}{\longrightarrow} H_j(M),\,\,i^\ast:H^j(M)\stackrel{\cong}{\longrightarrow} H^j(\textup{Cu}(p))
\edd
for $j\neq d,d-1$ (cf. \cite[Proposition 4.5 (2)]{Sak96}). 
\begin{rem}
The long exact sequence \eqref{lesThom} can be interpreted as the dual to the long exact sequence in cohomology of the pair $(M,N)$. If $N=N_1\sqcup \cdots\sqcup N_l$ is a disjoint union of submanifolds of dimension $k_1, \ldots, k_l$ respectively, then the Thom isomorphism implies that 
\bgd
\widetilde{H}_j(\textup{Th}(\nu))\cong \widetilde{H}_j(\textup{Th}(\nu_1))\oplus \cdots\oplus \widetilde{H}_j(\textup{Th}(\nu_l))\cong H_{j-(d-k_1)}(N_1)\oplus\cdots\oplus H_{j-(d-k_l)}(N_l),
\edd
where $\nu_j$ is the normal bundle of $N_j$. Applying Poincar\'{e} duality to each $N_j$, we obtain isomorphisms
\bgd
\widetilde{H}_j(\textup{Th}(\nu))\cong \oplus_{i=1}^l H^{d-j}(N_i)= H^{d-j}(N).
\edd
Poincar\'{e}-Lefschetz duality applied to the pair $(M,N)$ provides isomorphisms
\begin{equation}\label{MNcutn}
\check{H}^j(M,N)\cong H_{d-j}(M-N).
\end{equation}
As $M$ and $N$ are triangulable, $\check{\textup{C}}$ech cohomology may be replaced by singular cohomology. Since $M-N$ deforms to $\cutn$ by Lemma \ref{defretM-N}, we have isomorphisms
\begin{equation}\label{MNcutn2}
H^j(M,N)\cong H_{d-j}(\cutn).
\end{equation}
Combining all these isomorphisms, we obtain the long exact sequence in cohomology for $(M,N)$ from \eqref{lesThom}. 
\end{rem}
\begin{lmm}\label{homcutn}
Let $N$ be a closed submanifold of $M$ with $l$ components. If $M$ has dimension $d$, then $H_{d-1}(\cutn)$ is free abelian of rank $l-1$ and $H_{d-j}(\cutn)\cong H^j(M)$ if $j-2\geq k$, where $k$ is the maximum of the dimension of the components of $N$.
\end{lmm}
\begin{proof}
It follows from \eqref{MNcutn} that 
\bgd
H_{d-1}(\cutn)\cong H^1(M,N).
\edd
Consider the long exact sequence associated to the pair $(M,N)$ 
\bgd
0\to H^0(M,N)\to H^0(M)\stackrel{i^\ast}{\rightarrow} H^0(N)\to H^1(M,N)\to H^1(M)\to H^1(N)\to \cdots
\edd
If $N$ has $l$ components, i.e., $N=N_1\sqcup \cdots\sqcup N_l$ where $N_j$ has dimension $k_j$, then $H^1(M,N)$ is torsion-free. This follows from the fact that $i^\ast(1)=(1,\ldots,1)$ and $H^1(M)$ is free abelian. In particular, if $H^1(M)=0$, then $H_{d-1}(\cutn)\cong \mathbb{Z}^{l-1}$. \\
\hf The long exact sequence for the pair $(M,N)$ imply that there are isomorphisms
\begin{equation}\label{cutnM}
H_{d-j}(\cutn)\cong H^j(M,N)\stackrel{\cong}{\longrightarrow} H^j(M)
\end{equation}
if $j\geq k+2$, where $k=\max\{k_1,\ldots,k_l\}$. 
\end{proof}
\begin{rem}
Cut locus can be very hard to compute. For a general space, we have the notion of topological dimension. This notion concides with the usual notion if the space is triangulable. However, Barratt and Milnor \cite{BaMi62} proved that the singular homology of a space may be non-zero beyond its topological dimension. $\check{C}$ech (co)homology is better equipped to detect topological dimension and is the reason why one may prefer it over singular homology due to the generic fractal like nature of cut loci (see the remarks following Theorem C in \S \ref{intro}). Although the topological dimension of $\cutn$ is at most $d-1$, it is not apparent that $H_{d-1}(\cutn)$ is a free abelian group. 
\end{rem}
\noindent \hf There are several applications of this discussion.
\begin{thm}\label{homsph}
Let $N$ be a smooth homology $k$-sphere embedded in a Riemannian manifold homeomorphic to $S^d$. If $d\geq k+3$, then the cut locus $\cutn$ is homotopy equivalent to $S^{d-k-1}$.
\end{thm}
\begin{proof}
As $N$ has codimension at least $3$, its complement is path-connected. It follows from \eqref{isopi} and Lemma \ref{defretM-N} that $M-N$ is $(d-k-2)$-connected. In particular, $M-N$ is simply-connected and by Hurewicz isomorphism, $H_j(M-N)=0$ if $j\leq d-k-2$.  Note that $H_d(M-N)=0$ as $M-N$ is a non-compact manifold of dimension $d$. \\
\hf If $k>0$, then by Lemma \ref{homcutn}, $H_{d-1}(M-N)=0$. Moreover, by Poincar\'{e}-Lefschetz duality \eqref{MNcutn}, we infer that 
the only non-zero higher homology of $M-N$ is $H_{d-k-1}(M-N)\cong\mathbb{Z}$. By Hurewicz Theorem there is an isomorphism, $\pi_{d-k-1}(M-N)\cong\mathbb{Z}$. Let 
\bgd
\alpha:S^{d-k-1}\to M - N
\edd
be a generator. The map $\alpha_\ast$ induces an isomorphism on all homology groups between two simply-connected CW complexes. It follows from Whitehead's Theorem that $\alpha$ is a homotopy equivalence. Using Lemma \ref{defretM-N}, we obtain our homotopy equivalence $H_1\comp \alpha:S^{d-k-1}\to \cutn$.\\
\hf If $k=0$, then by Lemma \ref{homcutn}, $H_{d-1}(M-N)\cong \mathbb{Z}$. Arguments similar to the $k>0$ case now applies to obtain a homotopy equivalence with $S^{d-1}$.
\end{proof}
The above result is foreshadowed by example \ref{join} where we showed that the cut locus of $N=\mathbb{S}^k_i$ inside $M=\mathbb{S}^d$ is $\mathbb{S}^{d-k-1}_l$. It also differs from Poincar\'{e}-Lefschetz duality in that we are able to detect the exact homotopy type of the cut locus. In fact, when $M$ and $N$ are real analytic and the embedding is also real analytic, then by Theorem \ref{Buchner} we infer that $\cutn$ is a simplicial complex of dimension at most $d-1$. Towards this direction, Theorem \ref{homsph} can be pushed further.
\begin{prpn}\label{homsph2}
Let $N$ be a real analytic homology $k$-sphere embedded in a real analytic homology $d$-sphere $M$. If $d\geq k+3$, then the cut locus $\cutn$ is a simplicial complex of dimension at most $(d-1)$, having the homology of $(d-k-1)$-sphere with fundamental group isomorphic to that of $M$.
\end{prpn}
\noindent The proof of this is a combination of ideas used in the proof of Theorem \ref{homsph}. The homotopy type cannot be deduced here due to the presence of a non-trivial fundamental group. An intriguing example can be obtained by combining Proposition \ref{homsph2} and Poincar\'{e} homology sphere.
\begin{eg}[Cut locus of $0$-sphere in Poincar\'{e} sphere]\label{eg: Poincare}
Let $\tilde{I}$ be the binary icosahedral group. It is a double cover of $I$, the icosahedral group, and can be realized a subgroup of $SU(2)$. It is known that $H_1(\tilde{I};\mathbb{Z})=H_1(\tilde{I};\mathbb{Z})=0$, i.e., it is perfect and the second homology of the classifying space $B\tilde{I}$ is zero. A presentation of $\tilde{I}$ is given by
\bgd
\tilde{I}=\langle s,t\,|\,(st)^2=s^3=t^5\rangle.
\edd
In fact, if we construct a cell complex $X$ of dimension $2$ using the presentation above, then $X$ has one $0$-cell, two $1$-cells and two $2$-cells. The cellular chain complex, as computed from the presentation, is given by
\begin{equation*}
\xymatrix@R+1pc@C+2pc{
0\ar[r] & \mathbb{Z}^2\ar[r]^{\spmat{-1 & 2 \\ 3 & -5}} & \mathbb{Z}^2\ar[r]^-{0} & \mathbb{Z}\ar[r] & 0
}
\end{equation*}
Therefore, $H_1(X)=H_2(X)=0$ while $\pi_1(X)=\tilde{I}$. \\
\hf In contrast, consider the cut locus $C$ of the $0$-sphere in $SU(2)/\tilde{I}$, the Poincar\'{e} homology sphere. As $SU(2)$ is real analytic, so is the homology sphere. By Proposition \ref{homsph2}, $C$ is a finite, connected simplicial complex of dimension $2$ such that $\pi_1(C)\cong \tilde{I}$ and $H_\bullet(C;\mathbb{Z})\cong H_\bullet(S^2;\mathbb{Z})$. The existence of this space is interesting for the followng reason: although $X\vee S^2$ has the same topological invariants we are unable to determine whether $X\vee S^2$ is homotopy equivalent to $C$.
\end{eg}
\hf In the codimension two case, we have two results. 
\begin{thm}\label{cutlocus-surface}
Let $\Sigma$ be a closed, orientable, real analytic surface of genus $g$ and $N$ a non-empty, finite subset. Then $\cutn$ is a connected graph, homotopy equivalent to a wedge product of $|N|+2g-1$ circles.
\end{thm}
\begin{proof}
As $M-N$ is connected, Lemma \ref{defretM-N} implies that $\cutn$ is connected. It follows from Theorem \ref{Buchner} that $\cutn$ is a finite $1$-dimensional simplicial complex, i.e., a finite graph. In this case, $\textup{Th}(\nu)$ is a wedge product of $|N|$ copies of $S^2$ (cf. \eqref{Thomwedge}). We consider \eqref{lesThom} with $j=2$:
\bgd
0\stackrel{i_\ast}{\longrightarrow} \mathbb{Z}\stackrel{q}{\longrightarrow} \widetilde{H}_2(\vee_{|N|} S^2)\stackrel{\partial}{\longrightarrow} H_{1}(\cutn)\stackrel{i_\ast}{\longrightarrow} H_{1}(\Sigma)\to 0
\edd
Note that $H_{d-1}(M)$ is torsion-free, whence all the groups appearing in the long exact sequence are free abelian groups. This implies that 
\begin{equation*}\label{bettid-1}
\dim_\mathbb{Z} H_{1}(\cutn)=2g+|N|-1.
\end{equation*}
As $\cutn$ is connected finite graph, collapsing a maximal tree $T$ results in a quotient space $\cutn/T$ which is homotopic to $\cutn$ as well being a wedge product of $|N|+2g-1$ circles.
\end{proof}
\begin{rem}
Itoh and V\^{i}lcu \cite{ItVi15} proved that every finite, connected graph can be realized as the cut locus (of a point) of some surface. There remains the question of orientability of the surface. As noted in the proof of Theorem \ref{cutlocus-surface}, if the surface is orientable and $|N|=1$, then the graph has an even number of generating cycles. If $\Sigma$ is non-orientable, then $\Sigma\cong (\mathbb{RP}^2)^{\# k}$ has non-orientable genus $k$ and the oriented double cover of $\Sigma$ has genus $g=k-1$. Recall that $H_1(\Sigma)\cong \mathbb{Z}^{k-1}\oplus\mathbb{Z}_2$ and $H_2(\Sigma)=0$. Looking at \eqref{lesThom} with $j=2$ we obtain
\bgd
0\to \mathbb{Z}\to H_1(\cu)\to \mathbb{Z}^{k-1}\oplus\mathbb{Z}_2\to 0.
\edd
Thus, $H_1(\cu)\cong\mathbb{Z}^k$ as homology of graphs are free abelian. Let $B_\ep(\cu)$ denote the $\ep$-neighbourhood of $\cu$ in $\Sigma$. For $\ep$ sufficiently small, this is a surface such that $\overline{B_\ep(\cu)}$ has one boundary component. The compact surface $B_\ep(\cu)$ is reminiscent of ribbon graphs. The surface $\Sigma$ can be obtained as the connect sum of a disk centered at $p$ and the closure of $B_\ep(\cu)$. Therefore, non-orientability of $\Sigma$ is equivalent to non-orientability of $B_\ep(\cu)$. A similar observation appears in the unpublished \cite[Theorem 3.7]{ItVi11}.
\end{rem}
\begin{eg}[Homology spheres of codimension two]\label{codim2}
In continuation of Theorem \ref{homsph}, let $N\hookrightarrow S^{k+2}$ be a homology sphere of dimension $k\geq 1$. Since $N$ has codimension two, $S^{k+2}-N$ is path connected and so is $\cutn$. We are not assuming that the metric on $S^{k+2}$ is real analytic. Using \eqref{MNcutn2} and the long exact sequence in cohomology of $(M,N)$, we infer that $H_1(\cutn)\cong\mathbb{Z}$ and all higher homology groups vanish. However, the Hurewicz Theorem cannot be used here to establish that $\pi_1(\cutn)\cong\mathbb{Z}$. \\
\hf In particular cases, we may conclude that $\cutn$ is homotopic to a circle. It was proved by Plotnick \cite{Plo82} that certain homology $3$-spheres $N$, obtained by a Dehn surgery of type $\frac{1}{2a}$ on a knot, smoothly embed in $S^5$ with complement a homotopy circle. Since $M-N$ deforms to $\cutn$, it follows that there is a map $\alpha:S^1\to \cutn$ inducing isomorphisms on homotopy and homology groups.\\
\hf If $k=1$, then a homology $1$-sphere is just a knot $K$ in $S^3$. Since $S^3-K$ deforms to $\mathrm{Cu}(K)$, the fundamental group of the cut locus is the knot group. Moreover, in the case of real analytic knots in $\mathbb{S}^3$, the cut locus is a finite simplicial complex of dimension at most $2$ (cf. Theorem \ref{Buchner}). Except for the unknot, the knot group is never a free group while the fundamental group of a connected, finite graph is free. This observation establishes that $\mathrm{Cu}(K)$ is always a $2$-dimensional simplicial complex, whenever $K$ is a non-trivial (real analytic) knot in $\mathbb{S}^3$.
\end{eg}


\subsection{Morse-Bott function associated to distance function}\label{Sec: Morse-Bott}

\hf We first prove that the closure of $\sen$ is the cut locus, closely following the proof given in \cite{Wol79} for the case of a point.
\begin{thm}\label{thm: Theorem 2}
Let $\cutn$ be the cut locus of a compact submanifold $N$  of a  complete Riemannian manifold $M$. The subset $\sen$ of $\cutn$ is dense in $\cutn.$
\end{thm}
\begin{proof}
Let $q\in \cutn $ but not in \sen. Choose an $N$-geodesic $\gamma$, joining $N$ to $q$, such that any extension of $\gamma$ is not an $N$-geodesic. This geodesic $\gamma$ is unique as $q\notin \sen$. We may write $\gamma(t) = \exp_\nu(tx)$, where $\gamma(0)=p\in N$ and $\gamma'(0)=x_0\in S(\nu_p)$. It follows from the definition of $\mathpzc{s}$ that $q=\exp_\nu\paran{\mathpzc{s}(x_0)x_0}$. We need to show that every neighborhood of $q$ in $\cutn$ must intersect $\sen$. Suppose it is false. Let $\delta>0$ and consider $\Ball$, the closed ball with center $x_0$ and radius $\delta$. Define the cone
\begin{equation*}
\co\! := \curlybracket{tx:0\le t\le 1,~ x\in \Ball\cap S(\nu)}.
\end{equation*}
Since $\Ball\cap S(\nu)$ is homeomorphic to a closed $(n-1)$ ball for sufficiently small $\delta$, the cone will be homeomorphic to a closed Euclidean $n$-ball.
\begin{figure}[H] 
	\centering 
	\includegraphics[angle=270, scale=0.8]{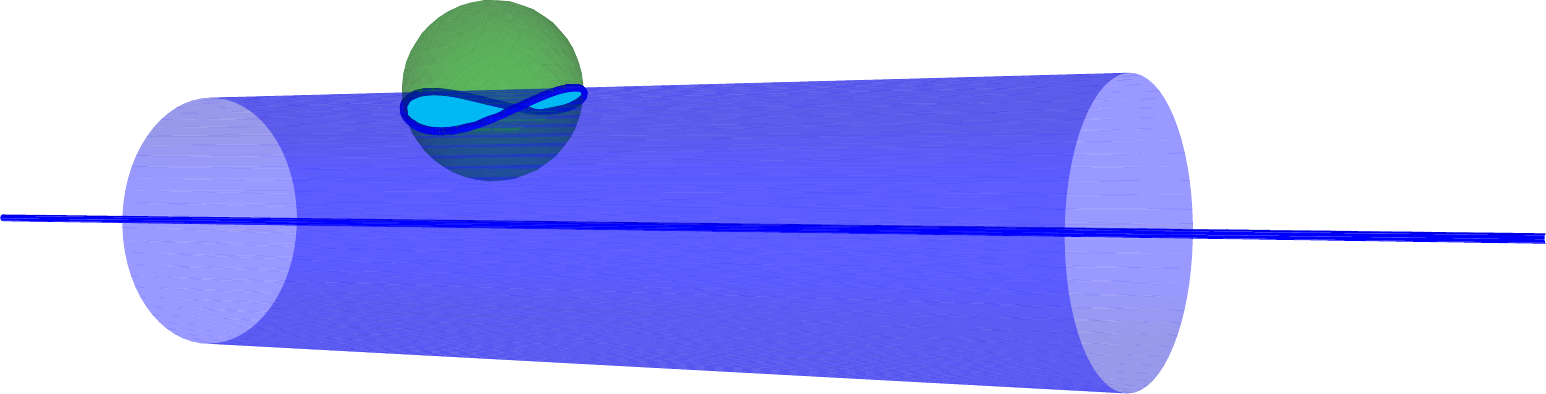} 
	\caption{\co} 
	\label{fig: cone} 
\end{figure}
Similarly, define another cone
\begin{equation*}
\costar\! := \big\{\mathpzc{s}\paran{ x/\norm{x}}x\,|\, x\in \co\!, ~x\neq 0\big\}\cup \{0\}.
\end{equation*} 
Note that $\mathpzc{s}(x_0)$ is finite. As $\mathpzc{s}$ is continuous, due to Proposition \ref{snucts}, for sufficiently small $\delta$ the term $\mathpzc{s}\paran{x/\norm{x}}$ is still finite, whence \costar is well defined. We claim that \costar is also homeomorphic to a closed Euclidean $(n-k)$-ball. Indeed,  a non-zero $x\in \co\!$ implies $x=\lambda \hat{x}$, for some $\lambda \in (0,1]$ and $\hat{x}\in \Ball\cap S(\nu)$. Since $\mathpzc{s}(\hat{x})x = \lambda \mathpzc{s}(\hat{x})\hat{x}$, it follows that \costar is the cone of the set
\bgd
\{\mathpzc{s}(\hat{x})\hat{x}\,|\,\hat{x}\in \Ball\cap S(\nu)\},
\edd
which is homeomorphic to $\Ball\cap S(\nu)$.
Now we have a dichotomy: \\
\hf (a) For a fixed small $\delta>0$, the restriction of $\exp_\nu$ to \costar is a homeomorphism to its image because it is injective, or\\
\hf (b) For any $\delta>0$, the restriction of $\exp_\nu$ to \costar is not injective. \\ 
If (b) holds, choose $v_n\neq w_n\in \mathrm{Co}^\star(x_0,\frac{1}{n})$ such that these map to $q_n$ under $\exp_\nu$. Thus,  $q_n\in \sen$ and compactness of $S(\nu)$ ensures that $q_n$ converges to $q$. If (a) holds, then let $B(q,\varepsilon)$ denote the open ball in $M$ centered at $q$ with radius $\varepsilon>0$. We claim that it intersects the complement of $\exp_\nu(\costar\!\!)$ in $M$. But it is true as $\mathpzc{s}(x_0)x_0$ lies on the boundary of \costar and hence it has a neighborhood in \costar which is homeomorphic to a closed $n$-dimensional Euclidean half plane. Since $\exp_{\nu}$ restricted to \costar was a homeomorphism, the open ball $B(q,\varepsilon)$ must intersect the points outside the image of $\exp_{\nu}(\costar\!\!)$. \\
\hf Now take $\varepsilon = \frac{1}{n}$. For each $n$, there exists $q_{n}\in B\paran{q,\frac{1}{n}}$ such that $q_{n}\notin \exp_{\nu}(\costar\!\!)$. Since $M$ is complete, for each point $q_{n}$ let $\gamma_{n}$ be a $N$-geodesics joining $p_n\in N$ to $q_{n}$. We may invoke the following result from Buseman's book \cite[Theorem 5.16, p. 24]{Bus55}. Let $\curlybracket{\gamma_{n}}$ be a sequence of rectifiable curves in a finitely compact set $X$ and the lengths $\ell(\gamma_n)$ are bounded. If the initial points $p_{n}$ of $\gamma_{n}$ forms a bounded set, then $\curlybracket{\gamma_n}$ contains a subsequence $\gamma_{n_k}$ which converges uniformly to a rectifiable curve $\tilde{\gamma}$ in $X$ and 
\begin{equation*}\label{Fact A eqn: 1} 
	\ell(\tilde{\gamma}) \le \lim\inf \ell\paran{\gamma_{n_k}}
\end{equation*}
Since $\{p_n\}$ lie in the compact set $N$, we obtain a rectifiable curve $\tilde{\gamma}$ such that 
\bgd
\ell(\tilde{\gamma})\leq  \lim\inf \ell\paran{\gamma_{n_k}}=\lim_k \ell(\gamma_{n_k})=\lim_k d(q_{n_k},N)=d(q,N).
\edd
Thus, $\tilde{\gamma}$ is actually an $N$-geodesic joining $p'=\lim_k p_{n_k}$ to $q$  and the unit tangent vectors $x_{n_k}=\gamma_{n_k}'(0)$ at $p_{n_k}$ converges to the unit tangent vector $\tilde{x}=\tilde{\gamma}'(0)$ at $p'$.  Since $x_{0}$ is an interior point of the set $\Ball\cap S(\nu)$, any sequence in $S(\nu)$ converging to $x_{0}$ must eventually lie in \co\!. According to our choice, $q_{n_k}\notin \exp_{\nu}(\costar\!\!)$ and $x_{n_k}$ all lie outside of \co\!. Hence $x_{0}\neq \tilde{x}$ and $\gamma \neq \tilde{\gamma}$. Thus, there are two distinct $N$-geodesics $\gamma$ and $\tilde{\gamma}$ joining $N$ to $q$, a contradiction to $q\notin \sen$. This completes the proof.
\end{proof}
We have seen (cf. Lemma \ref{Lmm: singdsq}) that $d^2$ is smooth away from the cut locus. It follows from Theorem \ref{thm: Theorem 2} that the cut locus is the closure of the singularity of $d^2$. The following example suggests that $d^2$ can be differentiable at points in $\cutn -\sen$ but not twice differentiable.
\begin{eg}[Cut locus of an ellipse]
We discuss the regularity of the distance squared function from an ellipse $x^2/a^2+y^2/b^2=1$ (with $a^>b>0$) in $\R^2$. For a discussion of the cut locus for ellipses inside $\mathbb{S}^2$ and ellipsoids, refer to \cite[pp. 90-91]{Heb95}. Let $(x_0,y_0)$ be a point inside the ellipse lying in the first qudarant. The point closest to $(x_0,y_0)$ and lying on the ellipse is given by 
\bgd
x=\frac{a^2x_0}{t+a^2},\,\,y=\frac{b^2y_0}{t+b^2},
\edd
where $t$ is the unique root of the quartic
\bgd
\left(\frac{ax_0}{t+a^2}\right)^2+\left(\frac{by_0}{t+b^2}\right)^2=1
\edd
in the interval $(-b^2,\infty)$. Given $(\alpha,\beta)$ with $\beta>0$, we set $P_\ep(\alpha,\beta)=(\frac{a^2-b^2}{a}+\ep\alpha,\ep\beta)$; this defines a straight line passing through $P_0(\alpha,\beta)$ in the direction of $(\alpha,\beta)$. For $\ep>0$, $P_\ep(\alpha,\beta)$ lies in the first quadrant and we denote by $t=t(\ep)$ be the unique relevant root of the quartic
\bgd
\left(\frac{a(\frac{a^2-b^2}{a}+\ep\alpha)}{t+a^2}\right)^2+\left(\frac{b\ep\beta}{t+b^2}\right)^2=1.
\edd
Simpifying this after dividing by $\ep$ and taking a limit $\ep\to 0^+$, we obtain
\bgd
\frac{2a\alpha}{a^2-b^2}=\lim_{\ep\to 0^+}\left(\Big(\frac{2}{a^2-b^2}\Big)\frac{t+b^2}{\ep}-b^2\beta^2 \frac{\ep}{(t+b^2)^2}\right).
\edd
On the other hand, the point $Q_\ep(\alpha,\beta)$ on the ellipse closest to $P_\ep(\alpha,\beta)$ is given by
\begin{figure}[H]
\centering
\includegraphics[scale=0.7]{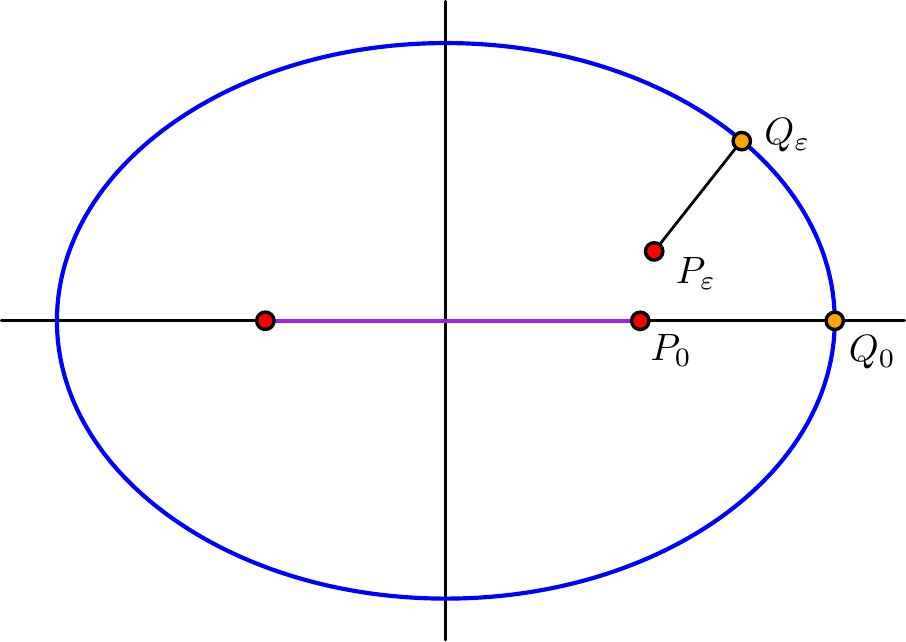}
\caption{Cut locus of an ellipse}
\end{figure}\vspace*{-0.3cm}
\bgd
x_\ep=\frac{a^2(\frac{a^2-b^2}{a}+\ep\alpha)}{t+a^2},\,\,y_\ep=\frac{b^2 \ep\beta}{t+b^2}.
\edd
It follows that
\begin{equation}
d_\ep^2(\alpha,\beta):=d^2(P_\ep,Q_\ep)=\frac{t^2}{a^2}\left(\frac{a^2-b^2+a\ep\alpha}{t+a^2}\right)^2+\frac{t^2}{b^2}\left(\frac{b\ep\beta}{t+b^2}\right)^2\label{dsqellip}
\end{equation}
Using $t(0)=-b^2$ and simplifications lead us to the following
\begin{align*}
\lim_{\ep\to 0^+}\frac{d_\ep^2-d_0^2}{\ep} & =\frac{2ab^4\alpha}{a^2(a^2-b^2)}-\lim_{\ep\to 0^+}\left(\frac{(t+b^2)(a^2b^2-a^2 t+2b^2t)}{\ep(t+a^2)^2}-\beta^2 \frac{t^2\ep}{(t+b^2)^2}\right)\\
& = \frac{2ab^4\alpha}{a^2(a^2-b^2)}-\frac{2b^2}{a^2-b^2}\lim_{\ep\to 0}\frac{t+b^2}{\ep}+\beta^2 b^4\lim_{\ep\to 0}\frac{\ep}{(t+b^2)^2}\\
& = \frac{2ab^4\alpha}{a^2(a^2-b^2)}-\frac{2ab^2\alpha}{a^2-b^2}=-\frac{2b^2\alpha}{a}.
\end{align*}
\hf On the other hand, for $\ep<0$, the point $P_\ep(\alpha,\beta)$ lies in the fourth quadrant. By symmetry, the distance between $P_\ep(\alpha,\beta)$ and $Q_\ep(\alpha,\beta)$ is the same as that between $P_{-\ep}(-\alpha,\beta)$ and $Q_{-\ep}(-\alpha,\beta)$. However, it is seen that 
\bgd
d^2(P_{-\ep}(-\alpha,\beta), Q_{-\ep}(-\alpha,\beta))=d^2_{-\ep}(-\alpha,\beta)
\edd
as defined in \eqref{dsqellip}. Therefore,
\begin{align*}
\lim_{\ep\to 0^-}\frac{d^2(P_{\ep}(\alpha,\beta), Q_{\ep}(\alpha,\beta))-d^2(P_{0}(\alpha,\beta), Q_{0}(\alpha,\beta))}{\ep} & =  \lim_{\ep\to 0^-}\frac{d^2 _{-\ep}(-\alpha,\beta)-d^2 _0(-\alpha,\beta)}{\ep}\\
& =  -\lim_{-\ep\to 0^+}\frac{d^2 _{-\ep}(-\alpha,\beta)-d^2 _0(-\alpha,\beta)}{-\ep}\\
& = -\frac{2b^2\alpha}{a},
\end{align*}
where the last equality follows from the right hand derivative of $d^2$, as computed previously. \\
\hf When $\beta=0$ we would like to compute $d^2_\ep(\alpha,0)$. If $\ep>0$, then 
\begin{equation}
\label{dsqP01}d^2_\ep(\alpha,0)=(b^2/a-\ep\alpha)^2=\frac{b^4}{a^2}-\frac{2b^2\alpha\ep}{a}+\alpha^2\ep^2.
\end{equation}
On the other hand, if $\ep<0$ is sufficiently small, then there are two points on the ellipse closest to $P_\ep(\alpha,0)=(\frac{a^2-b^2}{a}+\ep\alpha,0)$, with exactly one on the first quadrant, say $Q_\ep$. Since the segment $P_\ep Q_\ep$ must be orthogonal to the tangent to the ellipse at $Q_\ep$, we obtain the coordinates for $Q_\ep$:
\bgd
x_\ep=\frac{a^2(\frac{a^2-b^2}{a}+\ep\alpha)}{a^2-b^2},\,\,y^2_\ep=b^2\bigg(1-\frac{x_\ep^2}{a^2}\bigg),\,\,y_\ep>0.
\edd
We may compute the distance
\begin{equation}
\label{dsqP02}d^2_{\ep}(\alpha,0):=(d(P_\ep,Q_\ep))^2=\frac{b^4}{a^2}-\frac{2b^2\alpha\ep}{a}-\frac{b^2\alpha^2\ep^2}{a^2-b^2},
\end{equation}
where $\ep<0$. Combining \eqref{dsqP01} and \eqref{dsqP02} we conclude that $d^2$ is differentiable at $P_0=(\frac{a^2-b^2}{a},0)$, a point in $\cutn$ but not in $\sen$. However, comparing the quadratic part of $d^2$ in \eqref{dsqP01},\eqref{dsqP02} we conclude that $d^2$ is not twice differentiable at $P_0$. 
\end{eg}\vspace*{-0.2cm}
\begin{thm}\label{thm: Morse-Bott}
Let $N$ be a closed embedded submanifold of a complete Riemannian manifold $M$. Let $d:M\to\R$ be the distance function with respect to $N$. If $f=d^2$, then its restriction to $M-\cutn$ is a Morse-Bott function, with $N$ as the critical submanifold. Moreover, the gradient flow of $f$ deforms $M-\cutn$ to $N$.
\end{thm}
\begin{proof}
It follows from Lemma \ref{chsnu} the map $\exp_\nu^{-1}:M - \paran{\cutn \cup N}\to \nu - \{0\}$ is an (into) diffeomorphism and $\dist(N,q) = \norm{\exp^{-1}_\nu(q)}$ and hence the distance function is of class $C^{\infty}$ at $q\in M - \paran{\cutn\cup N}$. Using Fermi coordinates (cf. Proposition \ref{dsq-Fermi}), we have seen that the distance squared function is smooth around $N$ and therefore it is smooth on $M - \cutn$. By Corollary \ref{dsq-MB}, the Hessian of this function at $N$ is non-degenerate in the normal direction. It is well-known \cite[Proposition 4.8]{Sak96} that $\|\nabla d(q)\|=1$ if $d$ is differentiable at $q\in M$. Thus, for $q\in M-(\cutn\cup N)$ we have
\begin{equation}\label{graddsq}
\|\nabla f(q)\|=2d(q)\|\nabla d(q)\|=2d(q).
\end{equation}
Let $\gamma$ be the unique unit speed $N$-geodesic that joins $N$ to $q$, i.e., 
\bgd
\gamma:[0,d(q)]\to M,\,\,\gamma(0)=p,\,\gamma(d(q))=q,\,\|\gamma'\|=1.
\edd
We may write $\nabla f(q)=\lambda \gamma'(d(q))+w$, where $w$ is orthogonal to $\gamma'(d(q))$. But 
\bgd
\left\langle \nabla f\big|_q, \gamma'(d(q))\right\rangle = \frac{d}{dt}f(\gamma(d(q)+t))\Big|_{t=0}=\frac{d}{dt}(d(q)^2+2d(q)t+t^2)\Big|_{t=0}=2d(q).
\edd
Thus, $\lambda=2d(q)$ and combined with \eqref{graddsq}, we conclude that $\nabla f(q)=2d(q)\gamma'(d(q))$. Therefore, the negative gradient flow line initialized at $q\in M-\cutn$ is given by
\bgd
\eta(t)=\gamma(d(q)e^{-2t}).
\edd
These flow lines define a flow which deform $M-\cutn$ to $N$ in infinite time. 
\end{proof}
The reader may choose to revisit the example of $GL(n,\R)$ discussed in \S \ref{Sec: Example} and treat it as a concrete illustration of the Theorem above.

	
\section{Applications to Lie Groups}\label{Sec: Lie}

\hf Due to classical results of Cartan, Iwasawa and others, we know that any connected Lie group $G$ is diffeomorphic to the product of a maximally compact subgroup $K$ and an Euclidean space. In particular, $G$ deforms to $K$. For semi-simple groups, this decomposition is stronger and is attributed to Iwasawa. The Killing form on the Lie algebra $\mathfrak{g}$ is non-degenerate and negative definite for compact semi-simple Lie algebras. For such a Lie group $G$, consider the Levi-Civita connection associated to the bi-invariant metric obtained from negative of the Killing form. This connection coincides with the Cartan connection. \\
\hf We consider two examples, both of which are non-compact and non semi-simple. We prove that these Lie groups $G$ deformation retract to maximally compact subgroups $K$ via gradient flows of appropriate Morse-Bott functions. This requires a choice of a left-invariant metric which is right-$K$-invariant and a careful analysis of the geodesics associated with the metric. In particular, we provide a possibly new proof of the surjectivity of the exponential map for $U(p,q)$.

\subsection{Invertible matrices with positive determinant}\label{Sec: GLnSOn}

\hf Let $g$ be a left-invariant metric on $GL(n,\rbb)$, set of invertible matrices. Recall that a left-invariant metric $g$ on a Lie group is determined by its restriction at the identity. For $A\in GL(n,\rbb)$, consider the left multiplication map $l_A:GL(n,\rbb) \to GL(n,\rbb),~ B\mapsto AB$. This extends to a linear isomorphism from $M(n,\rbb)$ to itself. Thus, the differential $(Dl_A)_I:T_IGL(n,\rbb)\to T_AGL(n,\rbb)$ is an isomorphism and given by $l_A$ itself. For $X,Y\in T_IGL(n,\rbb)$, 
	\begin{equation*} \label{eq: left-invariant iso}
		g_I(X,Y) = g_A((Dl_A)_IX,(Dl_A)_IY)=g_A(AX,AY).
	\end{equation*}
We choose the left-invariant metric on $GL(n,\rbb)$ generated by the the Euclidean metric at $I$. Therefore,
	\begin{displaymath}
		g_{A^{-1}}(X,Y) = \innerprod{AX}{AY}_I := \trace{(AX)^T\!AY} = \trace{X^T\!A^T\!AY}.
	\end{displaymath}
Note that this metric is right-$O(n,\rbb)$-invariant. We are interested in the distance between an invertible matrix $A$ (with $\det(A)>0$) and $SO(n,\rbb)$. Since $SO(n,\rbb)$ is compact, there exists $B\in SO(n,\rbb) $ such that $d(A,B) = \dist(A,SO(n,\rbb))$. 
\begin{lmm}\label{CartanGLn}
If $D$ is a diagonal matrix with positive diagonal entries $\lambda_1,\cdots,\lambda_n$, then 
\begin{displaymath}
\dist(D,SO(n,\rbb)) = d(D,I).
\end{displaymath}
Moreover, $I$ is the unique minimizer and the associated minimal geodesic is given by $\gamma(t)=e^{t\log D}$.
\end{lmm}
\begin{proof}
Let $B\in SO(n,\rbb)$ satisfying $d(A,B) = \dist(A,SO(n,\rbb))$. Since with respect to the left-invariant metric $GL^+(n,\rbb)$ is complete, there exists a minimal geodesic $\gamma:[0,1]\to GL^+(n,\rbb)$ joining $B$ to $D$, i.e., 
\begin{displaymath}
\gamma(0) = B,~~\gamma(1) = D, ~~\text{ and } ~~ l(\gamma) = d(D,B).
\end{displaymath}
The first variational principle implies that $\gamma'(0)$ is orthogonal to $T_BSO(n,\rbb)$. It follows from \cite[\S 2.1]{MaNe16} that $\eta(t)=e^{tW}$ is a geodesic if $W$ is a symmetric matrix. Moreover, $\eta'(0)=W$ is orthogonal to $T_I SO(n,\rbb)$. As left translation is an isometry and isometry preserves geodesic, it follows that $\gamma(t) = Be^{tW}$ is a geodesic with $\gamma'(0)$ orthogonal to $T_BSO(n,\rbb)$. By the defining properties of $\gamma$,  $D=\gamma(1)=Be^W$. Since $e^W$ is symmetric positive definite, we obtain two polar decompositions of $D$, i.e., $D = ID$ and $ D = Be^W$. By the uniqueness of the polar decomposition for invertible matrices, $B=I$ and $D=e^W$.\\
\hf In order to compute $d(I,D)$, note that 
\begin{displaymath}
e^W = D = e^{\log D},
\end{displaymath}
where $\log D$ denotes the diagonal matrix with entries $\log \lambda_1,\cdots,\log \lambda_n$. As $W$ and $\log D$ are symmetric, and matrix exponential is injective on the space of symmetric matrices, we conclude that $W = \log D$. The geodesic is given by $\gamma(t)=e^{t\log D}$ and 
\begin{equation}\label{eq: distance left invariant}
\dist(D,SO(n,\rbb)) = \norm{\gamma'(0)}_I = \norm{\log D}_I = \left(\sum_{i=1}^n (\log \lambda_i)^2\right)^{\frac{1}{2}}.
\end{equation}
Thus, the distance squared function will be given by $\sum_{i=1}^n (\log \lambda_i)^2$. 
\end{proof}
\hf Now for any $A\in GL^+(n,\rbb)$ we can apply the SVD decomposition, i.e., $A = UDV^T$ with $\sqrt{A^TA} = VDV^T$ and $\log\sqrt{A^TA} = V(\log D) V^T$. Note that $U,V\in SO(n,\rbb)$ and $D$ is a diagonal matrix with positive entries. The left-invariant metric is right-invariant with respect to orthogonal matrices. Thus,
\begin{displaymath}
\dist(A,SO(n,\rbb)) = \dist(D,SO(n,\rbb)) = \norm{\log D}_I,
\end{displaymath}
where the last equality follows from the lemma (see \eqref{eq: distance left invariant}). As 
\bgd
\norm{\log D}_I = \norm{V(\log D) V^T}_I = \norm{\log\sqrt{A^TA}}_I,
\edd
It follows from the arguments of the lemma and the metric being bi-$O(n,\rbb)$-invariant that
\bgd
\gamma(t)=Ue^{t \log D}V^T
\edd
is a minimal geodesic joining $UV^T$ to $A$, realizing $\dist(A,SO(n,\rbb))$. As the minimzer $UV^T$ is unique, $\mathrm{Se}(SO(n,\rbb))$ is empty, implying that $\mathrm{Cu}(SO(n,\rbb))$ is empty as well. In fact, $UV^T=A\sqrt{A^T A}^{-1}$ and 
\begin{equation}\label{GLdefOver2}
\gamma(t)=Ue^{t \log D}V^T=UV^T Ve^{t \log D}V^T=A\sqrt{A^T A}^{-1}e^{t\log \sqrt{A^T A}}.
\end{equation}
If we compare \eqref{GLdefOver1}, the deformation of $GL(n,\rbb)$ to $O(n,\rbb)$ inside $M(n,\rbb)$, with \eqref{GLdefOver2}, then in both cases, an invertible matrix $A$ deforms to $A\sqrt{A^T A}^{-1}$. Finally, observe that the normal bundle of $SO(n,\rbb)$ is diffeomorphic to $GL^+(n,\rbb)$.

\subsection{Indefinite unitary groups}\label{Sec: Upq}
	
\hf Let $n$ be a positive integer with $n=p+q$. Consider the inner product on $\C^n$ given by
\bgd
\lan (w_1,\ldots,w_n),(z_1,\ldots,z_n)\ran = z_1\overline{w_1}+\cdots+z_p\overline{w_p}-z_{p+1}\overline{w_{p+1}}-\cdots- z_n\overline{w_n}.
\edd
This is given by the matrix $I_{p,q}$ in the following way:
\bgd
\lan \mathbf{w},\mathbf{z}\ran=\overline{\mathbf{w}}^t I_{p,q} \mathbf{z}=\left(\begin{array}{ccc}
\overline{w}_1 & \cdots & \overline{w}_n
\end{array}\right)\left(\begin{array}{cc}
I_p & 0 \\
0 & -I_q
\end{array}\right)\left(\begin{array}{c}
z_1\\
\vdots\\
z_n
\end{array}\right)
\edd
Let $U(p,q)$ denote the subgroup of $GL(n,\C)$ preserving this indefinite form, i.e., $\mathcal{A}\in \upq$ if and only if $\mathcal{A}^\ast I_{p,q}\mathcal{A}=I_{p,q}$. In particular, $\det \mathcal{A}$ is a complex number of unit length. By convention, $I_{n,0}=I_n$ and $ I_{0,n}=-I_n$, both of  which corresponds to $U(n,0)=U(n)=U(0,n)$, the unitary group. In all other cases, the inner product is indefinite.\\
\hf The group $U(1,1)$ is given by matrices of the form
\bgd
\mathcal{A}=\left(\begin{array}{cc}
\alpha & \beta\\
\lambda \overline{\beta} & \lambda \overline{\alpha}
\end{array}\right),\,\,\lambda\in S^1,\,\,|\alpha|^2-|\beta|^2=1.
\edd
More generally, we shall use 
\bgd
\mathcal{A}=\left(\begin{array}{cc}
A & B\\
C & D
\end{array}\right)
\edd
to denote an element of $\upq$. It follows from definition that $\mathcal{A}\in \upq$ if and only if
\begin{eqnarray*}
A^\ast A-C^\ast C& = & I_p\\
A^\ast B-C^\ast D& = & 0_{p\times q}\\
B^\ast B-D^\ast D& = & -I_q.
\end{eqnarray*}
Observe that if $Av=0$, then
\bgd
0=A^\ast Av=C^\ast Cv+v,
\edd
which implies that $C^\ast C$, a positive semi-definite matrix, has $-1$ as an eigenvalue unless $v=0$. Therefore, $A$ is invertible and the same argument works for $D$.
\begin{lmm}
The intersection of $U(p+q)$ with $\upq$ is $U(p)\times U(q)$. Moreover, if $\mathcal{A}\in \upq$, then $\mathcal{A}^\ast,\sqrt{\mathcal{A}^\ast\mathcal{A}}\in \upq$.
\end{lmm}
\begin{proof}
If $\mathcal{A}\in U(p)\times U(q)$, then 
\begin{eqnarray*}
A^\ast A+C^\ast C& = & I_p\\
B^\ast B+D^\ast D& = & I_q.
\end{eqnarray*}
This implies that both $B$ and $C$ are zero matrices. If $\mathcal{A}\in \upq$, then $\mathcal{A}^\ast=I_{p,q}\mathcal{A}^{-1}I_{p,q}$ and
\bgd
(\mathcal{A}^\ast\mathcal{A})^\ast I_{p,q}(\mathcal{A}^\ast\mathcal{A})=(\mathcal{A}^\ast\mathcal{A}) I_{p,q}(\mathcal{A}^\ast\mathcal{A})=I_{p,q}\mathcal{A}^{-1}I_{p,q}\mathcal{A}I_{p,q}              I_{p,q}\mathcal{A}^{-1}I_{p,q}\mathcal{A}=I_{p,q}=\mathcal{A}^\ast I_{p,q}\mathcal{A}.
\edd
This also implies that $\mathcal{A}  I_{p,q}\mathcal{A}^\ast=I_{p,q}$. \\
\hf All the eigenvalues of $\mathcal{A}^\ast\mathcal{A}$ are positive. Moreover, if $\lambda$ is an eigenvalue of $\mathcal{A}^\ast\mathcal{A}$ with eigenvector $\mathbf{v}=(v_1,\ldots,v_p,v_{p+1},\ldots, v_n)$, then 
\bgd
I_{p,q}\mathbf{v}=\mathcal{A}^\ast\mathcal{A}\,I_{p,q}\mathcal{A}^\ast\mathcal{A}\mathbf{v}=\lambda(\mathcal{A}^\ast\mathcal{A}\,I_{p,q}\mathbf{v}),
\edd
which implies that $\lambda^{-1}$ is also an eigenvalue with eigenvector $\mathbf{v}'=(v_1,\ldots,v_p,-v_{p+1},\ldots, -v_n)$. If $\{\mathbf{v}_1, \ldots,\mathbf{v}_n\}$ is an eigenbasis of $\mathcal{A}^\ast\mathcal{A}$ with (possibly repeated) eigenvalues $\lambda_1,\ldots,\lambda_n$, then 
\bgd
\sqrt{\mathcal{A}^\ast\mathcal{A}}\,I_{p,q}\sqrt{\mathcal{A}^\ast\mathcal{A}}\mathbf{v}_j=\sqrt{\mathcal{A}^\ast\mathcal{A}}\,I_{p,q}\sqrt{\lambda_j}\mathbf{v}_j=\sqrt{\lambda_j}\sqrt{\mathcal{A}^\ast\mathcal{A}}\mathbf{v}_j'=\mathbf{v}_j'=I_{p,q}\mathbf{v}_j.
\edd
Thus, $\sqrt{\mathcal{A}^\ast\mathcal{A}}$ satisfies the defining relation for a matrix to be in $\upq$. \end{proof}
We may use the polar decomposition (for matrices in $GL(n,\C)$) to write
\bgd
\mathcal{A}=U |\mathcal{A}|,\,\,\textup{where}\,\,U=\mathcal{A}\left(\sqrt{\mathcal{A}^\ast\mathcal{A}}\right)^{-1}, |\mathcal{A}|=\sqrt{\mathcal{A}^\ast\mathcal{A}},
\edd
where $U,|\mathcal{A}|\in \upq$. For $\textup{U}(1,1)$ this decomposition takes the form
\bgd
\left(\begin{array}{cc}
\alpha & \beta\\
\lambda \overline{\beta} & \lambda \overline{\alpha}
\end{array}\right)=\left(\begin{array}{cc}
\frac{\alpha}{|\alpha|} & 0\\
0  & \lambda \frac{\overline{\alpha}}{|\alpha|}
\end{array}\right)\left(\begin{array}{cc}
|\alpha| & \frac{|\alpha|\beta}{\alpha}\\
\frac{|\alpha| \overline{\beta}}{\overline{\alpha}} & |\alpha|
\end{array}\right)
\edd
\hf The Lie algebra $\mathfrak{u}_{p,q}$ is given by matrices $X\in M_n(\C)$ such that
\bgd
X^\ast I_{p,q}+I_{p,q}X=0.
\edd
This is real Lie subalgebra of $M_{p+q}(\C)$. It contains the subalgebras $\mathfrak{u}_p, \mathfrak{u}_q$ as Lie algebras of the subgroups $U(p)\times I_q$ and $I_p \times U(q)$. Consider the inner product 
\bgd
\lan \cdot,\cdot\ran:\mathfrak{u}_{p,q}\times \mathfrak{u}_{p,q}\to\R,\,\,\,\lan X,Y\ran:=\textup{trace}(X^\ast Y).
\edd
\begin{lmm}
The inner product is symmetric and positive-definite. 
\end{lmm}
\begin{proof}
Note that 
\bgd
\lan X,Y\ran=\textup{trace}(-I_{p,q}XI_{p,q}Y)=\textup{trace}(-I_{p,q}YI_{p,q}X)=\lan Y,X\ran.
\edd
Since $\overline{\lan X,Y\ran}=\lan Y,X\ran$ due to the invariance of trace under transpose, we conclude that the inner product is real and symmetric. It is positive-definite as $\lan X,X\ran=\textup{trace}(X^\ast X)\geq 0$ and equality holds if and only if $X$ is the zero matrix. 
\end{proof}
The Riemannian metric obtained by left translations of $\lan\cdot,\cdot\ran$ will also be denoted by $\lan \cdot,\cdot\ran$. We shall analyze the geodesics for this metric. The Lie algebra $\mathfrak{u}_p\oplus \mathfrak{u}_q$ of $U(p)\times U(q)$ consists of 
\bgd
\left(\begin{array}{cc}
A & 0 \\
0 & D
\end{array}\right),\,\,A+A^\ast=0,\,D+D^\ast =0.
\edd
Let $\mathfrak{n}$ denote the orthogonal complement of $\mathfrak{u}_p\oplus \mathfrak{u}_q$ inside $\mathfrak{u}_{p,q}$. As $\mathfrak{n}$ is of (complex) dimension $pq$, and 
\bgd
\left\{\left(\begin{array}{cc}
0 & B\\
B^\ast & 0
\end{array}\right)\,\Big|\,B\in M_{p,q}(\C)\right\}
\edd
is contained in $\mathfrak{n}$, this is all of it. We may verify that
\begin{eqnarray*}
\left[ \left(\begin{array}{cc}
A & 0\\
0 & D
\end{array}\right),\left(\begin{array}{cc}
0 & B\\
B^\ast & 0
\end{array}\right)\right] & = & \left(\begin{array}{cc}
0 & AB-BD\\
DB^\ast-B^\ast A & 0
\end{array}\right)\in\mathfrak{n}\\
\left[ \left(\begin{array}{cc}
0 & B\\
B^\ast & 0
\end{array}\right),\left(\begin{array}{cc}
0 & C\\
C^\ast & 0
\end{array}\right)\right] & = & \left(\begin{array}{cc}
BC^\ast-CB^\ast & 0\\
0 A & B^\ast C-C^\ast B
\end{array}\right)\in\mathfrak{u}_p\oplus \mathfrak{u}_q.
\end{eqnarray*}
\begin{lmm}
Let $\gamma$ be the integral curve, initialized at $e$, for a left-invariant vector field $Y$. This curve is a geodesic if $Y(e)$ either belongs to $\mathfrak{n}$ or to $\mathfrak{u}_p\oplus\mathfrak{u}_q$.
\end{lmm}
\begin{proof}
The Levi-Civita connection $\nabla$ is given by the Koszul formula
\bgd
2\lan X,\nabla_Z Y\ran = Z\lan X,Y\ran+Y\lan X,Z\ran -X\lan Y,Z\ran +\lan Z,[X,Y]\ran+\lan Y,[X,Z]\ran - \lan X,[Y,Z]\ran. 
\edd
Putting $Z=Y$ and $X$, two left-invariant vector fields, in the above, we obtain
\bgd
\lan X, \nabla_Y Y\ran = \lan Y, [X,Y]\ran.
\edd
To prove out claim, it suffices to show that $\nabla_Y Y=0$, i.e., $\lan Y, [X,Y]\ran=0$ for any $X$. Let us assume that $Y(e)\in\mathfrak{n}$. If $X(e)\in\mathfrak{n}$, then $[X(e),Y(e)]\in \mathfrak{u}_p\oplus\mathfrak{u}_q$ which implies that $\lan Y(e), [X(e),Y(e)]\ran=0$. If $X(e)\in \mathfrak{u}_p\oplus\mathfrak{u}_q$, then
\begin{eqnarray*}
\lan Y, [X,Y]\ran & = &  \lan\left(\begin{array}{cc}
0 & B\\
B^\ast & 0
\end{array}\right), \left(\begin{array}{cc}
0 & AB-BD\\
DB^\ast-B^\ast A & 0
\end{array}\right)\ran\\
& = & \textup{trace}\left(\begin{array}{cc}
B(DB^\ast-B^\ast A) & 0\\
0 & B^\ast(AB-BD)
\end{array}\right)\\
& = & \textup{trace}(BDB^\ast-BB^\ast A)+\textup{trace}(B^\ast AB-B^\ast BD)\\
& = & 0
\end{eqnarray*}
by the cyclic property of trace. Thus, $\nabla_Y Y=0$ if $Y(e)\in\mathfrak{n}$; similar proof works if $Y(e)\in \mathfrak{u}_p\oplus\mathfrak{u}_q$. 
\end{proof}
\begin{rem}
An integral curve of a left-invariant vector field (also called $1$-parameter subgroups) need not be a geodesic in $\upq$. For instance, if $X+Y$ is a left-invariant vector field given by $X(e)\in\mathfrak{u}_p\oplus\mathfrak{u}_q$ and $Y(e)\in \mathfrak{n}$, then $\nabla_{X+Y}(X+Y)=0$ if and only if $\nabla_X Y=\frac{1}{2}[X,Y]$ and $\nabla_Y X=\frac{1}{2}[Y,X]$. This happens if and only if the metric is bi-invariant, i.e.,
\bgd
\lan [X,Z],Y\ran=\lan X, [Z,Y]\ran.
\edd 
This is not true; for instance, with $X(e)\in \mathfrak{u}_p\oplus\mathfrak{u}_q$ and linearly independent $Y(e),Z(e)\in \mathfrak{n}$, we get $\lan [X,Z],Y\ran-\lan X, [Z,Y]\ran\neq 0$.
\end{rem}
\hf Consider the matrix
\bgd
Y=\left(\begin{array}{cc}
0 & B\\
B^\ast & 0
\end{array}\right)\in\mathfrak{n}.
\edd
Let $B=U \sqrt{B^\ast B}$ and $B^\ast =\sqrt{B^\ast B}\,U^\ast$ be polar decompositions, where $U$ and $U^\ast$ are partial isometries. It follows from direct computation that
\begin{eqnarray*}
e^Y & = & \left(\begin{array}{cc}
I_p+\frac{BB^\ast}{2!} + \frac{(BB^\ast)^2}{4!}+\cdots & \frac{B}{1!}+\frac{B (B^\ast  B)}{3!} + \frac{B (B^\ast B)^2 }{5!}+\cdots \\
\frac{B^\ast}{1!}+\frac{(B^\ast B)B^\ast}{3!} + \frac{(B^\ast B)^2 B^\ast}{5!}+\cdots  & I_q+\frac{B^\ast B}{2!} + \frac{(B^\ast B)^2}{4!}+\cdots 
\end{array}\right)\\
& = & \left(\begin{array}{cc}
\cosh(\sqrt{BB^\ast}) & U\sinh(\sqrt{B^\ast B})\\
\sinh(\sqrt{B^\ast B})U^\ast & \cosh(\sqrt{B^\ast B})
\end{array}\right).
\end{eqnarray*}
It can be checked that 
\bgd
e^\mathfrak{n}\cap \left(U(p)\times U(q)\right)=\{I_n\}.
\edd
It is known that the non-zero eigenvalues of $Y$ are the non-zero eigenvalues of $\sqrt{BB^\ast}$ and their negatives. 
\begin{thm}\label{mainthm}
For any element $\mathcal{A}\in \upq$, the associated matrix $\sqrt{\mathcal{A}^\ast\mathcal{A}}$ can be expressed uniquely as $e^Y$ for $Y\in \mathfrak{n}$. Moreover, there is a unique way to express $\mathcal{A}$ as a product of a unitary matrix and an element of $e^\mathfrak{n}$, and it is given by the polar decomposition.
\end{thm}
In order to prove the result, we discuss some preliminaries on logarithm of complex matrices. In general, there is no unique logarithm. However, the Gregory series
\bgd
\log A=-\sum_{m=0}^{\infty}\frac{2}{2m+1}\left[(I-A)(I+A)^{-1}\right]^{2m+1}
\edd
converges if all the eigenvalues of $A\in M_n(\C)$ have positive real part \cite[\S 11.3, p. 273]{Hig08}. In particular, $\log A$ is well-defined for Hermitian positive-definite matrix. This is often called the \textit{principal logarithm} of $A$. This logarithm satisfies $e^{\log A}=A$. There is an integral form of logarithm that applies to matrices withour real or zero eigenvalues; it is given by
\bgd
\log A =(A-I)\int_0^1 \left[s(A-I)+I\right]^{-1}ds.
\edd
\begin{lmm}\label{inv}
The inverse of $\mathcal{A}^\ast\mathcal{A}+I_n$ for $\mathcal{A}\in \upq$ is given by
\bgd
\left[\mathcal{A}^\ast\mathcal{A}+I_n\right]^{-1}=\frac{1}{2}\left(\begin{array}{cc}
I_p & -A^{-1}B\\
-B^\ast (A^\ast)^{-1} & I_q
\end{array}\right).
\edd
\end{lmm}
\begin{proof}
Since $\mathcal{A}^\ast\mathcal{A}$ has only positive eigenvalues, $\mathcal{A}^\ast\mathcal{A}+I_n$ has no kernel. We note that 
\bgd
\mathcal{A}^\ast\mathcal{A}+I_n=\left(\begin{array}{cc}
2C^\ast C+2I_p & 2A^\ast B\\
2B^\ast A & 2B^\ast B+2I_q
\end{array}\right)=\left(\begin{array}{cc}
2A^\ast A & 2A^\ast B\\
2B^\ast A & 2D^\ast D
\end{array}\right).
\edd
The inverse matrix satisfies 
\bgd
\left(\begin{array}{cc}
2A^\ast A & 2A^\ast B\\
2B^\ast A & 2D^\ast D
\end{array}\right)\left(\begin{array}{cc}
E & F\\
F^\ast & G
\end{array}\right)=\left(\begin{array}{cc}
I_p & 0\\
0 & I_q
\end{array}\right).
\edd 
As the matrices are Hermitian, the three constraints that $E,F,G$ must satisfy (and are uniquely determined by) are
\begin{eqnarray*}
E & = & \textstyle{\frac{1}{2}}(A^\ast A)^{-1}-A^{-1}BF^\ast\\
G & = & \textstyle{\frac{1}{2}}(D^\ast D)^{-1}-D^{-1}CF\\
F & = & -A^{-1}BG.
\end{eqnarray*}
We note that $E=\frac{1}{2}I_p$, $G=\frac{1}{2}I_q$ and $F=-\frac{1}{2}A^{-1}B$ satisfy the above equations. For instance, 
\bgd
\textstyle{\frac{1}{2}}(A^\ast A)^{-1}-A^{-1}BF^\ast=\textstyle{\frac{1}{2}}(A^\ast A)^{-1}+\frac{1}{2}A^{-1}BB^\ast (A^{\ast})^{-1}=\textstyle{\frac{1}{2}}(A^\ast A)^{-1}+\frac{1}{2}A^{-1}(AA^\ast-I_p)(A^{\ast})^{-1}=\frac{1}{2}I_p,
\edd
where $BB^\ast=AA^\ast-I_p$ is a consequence of $\mathcal{A}^\ast\in \upq$. Yet another consequence is $AC^\ast=BD^\ast$, which is equivalent to
\bgd
A^{-1}B=(D^{-1}C)^\ast.
\edd
In a similar vein,
\bgd
\textstyle{\frac{1}{2}}(D^\ast D)^{-1}-D^{-1}CF=\textstyle{\frac{1}{2}}(D^\ast D)^{-1}+\frac{1}{2}D^{-1}CC^\ast (D^{\ast})^{-1}=\textstyle{\frac{1}{2}}(D^\ast D)^{-1}+\frac{1}{2}D^{-1}(DD^\ast-I_q)(D^{\ast})^{-1}=\frac{1}{2}I_q,
\edd
where $CC^\ast=DD^\ast-I_q$ is due to $\mathcal{A}^\ast\in \upq$.
\end{proof}
\begin{proof}[Proof of Theorem \ref{mainthm}]
We use Gregory series expansion for computing the principal logarithm of $\mathcal{A}^\ast\mathcal{A}$ along with Lemma \ref{inv}:
\begin{eqnarray*}
\log (\mathcal{A}^\ast\mathcal{A}) & = & \sum_{m=0}^\infty\textstyle{\frac{2}{2m+1}}\left[2\left(\begin{array}{cc}
A^\ast A-I_p & A^\ast B\\
B^\ast A & D^\ast D-I_q
\end{array}\right)
\frac{1}{2}\left(\begin{array}{cc}
I_p & -A^{-1}B\\
-B^\ast (A^\ast)^{-1} & I_q
\end{array}\right)\right]^{2m+1}\\
& = & \sum_{m=0}^\infty\textstyle{\frac{2}{2m+1}}\left(\begin{array}{cc}
0 & A^{-1}B\\
B^\ast (A^\ast)^{-1} & 0
 \end{array}\right)^{2m+1}.
\end{eqnarray*}
We set $Y=\frac{1}{2}\log (\mathcal{A}^\ast\mathcal{A})$. It is clear that $Y\in\mathfrak{n}$ and $e^Y=\sqrt{\mathcal{A}^\ast\mathcal{A}}$. It is known that the exponential map is injective on Hermitian matrices. This implies the uniqueness of $Y$. \\
\hf If $U_1e^{Y_1}=U_2 e^{Y_2}$ are two decompositions of $\mathcal{A}\in\upq$ with $U_i\in \textup{U}(p)\times\textup{U}(q)$ and $Y_i\in\mathfrak{n}$, then 
\bgd
e^{2Y_1}=e^{Y_1}U_1^\ast U_1 e^{Y_1}=e^{Y_2}U_2^\ast U_2 e^{Y_2}=e^{2Y_2}.
\edd
By the injectivity of the exponential map (on Hermitian matrices), we obtain $Y_1=Y_2$, which implies that $U_1=U_2$.
\end{proof}
We infer the following (see \cite{YaSt75} Lemma 1, p. 211 for a different proof) result.
\begin{cor}\label{expsurj}
The exponential map $\textup{exp}:\mathfrak{u}_{p,q}\to U(p,q)$ is surjective.
\end{cor}
\begin{proof}
Using the polar decomposition and Theorem \ref{mainthm}, 
\bgd
\mathcal{A}=\mathcal{A}\big(\sqrt{\mathcal{A}^\ast\mathcal{A}}\big)^{-1}\sqrt{\mathcal{A}^\ast\mathcal{A}}=\mathcal{A}\big(\sqrt{\mathcal{A}^\ast\mathcal{A}}\big)^{-1}e^Y.
\edd
Since the matrix exponential is surjective for $U(p)\times U(q)$, choose $Z\in \mathfrak{u}_p\oplus\mathfrak{u}_q$ such that $e^Z=\mathcal{A}(\sqrt{\mathcal{A}^\ast\mathcal{A}})^{-1}$. By Baker-Campbell-Hausdorff formula, we may express $e^Z e^Y$ as exponential of an element in $\mathfrak{u}_{p,q}$. 
\end{proof}
\hf The distance from any matrix $\mathcal{A}\in \upq$ to $U(p)\times U(q)$ is given by the length of the curve
\bgd
\gamma(t)=\mathcal{A}\big(\sqrt{\mathcal{A}^\ast\mathcal{A}}\big)^{-1}e^{tY},
\edd
which can be computed (and simplified via left-invariance) as follows
\bgd
\ell(\gamma)=\int_0^1 \|\gamma'(t)\|_{\gamma(t)}\,dt=\int_0^1 \|Y\|\,dt=\|Y\|.
\edd
Note that 
\bgd
\|Y\|^2=\textup{trace}(Y^\ast Y)=\textup{trace}\left[\textstyle{\frac{1}{4}}(\log (\mathcal{A}^\ast\mathcal{A}))^2\right].
\edd
Thus, the distance squared function is given by 
\bgd
d^2:\upq\to\R,\,\,\mathcal{A}\mapsto \textstyle{\frac{1}{4}}\textup{trace}\left[\left(\log (\mathcal{A}^\ast\mathcal{A})\right)^2\right].
\edd


\appendix

\section{The continuity of the map $\mathpzc{s}$} \label{Sec: cts-s}

Recall the statement of Proposition \ref{snucts}.
\begin{prpn}
The map $\mathpzc{s}:S(\nu)\to [0,\infty)$, as defined in \eqref{snu}, is continuous.
\end{prpn}
\noindent The proof relies on a characterization of $\mathpzc{s}(v)$.
\begin{lmm}\label{chsnu}
Let $u\in S_p(\nu)$. A positive real number $T$ is $\mathpzc{s}(u)$ if and only if $\gamma_u:[0,T]$ is an $N$-geodesic and at least one of the following holds:\\
	\hf (i) $\gamma_u(T)$ is the first focal point of $N$ along $\gamma_u$,\\
	\hf (ii) there exists $v\in S(\nu)$ with $v\neq u$ such that $\gamma_v(T)=\gamma_u(T)$.
\end{lmm}
\noindent Note that $\gamma_u(T)$ being a focal point of $N$ along $\gamma_u$ means that $(D\exp_\nu)(uT)$ is not of full rank, where $\exp_\nu$ is the normal exponential as defined in \eqref{norexp}. When $N$ is a point, then this notion of focal points reduces to that of conjugate points.\\
\noindent In order to prove the Lemma, we need the following observations.\\[0.2cm]
\textbf{Observation A}\,\textup{\cite[Lemma 2.11, p. 96]{Sak96}}\,\,\textit{Let $N$ be a submanifold of $M$ and $\gamma:[a,\infty)\to M$ a geodesic emanating perpendicularly from $N$. If $\gamma(b)$ is the first focal point of $N$ along $\gamma$, then for $t>b$, $\gamma|_{[a,t]}$ cannot be an $N$-geodesic, i.e., $L\paran{\gamma|_{[a,t]}}>d(N,\gamma(t))$.}\\[0.2cm]
\noindent Recall that a sequence $\{\gamma_n\}$ of geodesics, defined on closed intervals, is said to converge to a geodesic $\gamma$ if $\gamma_{n}(0)\to \gamma(0)$ and $\gamma_n'(0)\to \gamma'(0)$. It follows from the continuity of the exponential map that if $t_n\to t$, then $\gamma_n(t_n)\to \gamma(t)$.\\[0.2cm]
\textbf{Observation B}\,\,\textit{Let $\gamma_n$ be unit speed $N$-geodesics joining $p_n=\gamma_n(0)$ to $q_n=\gamma_n(t_n)$. If $\gamma_n$ converges to a geodesic $\gamma$ and $t_n\to l$, then $\gamma$ is a unit speed $N$-geodesic joining $p=\lim_n p_n$ to $q \defeq \gamma(l)=\lim_n \gamma_n(t_n)$.}
\begin{proof}
The unit normal bundle $S(\nu)$ is closed. Since $\gamma_n'(0)\to \gamma'(0)$, it follows that $\gamma'(0)\in S(\nu)$. Note that
	\begin{displaymath}
	d(N,q) = \lim_{n\to \infty} d(N,q_n) = \lim_{n\to \infty} d(p_n,q_n)= \lim_{n\to \infty} t_n = l = L\paran{\gamma|_{[0,l]}}
	\end{displaymath}
implies that $\gamma$ is a $N$-geodesic.
\end{proof}
\begin{proof}[Proof of the Lemma A.2]
If $\gamma_u(t)$ is the first focal point of $N$ along $\gamma_u$, then Observation A implies that $\gamma_u$ cannot be minimal beyond this value. If (ii) holds, then we need to show that for sufficiently small $\varepsilon>0~$ $\gamma_u|_{[0,T+\varepsilon]}$ is not minimal. Suppose, on the contrary, that $\gamma_u$ is minimal beyond $T$. Take a minimal geodesic $\beta$ joining $\gamma_{v}(T-\varepsilon)$ to $\gamma_{u}(T+\varepsilon)$. 
Observe that,
	\begin{align*}
	2\varepsilon & = d\paran{\gamma_{u}(T+\varepsilon),\gamma_u(T)} + d\paran{\gamma_v(T),\gamma_v(T-\varepsilon)}> d\paran{\gamma_u(T+\varepsilon),\gamma_v(T-\varepsilon)}.
	\end{align*}
If $p,q,r\in M$ such that $d(p,q)+d(q,r) = d(p,r)$ and there exist shortest normal geodesic $\gamma_{1}$ and $\gamma_{2}$ joining $p$ to $q$ and $q$ to $r$, respectively, then $\gamma_{1}\cup \gamma_2$ is smooth at $q$ and defines a shortest normal geodesic joining $p$ to $r$. Therefore, we have
	\begin{displaymath}
			L(\gamma_v|_{[0,T-\varepsilon]}\cup \beta) = T-\varepsilon + d(\gamma_v(T-\varepsilon),\gamma_u(T+\varepsilon)) < T + \varepsilon = L(\gamma_u|_{[0,T+\varepsilon]}).
	\end{displaymath}
This contradiction establishes that $\gamma_u|_{[0,T+\varepsilon]}$ is not minimal.\\
\noindent \hf For the converse, set $T=\mathpzc{s}(u)$ and observe that $\gamma_u|_{[0,T]}$ is an $N$-geodesic. Assuming that $q\defeq \gamma_{u}(T)$ is not the first focal point of $N$ along $\gamma_{u}$, we will prove that (ii) holds. Let $p=\gamma_u(0)$ and choose a neighbourhood $\tilde{U}$ of $Tu$ in $\nu$ such that $\exp_{\nu}|_{\tilde{U}}$ is a diffeomorphism. For sufficiently large $n$, $q_{n}\defeq \gamma_{u}\paran{T+1/n}\in \exp_{\nu}(\tilde{U})$. Take $N$-geodesics $\gamma_{n}$ parametrized by arc-length joining $p_n$ to $q_{n}$ and set $u_{n}\defeq \dot{\gamma}_{n}(0)\in S((T_{p_n}N)^\perp)$. Since $S((T_{p_n}N)^\perp)$ is compact, by passing to a subsequence, we may assume that $u_{n}$ converges to $v\in S(N_{p})$. By Observation B, 
	\begin{displaymath}
	\gamma_v(T) = \lim_{n\to \infty} \gamma_{u_n}\paran{T+\textstyle{\frac{1}{n}}} = \gamma_u(T).
	\end{displaymath}
If $v = u$, then for sufficiently large $n$, $d(p,q_n)u_{n} \in \tilde{U}$, whence 
	\begin{displaymath}
	\paran{T+\textstyle{\frac{1}{n}}}u = d(p,q_n)u_n.
	\end{displaymath}
Taking absolute values on both sides imply $T+1/n>d(p,q_n)$. This contradiction implies $v\neq u$.
\end{proof}
\begin{proof}[Proof of the Proposition A.1]
We will prove that $\mathpzc{s}(u_n)\to \mathpzc{s}(u)$ whenever $(p_n,u_n)\to (p,u)$ in the unit normal bundle $S(\nu)$. Let $T$ be any accumulation point of the sequence $\{\mathpzc{s}(u_n)\}$ including $\infty$. By Observation B, $\gamma_{u}|_{[0,T]}$ is an $N$-geodesic and hence $T\le \mathpzc{s}(u)$. If $T=+\infty$ we are done. So let us assume that $T<+\infty$. From Lemma A.2, at least one of the following holds for infinitely many $n$:\\
\hf (i) $\mathpzc{s}(u_n)$ is the first focal point to $N$ along $\gamma_{u_n}$\\
\hf (ii) there exist $v_{n}\in S(N_{p_n})$,  $v_{n}\neq u_n$ with $\gamma_{u_n}\paran{\mathpzc{s}(u_n)} = \gamma_{v_n}\paran{\mathpzc{s}(u_n)}$.\\
\hf If (i) is true for infinitely many $n$, then choose infinitely many unit vectors $\{w_n\}$ which belong to the kernel $\ker \paran{D\exp_{\nu}(\mathpzc{s}(u_n)u_n)}$ and are contained in a compact subset of $S(\nu)$. Choose a convergent subsequence whose limit $w$ is contained in $\ker \paran{D\exp_{\nu}{(Tu)}}$. Since $w\neq 0$, the rank of $D\exp_{\nu}{(Tu)}$ is less than $\dim M$. Thus, $\gamma_{u}(T)$ is the first focal point of $N$ along $\gamma_u$ and $T=\mathpzc{s}(u)$.\\
\hf If (ii) is true for infinitely many $n$, then we may assume that $v_{n}\to v\in S(\nu)$. If $v\neq u$, then Lemma A.2 (ii) holds for $T$, whence $T=\mathpzc{s}(u)$. If $v=u$, we claim that $\gamma_{u}(T)$ is the first focal point of $N$ along $\gamma_{u}$. If not, then the map $\exp_{\nu}$ is regular at $Tu\in \nu$ and hence the map 
\bgd
\Phi:\nu\to M\times M,~(p,u)\mapsto (p,\exp_\nu(p,u))
\edd
is regular at $Tu$. Therefore, $\Phi$ is a diffeomorphism if restricted to an open neighbourhood $\tilde{U}$ of $Tu$ in $\nu$. Since $v=u$, which implies for sufficiently large $n$,~ $(p_n,\mathpzc{s}(u_n)u_n)$ and $(p_n,\mathpzc{s}(u_n)v_n)$ belong to $\tilde{U}$ and are different. On the other hand, by assumption $\Phi(\mathpzc{s}(u_n)u_n) = \Phi(\mathpzc{s}(u_n)v_n) $, which is a contradiction. Therefore, $\gamma_u(T)$ is the first focal point and $T=\mathpzc{s}(u)$. 
\end{proof}

\section{Derivative of the square root map}\label{Sec: der-app}

\begin{lmm}\label{lemma: A.1}
Let $A$ be a positive definite matrix and $\psi:A\mapsto \sqrt{A}$. Then 
    \begin{equation*}\label{eq: Appendix-sqrtderivative}
        D\psi_A(H) = \int_0^\infty e^{-t\sqrt{A}}He^{-t\sqrt{A}}~dt,
    \end{equation*}
for any symmetric matrix $H$.
\end{lmm}
\begin{proof}
As $\psi(A)\cdot \psi(A) = A$, differentiating at $A$ we obtain
    \begin{equation}\label{eq: sylvester}
         D\psi_A(H)\psi(A) + \psi(A)D\psi_A(H) = H.
    \end{equation}
We will show the following:
    \begin{enumerate}
        \item [(i)] For any positive definite matrix $X$ and for any symmetric matrix $Y$ the integral 
        \begin{equation}\label{eq: matsqrtderivative}
            \int_0^\infty e^{-tX}Ye^{-tX}~\mathrm{d}t
        \end{equation}
converges. Since $X$ is positive definite, we can diagonalize it. Write $X = UDU^T$, where $U$ is an orthogonal matrix and $D$ is a diagonal matrix with the (positive) eigenvalues of $X$ in the diagonal. If $\lambda_1,\lambda_2,\cdots,\lambda_n$ are the (possibly repeated) eigenvalues of $X$, then
        \begin{align*}
            e^{-tX} & = Ue^{-tD}U 
                    \\
                    & = U \begin{pmatrix}
                        e^{-t\lambda_1} & & & \\
                                        & e^{-t\lambda_2} & & \\ 
                                        & & \ddots &  \\
                                        & & & e^{-t\lambda_n}
                    \end{pmatrix} U^T.
        \end{align*}
        Now for any $n\in \bb{N}$, 
        \begin{align*}
            \int_0^ne^{-tX}Ye^{-tX}~\mathrm{d}t & = \int_0^n \paran{Ue^{-tD}U^T} Y  \paran{Ue^{-tD}U^T}~\mathrm{d}t
            \\
            & = U\int_0^ne^{-tD}U^TYUe^{-tD}U^T~\mathrm{d}t
            \\
            & = (UU^T)Y U\int_0^ne^{-tD}e^{-tD}U^T~\mathrm{d}t
            \\
            & = YU\bigg(\int_0^ne^{-2tD}~\mathrm{d}t\bigg)U^T.
        \end{align*}
Since $D$ is a diagonal matrix with positive eigenvalues, the above integral converges and hence the integral (\ref{eq: matsqrtderivative}) converges.
        \item [(ii)] $D\psi_A(H)$ satisfies \eqref{eq: sylvester}. Observe that 
        \begin{align*}
              & \paran{\int_0^\infty e^{-t\sqrt{A}}\cdot H\cdot e^{-t\sqrt{A}} \mathrm{d}t } \sqrt{A} + \sqrt{A} \paran{\int_0^\infty e^{-t\sqrt{A}}\cdot H\cdot e^{-t\sqrt{A}}~\mathrm{d}t} 
            \\
            = & \int_0^\infty\paran{e^{-t\sqrt{A}}\cdot H\cdot e^{-t\sqrt{A}}  \sqrt{A} + \sqrt{A} e^{-t\sqrt{A}}\cdot H\cdot e^{-t\sqrt{A}}}~\mathrm{d}t
            \\
            = & \int_0^\infty\paran{e^{-t\sqrt{A}} H e^{-t\sqrt{A}}}'~\mathrm{d}t = H
        \end{align*}
    \end{enumerate}
    From (i), (ii) and the uniqueness of the derivative, the lemma is proved. 
\end{proof}

\begin{lmm}\label{lemma: A.2}
The map $g : M(n,\rbb)\to \rbb, ~A\mapsto \trace{\sqrt{A^TA}}$ is differentiable if and only if $A$ is invertible. 
\end{lmm}
\begin{proof}
Let $A$ be an invertible matrix. We will prove that the function $g$ is differentiable at $A$. Let $\cali{P}$  be the set of all positive definite matrices which is an open subset of the set of all symmetric matrices $\cali{S}$. We will prove that the map 
	\begin{displaymath}
		r : \cali{P}\to \cali{P}, ~ A\mapsto \sqrt{A}
	\end{displaymath}
	is differentiable. Define a function 
	\begin{displaymath}
		s : \cali{P}\to \cali{P},~ A\mapsto A^2.
	\end{displaymath}
	We will show that $s$ is a diffeomorphism and from the inverse function theorem $r$ will be differentiable. In order to show that $s$ is a diffeomorphism, we claim that for $A\in \cali{P},~Ds_{A}:T_{A}\cali{P}\to T_{A^2}\cali{P}$ is injective. Note that $\cali{P}$ is an open subset of a vector space $\cali{S}$ and therefore, $T_{A}\cali{P}\isom \cali{S} \isom T_{A^2}\cali{P}.$ So, take $B\in \cali{S}$ such that $Ds_{A}(B) = 0$. We will show that $B = 0.$ Recall that $Ds_{A}(B) = AB +BA.$ Now choose an orthonormal basis $\{\vbf_1,\vbf_2,\cdots,\vbf_n\}$ of eigenspace of $A$ and $A\vbf_{i} = \lambda_{i} \vbf_i$ ($\lambda_i>0$). Then,
	\begin{displaymath}
	 	A(B\vbf_i) = -BA\vbf_i = -B\lambda_i\vbf_i = -\lambda i(B\vbf_i)
	 \end{displaymath} 
	 which implies $Bv_{i}$ is also an eigenvector of $A$ with eigenvalue $-\lambda_{i}<0$. Hence, $Bv_{i} = 0$ which implies $B=0$.\\
\hf For the converse, we will show that if $A$ is a singular matrix, then the map $g$ is not directional differentiable. Let $A$ be a singular matrix. Using the singular value decomposition, we write 
     \begin{displaymath}
         A = U \begin{pmatrix}
             D & 0 \\ 
             0   & 0_k
         \end{pmatrix} V^T,    
     \end{displaymath}
where $D$ is a $(n-k)\times (n-k)$ diagonal matrix with positive entries. If 
     \begin{displaymath}
         B = U \begin{pmatrix}
             0_{n-k} & 0 \\
             0 & I_k
         \end{pmatrix}    
     \end{displaymath}
then we claim that $g$ is not differentiable in the direction of $B$. Since 
     \begin{displaymath}
         \sqrt{(A+tB)^T(A+tB)} = V \begin{pmatrix}
             D & 0 \\
             0 & I_k|t|
         \end{pmatrix} V^T
     \end{displaymath}
the limit 
     \begin{align*}
         \lim_{t\to 0} \dfrac{g(A+tB) - g(A)}{t} & = \lim_{t\to 0} \dfrac{\trace{V\begin{pmatrix}
             D & 0 \\
             0 & I_k|t|
         \end{pmatrix} V^T} - \trace{V \begin{pmatrix}
             D & 0 \\
             0 & 0_k
         \end{pmatrix} V^T}}{t}\\
         & = k \lim_{t\to 0} \dfrac{|t|}{t} 
     \end{align*}
does not exist and hence the function $g$ is not differentiable.
\end{proof}


\end{document}